\pdfoutput=1
\documentclass[11pt]{amsart}
\usepackage[style=alphabetic]{biblatex}
\usepackage{amsmath}

\usepackage{tikz-cd}
\usepackage{amsfonts,amssymb}
\usepackage{mathrsfs}
\usepackage{color}
\usepackage{quiver}
\usepackage[T1]{fontenc}
\usepackage{lmodern}
\definecolor{darkred}{rgb}{0.75,0,0}
\definecolor{darkblue}{rgb}{0,0,0.6} 

\providecommand{\pr}{\text{pr}}

\providecommand{\et}{\acute{e}t}
\newcommand{\mono}{\hookrightarrow}

\providecommand{\Fun}{\text{Fun}}
\providecommand{\End}{\text{End}}

\providecommand{\ob}{\text{ob}}

\providecommand{\Hom}{\text{Hom}}

\providecommand{\id}{\text{id}}
\providecommand{\ker}{\text{ker}}
\providecommand{\lim}{\text{lim}}
\providecommand{\Mod}{\text{Mod}}

\providecommand{\pr}{\text{pr}}
\providecommand{\Spec}{\mathrm{Spec}}
\DeclareMathOperator{\QCoh}{{QCoh}}

\DeclareMathOperator{\Fil}{{Fil}}
\DeclareMathOperator{\conj}{{conj}}

\newcommand{\dSt}{\operatorname{dSt}}
\newcommand{\dR}{\operatorname{dR}}

\newcommand{\dSch}{\operatorname{dSch}}
\newcommand{\Sch}{\operatorname{Sch}}
\newcommand{\FSch}{\operatorname{FSch}}
\newcommand{\Der}{\operatorname{Der}}

\newcommand{\Spc}{\operatorname{Spc}}

\newcommand{\cofib}{\operatorname{cofib}}
\newcommand{\Crys}{\operatorname{Crys}}
\newcommand{\fib}{\operatorname{fib}}

\DeclareMathOperator{\Vect}{{Vect}}
\DeclareMathOperator{\Cone}{{Cone}}
\DeclareMathOperator{\Ker}{{Ker}}
\DeclareMathOperator{\Perf}{{Perf}}

\newcommand{\bA}{{\mathbf A}}

\newcommand{\bE}{{\mathbb E}}
\newcommand{\bF}{{\mathbf F}}
\newcommand{\bG}{{\mathbf G}}

\newcommand{\bN}{{\mathbb N}}

\newcommand{\bQ}{{\mathbb Q}}

\newcommand{\bT}{{\mathbf T}}

\newcommand{\bV}{{\mathbf V}}

\newcommand{\bZ}{{\mathbf Z}}

\newcommand{\cA}{{\mathcal A}}
\newcommand{\cB}{{\mathcal B}}
\newcommand{\cC}{{\mathcal C}}
\newcommand{\cD}{{\mathcal D}}
\newcommand{\cE}{{\mathcal E}}

\newcommand{\cH}{{\mathcal H}}
\newcommand{\cI}{{\mathcal I}}
\newcommand{\cJ}{{\mathcal J}}

\newcommand{\cL}{{\mathcal L}}
\newcommand{\cM}{{\mathcal M}}

\newcommand{\cO}{{\mathcal O}}

\newcommand{\cR}{{\mathcal R}}

\newcommand{\cX}{{\mathcal X}}
\newcommand{\cY}{{\mathcal Y}}

\DeclareMathOperator{\gr}{{gr}}

 \DeclareMathOperator{\Spf}{{Spf}}
\DeclareMathOperator{\Sym}{{Sym}}

\DeclareMathOperator{\MIC}{{MIC}}
\DeclareMathOperator{\HIG}{{HIG}}

\newcommand{\Map}{\operatorname{Map}}
\newcommand{\CAlg}{\operatorname{CAlg}}

\newcommand{\Cech}{\operatorname{Cech}}
\newcommand{\Hig}{\operatorname{Hig}}
\newcommand{\DAlg}{\operatorname{DAlg}}
\newcommand{\an}{\operatorname{an}}

\usepackage{mathabx}
\usepackage{xcolor}

\usepackage{amsthm} 
\usepackage{thmtools}

\usepackage[colorlinks=true,citecolor=darkred,linkcolor=darkred,urlcolor=darkblue]{hyperref} 

\usepackage[nameinlink]{cleveref}

\theoremstyle{definition}
\newtheorem{theorem}{Theorem}[section]

\newtheorem{corollary}[theorem]{Corollary}
\newtheorem{question}[theorem]{Question}
\newtheorem{construction}[theorem]{Construction}

\newtheorem{definition}[theorem]{Definition}

\newtheorem{lemma}[theorem]{Lemma}

\newtheorem{remark}[theorem]{Remark}

\usepackage[margin=1.25in]{geometry}
\usepackage{tikz-cd}
\usetikzlibrary{arrows.meta}

\title{Ogus-Vologodsky equivalence via stacks}
\author[Terentiuk]{Gleb Terentiuk}

\begin{document}
\maketitle

\setcounter{tocdepth}{1}
\begin{abstract}
    Using the relative de Rham stack for a family $X \to S$ in characteristic $p,$ we reprove the (local and global) Ogus-Vologodsky equivalence. Moreover, we observe that a lift of $S$ is not necessary. Instead, we use a lift of $X$ to the second Witt vectors of $S.$ The main ingredient is that, for a quasi-syntomic family $X/S,$ the relative de Rham stack admits a structure of a torsor over $X'$ which is the analogue of the Azumaya property of the algebra of differential operators. This can be applied to families of (reasonable) algebraic stacks, which gives rise to a logarithmic version of the Cartier equivalence. Along the way, we also obtain a decompleted version of the global Cartier equivalence.
\end{abstract}
\tableofcontents

\section{Introduction}
For a compact K\"ahler manifold $X$, the non-abelian Hodge correspondence provides an equivalence between suitable categories of $D$-modules and Higgs modules. Despite being highly-transcendental, it admits a characteristic $p$ counterpart established by Ogus and Vologodsky in their celebrated paper \cite{MR2373230}. To formulate their results, let us introduce some notations. Let $X\to S$ be a smooth family of schemes over $\bF_p.$ One of the main tools in analyzing modules with integrable connection in characteristic $p$ is the Azumaya property of the algebra of differential operators. Namely, there exists a map $\psi: ST_{X'/S} \to F_{X/S, *}D_{X/S}$ exhibiting the source as the center of the target. It is defined by $\psi(v)=\partial_{v}^{p} - \partial_{v^{[p]}} $ where $v^{[p]}$ is a vector field obtained by noticing that the $p$-fold composition of a derivation is again a derivation.
\begin{itemize}
    \item[(1)] Denote by $\MIC(X/S)$ the category of $F_{X/S, *}D_{X/S}$-modules. Then $\MIC_{\le n}(X/S)$ is a full subcategory such that the action of $ST_{X'/S}$ factors through $S^{\le n}T_{X'/S}.$ Also, denote $\MIC^{\cdot}(X/S)$ to be the full subcategory of nilpotent $D$-modules, i.e., $(M, \nabla)$ such that every $m \in M$ is killed by $S^{\ge N}T_{X'/S}$ for some $N$.
    \item[(2)] Denote $\HIG(X'/S) \simeq \QCoh(\bT^*_{X'/S})$ to be the category of Higgs modules on $X'/S$. Then $\HIG_{\le n}(X'/S)$ will stand for the full subcategory of those Higgs modules where the action of $ST_{X'/S}$ factors through $S^{\le n}T_{X'/S}$, i.e., Higgs modules supported on $n$-th neighbourhood of the zero section. Similarly, $\HIG^{\cdot}(X'/S) = \QCoh(\widehat{\bT_{X'/S}^*})$ is the category of nilpotent Higgs modules.
\end{itemize}
One of the main results they prove, which is referred to as the \textit{local Cartier equivalence}, is the following.
\begin{theorem}(\cite[Theorem 2.11]{MR2373230})
   Let $X \to S$ be a smooth morphism. Any lift  of the relative Frobenius $X \to X'$ to a flat lift $\tilde{S} \to \bZ/p^2$ induces a symmetric monoidal equivalence $\HIG^{\cdot}(X'/S) \simeq \MIC^{\cdot}(X/S).$
\end{theorem}
Even when $S=\Spec(\bF_p)$, having a lift of the relative Frobenius is extremely restrictive. To formulate a more flexible result that they prove, let us also denote $$\HIG^{\cdot}_{\gamma}(X'/S) := \QCoh(\widehat{\bT^{*, \gamma}_{X'/S})}$$ and $\MIC^{\cdot}_{\gamma}(X/S)$ to be the category of locally nilpotent modules over $$F_{X/S, *}D_{X/S} \otimes_{S T_{X'/S}} \hat \Gamma T_{X'/S}.$$ We will refer to the following as \textit{the global Cartier equivalence}.
\begin{theorem}(\cite[Theorem 2.8]{MR2373230})\label{global-of-intro}
    Let $X \to S$ be a smooth morphism. Any flat lift of $X' \to S$ to $\bZ/p^2$ induces a symmetric monoidal equivalence $\MIC^{\cdot}_{\gamma}(X/S) \simeq \HIG^{\cdot}_{\gamma}(X'/S).$ In particular, one obtains an equivalence $\MIC_{\le p-1}(X/S) \simeq \HIG_{\le p-1}(X'/S).$
\end{theorem}
Their proof is based on exploiting the Azumaya property of the algebra of differential operators $F_{X/S, *} D_{X/S}$. That is, they use the given data to produce a splitting module for $F_{X/S, *}D_{X/S}$ over an appropriate neighbourhood of the zero section in $\bT^*_{X'/S}.$ Namely, they observe that a lift $\tilde{X'}/\tilde{S}$ gives rise to a $F_{X/S}^* T_{X'/S}$-torsor $\cL_{\cX}$ on $X$ with a flat connection\footnote{Note $F_{X/S}^*T_{X'/S}$ has a canonical connection with the vanishing $p$-curvature.} and, moreover, the dual of the ring of functions is a splitting module for $F_{X/S, *}D_{X/S}$ over $\widehat{\bT}^{*,\gamma}_{X'/S}$, the completed PD-envelope of the zero section in $\bT^*_{X'/S}.$ Any lift of $F_{X/S}$ gives a section $X \to \cL_{\cX}$ and they observe that its PD-envelope is preserved by the connection and (the dual of) the functions on this PD-envelope define a splitting module for $F_{X/S, *}D_{X/S}$ over $\widehat{\bT}_{X'/S}.$
\newline
Our approach can be seen as Koszul-dual. Instead of studying the algebra of differential operators $F_{X/S, *}D_{X/S}$, we focus on studying the de Rham complex $F_{X/S, *}\dR_{X/S} \in \cD_{qc}(X').$ One advantage of this approach is that the de Rham complex gives rise to a geometric object, called the \textit{de Rham stack}, and denoted by $(X/S)^{dR}.$ One of the key structures possessed by this stack is a map $\nu_{X/S}: (X/S)^{dR} \to X'$ which, when $X/S$ is smooth, exhibits the source as a $T_{X'/S}^{\sharp}$-gerbe over the target and which can be viewed as an avatar of the Azumaya property of $F_{X/S, *}D_{X/S}.$ Using the map $T_{X'/S}^{\sharp} \to T_{X'/S},$ we construct a $T_{X'/S}$-gerbe over $X'$ which we denote by $(X/S)^{dR, \gamma}.$ One of the main results is the following:
\begin{theorem}\label{thm:gerbe-of-lifts}
    Let $X/S$ be a representable quasi-syntomic morphism of algebraic stacks over $\bF_p$. The gerbe of splittings of $\nu_{X/S}^{\gamma}: (X/S)^{dR, \gamma} \to X'$ is identified with the gerbe of liftings of $X/S$ to $W_2(S).$ 
\end{theorem}
In particular, any flat lift of $X/S$ to $W_2(S)$ gives rise to a global Cartier equivalence.
\begin{corollary}\label{global-cartier-from-W}
    Any flat lift $\tilde{X}/W_2(S)$ gives rise to an equivalence $(X/S)^{dR, \gamma} \simeq B_{X'}T_{X'/S}$ of $X'$-stacks. One can identify quasi-coherent sheaves on $(X/S)^{dR, \gamma}$ with locally nilpotent modules over $F_{X/S, *} D_{X/S} \otimes_{S T_{X'/S}} \hat\Gamma T_{X'/S}$ and Cartier duality identifies $\QCoh(B_{X'}T_{X'/S}) \simeq \QCoh(\widehat{\bT_{X'/S}^{*,\gamma}})$ implying $\MIC_{\gamma}^{\cdot}(X/S) \simeq \HIG^{\cdot}_{\gamma}(X'/S).$ In particular, one obtains an equivalence $\MIC_{\le p-1}(X/S) \simeq \HIG_{\le p-1}(X'/S).$
\end{corollary}
Moreover, any flat lift $\tilde{S} \to \bZ/p^2$ identifies the gerbe of flat lifts of $X/S$ to $W_2(S)$ with the gerbe of flat lifts of $X'/S$ to $\tilde{S}.$ In particular, it makes sense to compare the equivalence from \Cref{global-cartier-from-W} with the global Cartier equivalence of \Cref{global-of-intro} in the presence of $\tilde{S}.$ In \Cref{main:global-ov} we prove that they are equivalent.
By construction, $(X/S)^{dR, \gamma}$ comes equipped with a map $(X/S)^{dR} \to (X/S)^{dR, \gamma}$. By \Cref{thm:gerbe-of-lifts}, any lift $\tilde{X}/W_2(S)$ defines a splitting $(X/S)^{dR, \gamma} \simeq B_{X'}T_{X'/S},$ thus we get a map $(X/S)^{dR} \to (X/S)^{dR, \gamma} \simeq B_{X'}T_{X'/S}$. This defines a $F_{X/S}^*T_{X'/S}$-torsor on $(X/S)^{dR}$ which, by descent along $X \to (X/S)^{dR},$ can be viewed as a $F_{X/S}^*T_{X'/S}$-torsor with a flat connection on $X$. We identify the underlying torsor on $X$ with the torsor of strong lifts of $F_{X/S}$ to $\tilde{X},$ see \Cref{def:strong-frob-lift}. Moreover, we prove that any flat lift $\tilde{S}$ of $S$ to $\bZ/p^2$ identifies this torsor on $(X/S)^{dR}$ with the torsor constructed by Ogus and Vologodsky for $\tilde{X'} := \tilde{X} \times_{W_2(S)} \tilde{S}$, where $\tilde{S} \to W_2(S)$ is the canonical map.
\subsection{Local Cartier transform} Unfortunately, we do not have a statement similar to \Cref{thm:gerbe-of-lifts} in the case of $\nu_{X/S}: (X/S)^{dR} \to X'$. One source of splittings of $\nu_{X/S}$ comes from the fact that $\nu_{X/S}$ is an affine map. Namely, it was observed by Bhatt that lifts of $X$ and $S$ to $\bZ/p^2$ with a compatible lift of Frobenii give rise to an isomorphism $$F_{X/S, *}\dR_{X/S} \simeq (\bigoplus_i \Omega^i_{X'/S}[-i], 0)$$
of derived $\cO_{X'}$-algebras. In particular, this data defines a splitting of the $T_{X'/S}^\sharp$-gerbe $\nu_{X/S}: (X/S)^{dR} \to X'.$  In the absence of a lift of $S$, and assuming only a lift of $X$ to $W_2(S)$ with a strong Frobenius lift, we adapt Bhatt’s argument to prove the following.
\begin{theorem}\label{th::local-OV-intro}
Given a flat lift $\tilde{X}/W_2(S)$ of  $X/S$ and a map $f: W_2(X) \to \tilde{X}$, which reduces to the identity map on $X$, via the natural embedding $X \mono W_2(X)$ and $X \mono \tilde{X},$ we obtain a section of $\nu_{X/S}: (X/S)^{dR} \to X'$. 
\end{theorem}
Another way to construct the splitting of $\nu_{X/S}$ is based on the following observation. Let $E$ be a locally free sheaf on $(X/S)^{dR}$ and $f: (X/S)^{dR} \to B(E)$ be a map classifying an $E$-torsor on $(X/S)^{dR}$. Denote $\bar{f}: X \xrightarrow{\pi_{X/S}} (X/S)^{dR} \xrightarrow{f} B(E)$ to be the underlying $E$-torsor $Y \to X$ on $X.$ Applying $(-/S)^{dR}$ gives $(Y/S)^{dR} \to (X/S)^{dR}$ an $(\bV(E)/S)^{dR}$-torsor and we claim it is classified by the composition 
$$(X/S)^{dR} \xrightarrow{f} B(E) \to B(E/E^\sharp) \simeq B((\bV(E)/S)^{dR}).$$
Therefore, any trivialization of $\bar{f}: X \to B(E)$  gives rise to a map 
$$(X/S)^{dR} \to \fib(B(E) \to B((\bV(E)/S)^{dR})) \simeq B(E^\sharp) $$
and we describe this $E^\sharp$-torsor on $X$ with its flat connection explicitly. Namely, if $s: X \to Y$ is a section provided by the trivialization of $\bar{f},$ then its PD-envelope $D_{X}(Y)$ has a natural connection compatible with the connection on $Y$. This is motivated by \cite[Remark 2.4]{MR2373230}.
\newline
Now, a lift $\tilde{X}/W_2(S)$ gives rise to $$(X/S)^{dR} \to (X/S)^{dR, \gamma} \simeq B_{X'}(T_{X'/S}), $$
a $F_{X/S}^*T_{X'/S}$-torsor on $X$ equipped with a flat connection. We observe that this torsor can be split in the presence of $f: W_2(X) \to \tilde{X}$ which reduces to the identity map on $X.$ Applying the observation above, we obtain a map $(X/S)^{dR} \to B_{X'}(T_{X'/S}^\sharp)$ which we check to be an isomorphism of $T_{X'/S}^\sharp$-gerbes.
\begin{corollary}
    Let $X/S$ be a smooth morphism. Let $\tilde{X}$ be a lift to $W_2(S)$ equipped with a strong Frobenius lift\footnote{We refer to \Cref{def:strong-frob-lift} for the definition.}. One obtains a symmetric monoidal equivalence $C_f: \MIC^\cdot(X/S) \simeq \HIG^\cdot(X'/S).$ If, moreover, we are given a flat lift $\tilde{S} \to \bZ/p^2$ of $S$ and a lift $\tilde{X}_S$ of $X$ to $\tilde{S}$, then $f$ gives rise to a map $\tilde{F}_f: \tilde{X}_S \to \tilde{X'} := \tilde{X} \times_{W_2(S)} \tilde{S}$ lifting the relative Frobenius $F_{X/S}: X \to X'.$ Then $C_f$ is equivalent to the local Cartier transform of Ogus-Vologodsky which is a symmetric monoidal equivalence $C_{\tilde{F}_f}: \MIC^{\cdot}(X/S) \simeq \HIG^{\cdot}(X'/S)$ constructed from $\tilde{F}_f: \tilde{X}_S \to \tilde{X'},$ see \cite[Theorem 2.11]{MR2373230}.
\end{corollary}
We also verify that the two splittings of the de Rham gerbe constructed from a strong Frobenius lift agree.

As a corollary, we obtain the local Cartier equivalence for all toric fibrations \(X \to S\), since they satisfy the conditions of \Cref{th::local-OV-intro}, and for ordinary abelian schemes \(A/S\); see \cite[Theorem A]{MR4157014}.

\subsection{Powers of Frobenii}
For a smooth family $X \to S$ equipped with a lift $\tilde{X}/W_2(S),$ the global Cartier equivalence yields an equivalence between certain categories of 'complete' modules.
In order to get a decompleted version of this equivalence, we study the $\hat T_{X'/S}$-gerbe over $X'$ obtained from the $T_{X'/S}^\sharp$-gerbe $(X/S)^{dR} \to X'$ by pushing out along $T_{X'/S}^\sharp \to \hat T_{X'/S},$ we denote it by $\widehat{\nu}_{X/S}: \widehat{(X/S)^{dR}} \to X'$. In pursuit of this, we prove the following.
\begin{theorem}\label{decompleted}
Let $X/S$ be a smooth family of schemes over $\bF_p.$
\begin{itemize}
    \item Given a flat lift $\tilde{X}/W_2(S)$ of $X/S$ together with a lift of some power of $F_{X/S},$ we obtain a splitting of $\widehat{\nu}_{X/S}: \widehat{(X/S)^{dR}} \to X'.$ 
    \item $\cD_{qc}(\widehat{(X/S)^{dR}})$ is identified with the derived category of modules over the algebra $D_{X/S}^{\gamma} := F_{X/S, *}D_{X/S} \otimes_{ST_{X'/S}} \Gamma T_{X'/S}$.
\end{itemize}
\end{theorem}
The first part of the theorem is a corollary of a refined statement. Namely, by pushing out the de Rham stack along $T_{X'/S}^{\sharp} \to T_{X'/S} \otimes_{\bG_a} \alpha_{p^n}$ we get an fppf-torsor for $T_{X'/S} \otimes_{\bG_a} \alpha_{p^n}$ which we denote by $(X/S)^{dR, n} \to X'.$ This gerbe is equivalent to the gerbe of lifts $F_{X/S}^n$ to $W_2(S).$ We denote the category of quasi-coherent sheaves on a trivial $\hat T_{X'/S}$-gerbe by $\HIG_{\gamma}(X'/S).$ Via the Cartier duality it is identified with $\QCoh(\bT^{*, \gamma}_{X'/S})$ equipped with the convolution monoidal structure. Combining both parts in \Cref{decompleted} gives the following.
\begin{corollary}
    For smooth $X/S$ with a flat lift $\tilde{X}/W_2(S)$ together with a lift of $F_{X/S}^n,$ one obtains a symmetric monoidal equivalence $\MIC_{\gamma}(X/S) \simeq \HIG_{\gamma}(X'/S)$.
\end{corollary}
\subsection{Applications.}
The advantage of allowing families $X \to S$ whose inputs are (reasonable) stacks is that one can leverage the torsor property to obtain useful results about schemes. The two cases we consider are the following.
\begin{itemize}
    \item[(1)] If $(X, D)$ is an snc pair with smooth $X/\bF_p$, one obtains $f: X \to (\bA^1/\bG_m)^n =:  S_n$ classifying $D$. Contemplating various versions of the de Rham stacks, one obtains versions of non-abelian Hodge correspondence related to the logarithmic geometry of $(X, D).$ See \Cref{log-F-lift} and \Cref{log-just-lift}.
    \item[(2)] If $G$ is an affine algebraic group acting on $X,$ one can contemplate either various de Rham stacks of $X/G$, or of the family $X/G \to BG.$ This gives versions of the non-abelian Hodge correspondence for weakly and strongly $G$-equivariant $D$-modules on $X$, see \Cref{cor:G-equiv-local}.
\end{itemize}
A version of logarithmic non-abelian Hodge theory in characteristic $p$ was developed by Schepler in \cite{MR2708584}. Our results appear to be of a different nature: on the de Rham side, we do not impose any nilpotence condition on the residues, but allow all logarithmic connections with nilpotent $p$-curvature. On the Higgs side, however, we consider Higgs modules in $1/p$-parabolic sheaves relative to the divisor defining the logarithmic structure.

\subsection{Conventions.}
We list some conventions.
    \begin{itemize}
        \item[1.] All products are derived unless otherwise stated. In particular, for a module $M$, by $M/p$ we mean $M \otimes^L \bF_p$.
        \item[2.] Let $S \mono \tilde{S}$ be a nilpotent thickening. By a lift of $X/S$ to $\tilde{S}$ we mean a map $\tilde{X} \to \tilde{S}$ such that $\tilde{X} \times_{\tilde{S}} S \simeq X.$
        \item[3.] For a quasi-syntomic map $X \to S$ we denote by $T_{X/S} \in \Perf(X)$ the $\cO_X$-linear dual of $L_{X/S}.$
    
    \end{itemize}

\subsection{Acknowledgments}
This paper owes its existence to Bhargav Bhatt, who explained the main idea and provided constant support. I am grateful to Bogdan Zavyalov, with whom I started this project, but who declined to be a coauthor. I am also extremely grateful to Vadim Vologodsky for many enlightening discussions.

I also thank Piotr Achinger, Michael Barz, Andy Jiang, Mitya Kubrak, and Sasha Petrov. Special thanks to Michael Barz for helpful comments on an earlier draft.

\section{Preliminaries}
\subsection{Animated rings}
Even though the main results concern schemes, it will be useful for us to work with derived stacks. For this, recall the $\infty$-category $\CAlg_{\bZ}^{\an}$ of animated commutative rings, which can be defined by localizing the category of simplicial rings at weak equivalences. The resulting $\infty$-category inherits many good properties. For example, it can be shown that it is the universal $\infty$-category containing the $1$-category $\CAlg_{\bZ}$ and closed under sifted colimits. This procedure (\textit{i.e.}, free generation under sifted colimits) is called animation, and we refer to \cite[\S 5.1.4]{MR4681144}. Using it, one defines the category of derived schemes $\dSch_{\bZ}$ and derived stacks $\dSt_{\bZ}$. For example, for us a derived stack will be an accessible fppf sheaf $\CAlg_{\bZ} \to \Spc$. For an animated ring $A$, one also has the category of animated $A$-algebras $\CAlg_{A}^{\an}$, defined as the slice category $(\CAlg_\bZ^{\an})_{/A}$.

Let us recall the cotangent complex for derived stacks. Let $X/A$ be a derived stack over an animated ring $A$. Its cotangent complex $L_{X/A},$ when it exists, satisfies the following universal property: for any $x \in X(R)$ one has an equivalence $\Map_{R}(x^* L_{X/A}, M) \simeq \Der_x(X/A, M)$ where the right-hand side is the space of $A$-linear derivations of $X \to \Spec(A)$ at $x$ valued in an $R$-animated module $M$, see \cite{MR4681144}. In particular, let $R' \to R$ be a square zero extension of animated $A$-algebras with the fiber $N \in \cD(R)$ (see \cite[\S 5.1.9]{MR4681144}). The fiber of $X(R') \to X(R)$ over $x \in X(R)$, if non-empty, is a torsor for the animated abelian group $\Map_R(x^*L_{X/A}, N). $

Mostly, we will be working with $\CAlg_{\bF_p}^{\an}$ which can also be defined via animation of $\CAlg_{\bF_p}.$ In particular, every animated $\bF_p$-algebra $A$ carries the Frobenius endomorphism $F_A: A \to A$.  Let us collect some facts about animated $\bF_p$-algebras.

\begin{lemma}\label{Frob-factors-1-truncated}
    Let $A$ be a $1$-truncated animated $\bF_p$-algebra. Then there exists an essentially unique map of $\bF_p$-algebras $F_A': \pi_0(A) \to A$ such that the composition $A \to \pi_0(A) \xrightarrow{F_{A'}} A$ is homotopic to $F_{A}$. In particular, this map is characterized by the property that $\pi_0(A) \xrightarrow{F_A'} A \to \pi_0(A)$ is equal to $F_{\pi_0(A)}.$
\end{lemma}

\begin{proof}
    Uniqueness follows from a more general statement: let $f: A \to B$ be a map of $1$-truncated animated rings such that $\pi_1(f)=0$, then the fiber of $\Map(\pi_0(A), B) \to \Map(A, B)$ over $f$, if non-empty, is a torsor for $\Map(L_{\pi_0(A)/A}, \pi_1(B)[1])$ and the latter space is contractible since $L_{\pi_0(A)/A}$ is $2$-connective. In particular, such map $f$ is uniquely determined by $\pi_0(f).$ For the existence part: given $f: A \to B$ a map of $1$-truncated animated algebras, the obstruction of deforming the map $\pi_0(A) \to \pi_0(B)$ of $A$-algebras to $\pi_0(A) \to B$ lies in $\Map_{A}(L_{\pi_0(A)/A}, \pi_1(B)[2])$. Since $\pi_2 L_{\pi_0(A)/A} = \pi_1(A),$ the latter space is discrete and equivalent to $\Hom_{\pi_0(A)}(\pi_1(A), \pi_1(B))$. Under this identification, the obstruction class corresponds to $\pi_1(f).$ To finish the proof note that the Frobenius map for animated $\bF_p$-algebras induces $0$ on higher homotopy groups.
\end{proof}

\subsection{Ring stacks}
For any smooth scheme $X/\bQ$, Simpson introduces a stack $X^{dR}$ which \textit{geometrizes} de Rham cohomology in the sense that $\Vect(X^{dR})$ is identified with the category of vector bundles on $X$ endowed with a flat connection. Moreover, for any $(E, \nabla) \in \Vect(X^{dR})$, one has $$R\Gamma(X^{dR}, (E, \nabla)) \simeq R\Gamma(X, E \xrightarrow{\nabla} E \otimes \Omega^1_{X} \to \cdots).$$ Following this idea, Drinfeld and Bhatt-Lurie constructed, for any $p$-adic formal scheme $X$, certain stacks that \textit{geometrize} the theory of prismatic cohomology. The crucial notion for this construction is the notion of a \textit{ring stack}. Namely, given an fppf sheaf on $A$-algebras $$\cR: \CAlg_{A} \to \CAlg_{B}^{\an}$$ valued in animated $B$-algebras, one defines a cohomology theory on the category of $B$-schemes.

\begin{construction}(transmutation)
    Let $\cR$ be an $A$-algebra stack. Define a prestack $\dSt_{B} \to \Fun(\CAlg_{A}^{\an}, \Spc)$ which takes $X/B$ to $X^{\cR}$ defined by $X^{\cR}(R) := \Map_{B}(\Spec(\cR(R)), X).$
\end{construction}
The ring stacks we will need are $1$-truncated. These admit an explicit model discovered by Drinfeld in \cite{drinfeld2021notion}. Namely, he constructed a $(2,1)$-category whose objects are morphisms of $R$-modules $d \colon I \to R$, called quasi-ideals, where $R$ is a commutative ring, $I$ is an $R$-module, and such that for any $x, y \in I$ one has $d(x)y = x d(y)$. He also explicitly describes the mapping groupoids, but we do not need the full description. Let us only mention that an obvious commutative square $(I \to R) \to (J \to S)$ of quasi-ideals, where $f \colon R \to S$ is a map of rings and $I \to f_*J$ is a map of $R$-modules, gives a point of the mapping groupoid. This notion also makes sense when working with schemes, i.e., if $R$ is a ring scheme and $I$ is an $R$-module, then $d \colon I \to R$ is a quasi-ideal in schemes if $d(x)y = x d(y)$ for any test points $x, y$ of $I$.
To give an example, recall the Witt vector ring scheme $W$. It can be extended to all animated rings; see \cite[Appendix A]{bhatt2022prismatization}.
\begin{definition}\label{def:de-Rham-ring-stack}
    Denote $\bG_a^{dR}$ to be the $W$-algebra stack $\Cone(F_*W \xrightarrow{p} F_*W).$
\end{definition}

\begin{remark}\label{rem:fp-str-two-frobs-the-same}
For any ring $R$ we have a commutative diagram 
\[\begin{tikzcd}
	{F_*W(R)} && {W(R)} \\
	{F_*W(R)} && {F_*W(R)}
	\arrow["V", from=1-1, to=1-3]
	\arrow[equals, from=1-1, to=2-1]
	\arrow["F", from=1-3, to=2-3]
	\arrow["p", from=2-1, to=2-3].
\end{tikzcd}\] It is a map of quasi-ideals and thus defines $\pi: \bG_a \to \bG_a^{dR}$. In particular, $\bG_a^{dR}$ is an $\bF_p$-algebra stack and one has its Frobenius endomorphism $F: \bG_a^{dR} \to \bG_a^{dR}$. Let $R$ be an $\bF_p$-algebra, then we have two endomorphisms of $\bG_a^{dR}(R):$ namely, $\bG_a^{dR}(F_R)$ and $F(R).$ We note that they are naturally homotopic.\footnote{We warn the reader it is not automatic for an $\bF_p$-algebra stack. For example, it fails for $\bG_a \otimes^L_{\bZ} \bF_p$.} See \cite[Remark 5.1.10]{bhatt-lecture-notes} or use \Cref{Frob-factors-1-truncated}. In particular, for any $X/\bF_p$ one has a natural homotopy between $F_{X^{dR}}, F_{X}^{dR}: X^{dR} \to X^{dR}.$
\end{remark}

The map $\pi: \bG_a \to \bG_a^{dR}$ from \Cref{rem:fp-str-two-frobs-the-same} gives rise to a map of commutative animated rings $\pi_A: A \to W(A)/p$. In particular, if $A=W(R)/p,$ there is another natural map $W(\pi_R)/p: A \to W(A)/p$ obtained by reducing $W(\pi_R)$ modulo $p.$ 
We learned the following argument from Bhatt.
\begin{lemma}\label{rem:dRdRvsdR}
    Let $R$ be an animated $\bF_p$-algebra. For $A=W(R)/p$, the two maps $\pi_A$ and $W(\pi_R)/p$ are naturally homotopic.
\end{lemma}
\begin{proof}

    \textit{Step 1.} Let $A$ be a $\delta$-ring. Let $u_A: A \to W(A/p)$ be a $\delta$-lift of $A \xrightarrow{\phi_A} A \to A/p.$ The following diagram is commutative 
\[\begin{tikzcd}
	A &&& {W(A/p)} \\
	{A/p} &&& {W(A/p)/p}
	\arrow["u_A", from=1-1, to=1-4]
	\arrow[from=1-1, to=2-1]
	\arrow[from=1-4, to=2-4]
	\arrow["{\pi_{A/p}}", from=2-1, to=2-4].
\end{tikzcd}\]
Indeed, it is enough to check it for a free $\delta$-ring. Then all rings involved are discrete and one can check commutativity on elements. The composition $A \to A/p \xrightarrow{\pi_{A/p}} W(A/p)/p$ sends $x \to \bar{x} \to [\bar{x}]^p$. Note the delta lift of $A \to A/p$ carries $x$ to $[x]+V(z)$ for some $z,$ applying Frobenius and reducing modulo $p$ gives the desired commutativity.
\newline
\textit{Step 2.} Now let $A=W(R).$ Note $u_A$ coincides with $W(\pi_R).$ Indeed, both maps are $\delta$-maps so it is enough to check it after composing with $W(W(R)/p) \xrightarrow{\text{Res}} W(R)/p$. Then the claim amounts to observing that $W(R) \xrightarrow{F} W(R) \to W(R)/p$ is homotopic to $W(R) \xrightarrow{\text{Res}} R \xrightarrow{\pi_R} W(R)/p$ which follows from the definition of $\pi_R.$
\newline
\textit{Step 3.}
Therefore, the diagram 
\[\begin{tikzcd}
	{W(R)} &&&& {W(W(R)/p)} \\
	{W(R)/p} &&&& {W(W(R)/p)/p}
	\arrow["{W(\pi_R)}", from=1-1, to=1-5]
	\arrow[from=1-1, to=2-1]
	\arrow[from=1-5, to=2-5]
	\arrow["{\pi_{W(R)/p}}", from=2-1, to=2-5]
\end{tikzcd}\]
is commutative. This finishes the proof as $\pi_{W(R)/p}$ is a morphism of $\bF_p$-algebras.
\end{proof}

For a ring stack $\cR$, we will want to use descent properties of $X^\cR$. The following lemma is essentially proven in \cite[Lemma 3.3]{petrov2025decompositionrhamcomplexquasifsplit}, but we include the proof for the reader’s convenience.
\begin{lemma}
    Let $\cR$ be a $1$-truncated $A$-algebra stack. Suppose $\pi_0(\cR), \pi_1(\cR)$ satisfy \'etale descent. For a stack $X/A$ assume that $X^{\pi_0(\cR)}$ satisfies \'etale descent. Then $X^{\cR}$ also satisfies \'etale descent.
\end{lemma}
\begin{proof}
    Given $R \to S$ an \'etale cover in $\CAlg^{an}$, we have to prove that $X^{\cR}(R) \to \lim X^{\cR}(S^{\otimes_{R} n})$ is an equivalence. Consider
\[\begin{tikzcd}
	{X(\cR(R))} &&& {\lim X(\cR(S^{\otimes_{R} n}))} \\
	\\
	{X(\pi_0(\cR)(R))} &&& {\lim X(\pi_0(\cR)(S^{\otimes_{R} n}))}
	\arrow[from=1-1, to=1-4]
	\arrow[from=1-1, to=3-1]
	\arrow[from=1-4, to=3-4]
	\arrow[from=3-1, to=3-4]
\end{tikzcd}\]
By assumption the lower horizontal arrow is an equivalence. Thus, it is enough to check it fiberwise. For $x \in X(\pi_0(\cR)(R))$ the fiber of the left vertical arrow is $\Map(x^* L_{X}, \pi_1(\cR)[1]). $ Fiber on the right over $x$ is $\lim \Map_{\pi_0(\cR)(S^{\otimes_R n})}(x^* L_X, \pi_1(\cR)(S^{\otimes_R n})[1])$.
\end{proof}

 \begin{remark}\label{rem:transmutations-commute-limits}
     For a ring stack $\cR$, the presheaf sending a derived stack $X$ to the prestack $X^{\cR}$ commutes with arbitrary limits. Indeed, for any diagram $X_i$ of derived stacks one has $(\lim X_i)(A) = \lim X_i (A)$ for any animated ring $A$.
 \end{remark}


 \begin{lemma}\label{torsor-ring-stack}
    Let $\cR$ be a $\bF_p$-algebra stack with $\pi_0 (\cR)=\bG_a$ and $\pi_1(\cR) := G$ representable by an affine group scheme such that semiperfect algebras are $\pi_1(\cR)$-acyclic. Assume that $\pi_0(\cR(S)) = S$ for any semiperfect algebra $S$. If $X/\bF_p$ is lci, then the natural map $\nu_X: X^{\cR} \to X$ realizes the source as a quasi-syntomic $G \otimes_{\bG_a}  L_{X}^{\vee}[1]$-torsor over the target.
\end{lemma}
\begin{proof}
    This is essentially proven in \cite[Footnote 18]{bhatt-lecture-notes} but we recall the idea. The action of $\Map(L_X, G[1])$ on $X^{\cR} \to X$ comes from the derived deformation theory since $\nu: \cR \to \bG_a$ is a square-zero extension by $G[1].$ Namely, for a point $x: \Spec(R) \to X$ the fiber of $X^{\cR}(R) = X(\cR(R)) \to X(R)$ over $x$, if nonempty, is a torsor for $\Map_{R} (x^* L_{X}, G[1](R))$. When $X$ is lci over $\bF_p$ we get $$\Map_{R} (x^* L_{X}, G[1](R)) \simeq x^* L_{X}^{\vee} \otimes_{R} G[1](R) \in \cD_{qc}^{\le 0}(R)$$ due to dualizability of $L_{X}.$ To finish the proof it is enough to construct a section of $X^{\cR} \to X$ fppf-locally. Passing to open subschemes of $X$, we can assume $X=\Spec(R)$ where $R$ is quasi-syntomic, then giving a section of $\nu_X$ is equivalent to giving a section of $\nu(R): \cR(R) \to R.$    Moreover, we can find a faithfully flat map $f: R \to S$ where $S$ is quasi-regular semiperfect and $\pi_1L_S$ is a finite free $S$-module by \cite[Remark 4.29]{MR3949030}. Thus, we assume $X=\Spec(S)$ a qrsp ring with $L_X=M[1]$ where $M$ is a finite free $S$-module of finite rank. In that case, the map $\cR(S) \to S$ exhibits the source as a square-zero extension with the fiber $G[1](S).$ Therefore, we get the obstruction class $\nu^*L_S \to G[2](S)$ whose vanishing is sufficient for existence of a section. Since $L_S=M[1]$ with $M$ being finite free, the claim follows from the assumption that $G$ is acyclic on semiperfect rings.

\end{proof}

\subsection{Derived rings and affine stacks}
We extensively use the notion of affine stack which is based on a notion of a derived ring, discovered by Mathew. This notion is supposed to extend animated rings to the nonconnective setting. A typical example is the complex $R\Gamma(X, \cO_{X}) \in \cD(k)$ for a scheme $X/k.$ More generally, any cosimplicial commutative ring gives rise to an example. Moreover, coconnective derived rings are exactly cosimplicial commutative rings. For us, the main example of a derived ring will be derived de Rham cohomology $\dR_{B/A}$ of a map $A \to B$ between commutative rings.  In general, it is not represented by a cosimplicial commutative ring since $H^i(\dR_{B/A})$ might not vanish for $i<0.$ For an animated algebra $A$ we denote the category of derived $A$-algebras by $\DAlg_A$, see \cite[\S 4.2]{raksit2020hochschildhomologyderivedrham}. Any derived $A$-algebra $R \in \DAlg_{\bF_p}$ gives rise to a derived stack $\Spec(R)$ whose functor of points takes $A \in \CAlg^{\an}_{A}$ to the space $(\Spec(R))(A) = \Map_{\DAlg_{\bF_p}}(R, A).$ Analogously, one defines affine stack in the relative setting. 
\begin{definition}(\cite[Remark 4.5]{derivedstacks})
A morphism $X \to S$ of derived $k$-stacks is said to be affine if for every derived affine scheme $\Spec(A)$ with a map $\Spec(A) \to S,$ the fiber product $X \times_S \Spec(A)$ is an affine derived stack.
\end{definition}
\begin{remark}
    It is clear from the definition that affine $S$-stacks are closed under limits. Similarly, if $X \to Y$ is affine, then for any $Y' \to Y,$ the map $X \times_Y Y' \to Y'$ is affine. Therefore, if $f: X \to Y$ is an $S$-map with affine $X\to S,$ then $f$ is affine. Indeed, for any $\Spec(A) \to Y$ one has $\Spec(A) \times_Y X = (X \times_S Y) \times_Y \Spec(A)$.
\end{remark}
For a scheme $X$ with a derived $\cO_X$-algebra $\cA \in \DAlg(\cD(X))$ one defines an affine map $\underline{\Spec}_X(\cA) \to X$. The main object of interest in this paper is the de Rham stack $X^{dR}$ for a scheme $X/\bF_p$ and its relative version $(X/S)^{dR}$ for a morphism $X \to S$ of $\bF_p$-schemes.  We will use that $(X/S)^{dR}$ is an affine $X \times^L_{S, F_S} S  := X'$-stack represented by $F_{X/S, *} \dR_{X/S} \in \cD_{qc}(X').$ 

\begin{remark}\label{rem:affine-torsor}
    For a vector bundle $E$ on a scheme $X$, a stack $B^n_X E^{\sharp}$ is affine over $X$ for any $n \ge 0.$ Indeed, Zariski localising $X$ we can assume $E=\bG_a^n$ and then it follows since $B^n\bG_a^\sharp$ is an affine stack, see \cite[Proposition 4.14, 4.19]{derivedstacks}. Moreover, if $a: E_1 \to E_0$ is a map of vector bundles, then $B^n_XE^{\sharp} \to X$ is affine for any $n$ where $E:=\cofib(a)$. Indeed, note $B^n_{X}E^{\sharp} = \fib(B^{n+1}E_0^{\sharp} \to B^{n+1}E_1^{\sharp})$. In particular, if $f:Y \to X$ is a torsor for $E^{\sharp},$ then $f$ is affine. Indeed, if $X \to B_XE^{\sharp}$ is the map classifying $Y$, then $f$ is obtained as the pullback of the canonical map $X \to B_XE^{\sharp}$ along $f.$
\end{remark}

\section{Relative de Rham stack}

\subsection{Preliminaries on de Rham stack}

Recall from \Cref{def:de-Rham-ring-stack} the $\bG_a$-algebra stack $\bG_a^{dR} := \operatorname{Cone}(F_*W \xrightarrow{p} F_*W)$ on the category of animated $\bF_p$-algebras.
\begin{remark}\label{rem:pi_i-of-gadr}
    Recall $\pi_0(\bG_a^{dR}) =F_*\bG_a$ and $\pi_1(\bG_a^{dR})=F_*\bG_a^{\sharp}$ as $W$-modules. In particular, there is a map of ring stacks $\nu: \bG_a^{dR} \to \bG_a$ that is a square-zero extension with the fiber $\bG_a^{\sharp}[1].$ There is also a map $\pi: \bG_a \to \bG_a^{dR}$ such that $\nu \circ \pi = F_{\bG_a}.$ 
\end{remark}
\begin{definition}\label{def:dR-stack}
    For a stack $X/\bF_p$ we define a derived prestack $X^{\dR}$ by the rule  $$X^{\dR}(R) := \Map_{\dSt_{\bF_p}}(\Spec(\bG_a^{dR}(R)), X)$$.
\end{definition}
\begin{remark}\label{rem:def-of-nu-pi-for-X}
    Transmuting  $\pi, \nu$ from \Cref{rem:pi_i-of-gadr} we get $\pi_{X}: X \to X^{dR}$ and $\nu: X^{dR} \to X$ for any derived stack $X/\bF_p.$
\end{remark}
\begin{remark}\label{DRDR-geometric}
    Recall the natural transformation $\pi_{X}:X \to X^{dR}$ from \Cref{rem:def-of-nu-pi-for-X}. It gives rise to two natural transformations $(\pi_{X})^{dR}, \pi_{X^{dR}}: X^{dR} \to (X^{dR})^{dR}$. They are naturally homotopic, see \Cref{rem:dRdRvsdR}  
\end{remark}

We will need a version of the de Rham stack for any morphism $X \to S.$ 
\begin{definition}
    Given a morphism $f: X \to S$ of $\bF_p$-stacks, one defines $(X/S)^{dR}$ to be the fiber product $X^{dR} \times_{S^{dR}} S$ in the $\infty$-category of derived prestacks, where the maps are $f^{dR}$ and $\pi_S: S \to S^{dR}$. Sometimes, we will also denote it by $(X \xrightarrow{f} S)^{dR}$.
\end{definition}

\begin{remark}
    The relative de Rham stack $(X/S)^{dR}$ is an $S$-stack via the projection map $(X/S)^{dR} \to S.$ Explicitly, it represents the functor which takes $\Spec (R) \to S$ to $\Map_{S} (\Spec(F_* W(R) \otimes^L \bF_p), X)$; here $\Spec (F_* W(R) \otimes^L \bF_p)$ is a stack over $S$ via $$\Spec (F_* W(R) \otimes^L \bF_p) \to \Spec (R) \to S$$ where the first map is from \Cref{rem:pi_i-of-gadr}. In particular, this definition coincides with the one in \cite[Definition 2.5.3]{bhatt-lecture-notes}.
\end{remark}

\begin{remark}
Recall that for an $\infty$-category $\cC$ the arrow category $\cC^{\Delta^1}$ defined as $\Fun(\Delta^1, \cC).$ The relative de Rham stack defines a functor ${(-/-)}^{dR}: \operatorname{dStk}_{\bF_p}^{\Delta^1} \to \operatorname{dStk}_{\bF_p}$.
\end{remark}

\begin{remark}\label{dR-commutes-colimits}
    The relative de Rham stack functor ${(-/-)}^{dR}: \operatorname{dStk}_{\bF_p}^{\Delta^1} \to \operatorname{dStk}_{\bF_p}$ commutes with all limits. Indeed, it is true for the functor $(-)^{dR}: \operatorname{dStk}_{\bF_p} \to \operatorname{dStk}_{\bF_p}$ by \Cref{rem:transmutations-commute-limits}. It formally implies the statement as $(X/S)^{dR} = X^{dR} \times_{S^{dR}} S$ and fiber products commute with arbitrary limits.
\end{remark}

\begin{corollary} The restriction of $(-/-)^{dR}: \operatorname{dStk}_{\bF_p}^{\Delta^1} \to \operatorname{dStk}_{\bF_p}$ to the subcategory $\operatorname{Sch}_{\bF_p}^{\Delta^1}$ commutes with Tor-independent limits.
    
\end{corollary}

\begin{remark}
    Let $f: X \to S$. There is a map $X \to (X/S)^{dR} := S \times_{S^{dR}} X^{dR}$ given by $(f, \pi_X)$ with the homotopy between two compositions provided by naturality of $\pi_{-}$. We will denote it by $\pi_{X/S}.$
\end{remark}
\begin{lemma}\label{covering-for-smooth}
    If $f: X \to S$ is smooth, then $\pi_{X/S}: X \to (X/S)^{dR}$ is surjective fpqc locally.
\end{lemma}
\begin{proof}
Let $V \to S$ be an affine open subscheme. It follows from \Cref{dR-commutes-colimits} that both squares in the following diagram are cartesian.
\[\begin{tikzcd}
	{X \times_{S} V} & X \\
	{(X \times_S V/V)^{dR}} & {(X/S)^{dR}} \\
	V & S
	\arrow[from=1-1, to=1-2]
	\arrow[from=1-1, to=2-1]
	\arrow[from=1-2, to=2-2]
	\arrow[from=2-1, to=2-2]
	\arrow[from=2-1, to=3-1]
	\arrow[from=2-2, to=3-2]
	\arrow[from=3-1, to=3-2]
\end{tikzcd}\] Therefore, we can assume that $S$ is affine. In this case we refer to the proof of \cite[Theorem 2.5.6]{bhatt-lecture-notes} where the statement is proven.

\end{proof}

\begin{remark}\label{for-lci-dR-torsor}
For any derived stack $X$ one obtains $\nu_{X}: X^{dR} \to X$ by transmuting $\nu$ from \Cref{rem:pi_i-of-gadr}. When $X/\bF_p$ is an lci scheme, the map $X^{dR} \to X$ realizes the source as a quasi-syntomic torsor for the group stack $L_{X}^{\vee}[1] \otimes_{\bG_a} \bG_a^{\sharp}.$ This follows from \Cref{torsor-ring-stack} since $\pi_1 (\bG_{a}^{dR}) = F_* \bG_{a}^{\sharp}$ as sheaves of $W$-modules for fppf topology and so semiperfect algebras are $\pi_1(\bG_a^{dR})$-acyclic by \cite[Corollary 2.6.4]{bhatt-lecture-notes}. If, moreover, $X/\bF_p$ is smooth, then $X^{dR} \to X$ realizes the source as a quasi-syntomic $T_{X}^{\sharp}$-gerbe over the target. 
\end{remark}
If $X$ is any derived stack over $\bF_p$ and $s \colon X \to X^{dR}$ is any section of $\nu_X \colon X^{dR} \to X$, we obtain two points in $X^{dR}(X)$, namely $s \circ F_X$ and $\pi_X$. In the next lemma, we observe that they are canonically homotopic.

\begin{lemma}\label{rem:any-section-factorizes-pi_X}
    Let $s: X \to X^{dR}$ be any map satisfying $\nu_X \circ s \simeq \operatorname{Id}_X.$ Then there exists a canonical homotopy $F_X\circ s \simeq \pi_X.$
\end{lemma}
\begin{proof}
    Indeed, one has a canonical isomorphism $s \circ F_{X} \simeq F_{X^{dR}} \circ s$ since Frobenius commutes with any morphism of derived $\bF_p$-stacks. By \Cref{Frob-factors-1-truncated} one has a canonical identification $F_{X}^{dR} \simeq \pi_{X} \circ \nu_{X}$ and since $F_{X}^{dR} \simeq F_{X^{dR}}$ canonically by \Cref{rem:fp-str-two-frobs-the-same}, we get $F_{X^{dR}} \circ s \simeq \pi_{X} \circ \nu_{X} \circ s \simeq \pi_{X}.$
\end{proof}


\begin{lemma}\label{BTsharp-equivariant-map}
    Let $f: Y \to X$ be a map over $\bF_p.$ Then $f^{dR}: Y^{dR} \to X^{dR} \times_{X'} Y'$ is equivariant for the map $\Map(f^*L_X, \bG_a^\sharp[1]) \to \Map(L_Y, \bG_a^\sharp[1])$ of $Y$-group stacks.
\end{lemma}
\begin{proof}
    This follows since the action comes from derived deformation theory. Let $A' \to A$ be a square-zero extension with the ideal $I \in \cD^{\le 0}(A).$ Fix $x \in X(A)$ and $y \in Y(A)$ compatibly with respect to $f.$ The fiber of $X(A') \to X(A)$ over $x$, if non-empty, is a torsor for $\Map_{A}(x^*L_X, I)$ and similarly for $y.$ Pulling back the map $df: f^*L_X \to L_Y$ along $x$ gives $\Map_A(y^*L_Y, I) \to \Map_{A}(x^*L_X, I)$ and the map $Y(A') \to X(A')$ is equivariant with respect to this map.
\end{proof}

\begin{remark}\label{delta-absolute-splitting}
    Let $\tilde{X}$ be a flat $\bZ_p$-scheme with a $\delta$-structure. Let us show how that splits $X^{dR} \to X,$ it is due to \cite[Remark 5.13]{bhatt2022prismatization}. First note  $$ X(R) = \Map_{\bF_p} (\Spec (R), X) = \Map_{\bZ_p} (\Spec (R), \tilde{X}) = \Map_{\delta-\bZ_p} (\Spec (W(R)), \tilde{X}).$$ Using the natural map $\Map_{\delta-\bZ_p}(\Spec (W(R)), \tilde{X}) \to \Map_{\bF_p} (\Spec (W(R)) \otimes^L \bF_p, X),$ one obtains a splitting $X \to X^{dR}.$ It was observed by Bhatt and later by Petrov and Vologodsky, that it is enough to have a flat lift of $X$ to $\bZ/p^2$ with a lift of Frobenius to split $X^{dR} \to X$. One proof is based on exploiting affinness of $X^{dR} \to X$, we refer to \cite[Proposition 5.14]{petrov2023nondecomposabilityrhamcomplexnonsemisimplicity} for the argument. Moreover, it was observed by Petrov and Vologodsky that the splitting coming from a flat $\delta$-scheme $\tilde{X}$ only depends on $\tilde{X} \times \Spec (\bZ/p^2)$ with the Frobenius on it. Moreover, the splitting is isomorphic to the splitting constructed using affinness of the de Rham stack.
\end{remark}

\begin{lemma}\label{torsor-structure}
    There exists a map $\nu_{X/S}: (X/S)^{dR} \to X' := X \times_{S, F_S} S$ such that for an lci $X/S$ it exhibits the source as a quasi-syntomic torsor for $L_{X'/S}^{\vee}[1] \otimes_{\bG_a} \bG_a^{\sharp}$.
\end{lemma}
\begin{proof}
The map $(X/S)^{dR} \to X'$ comes from the following commutative diagram
\[\begin{tikzcd}
	&& {X^{dR}} &&&& X \\
	\\
	{(X/S)^{dR}} &&&& {S^{dR}} &&&& S \\
	\\
	&& S
	\arrow[from=3-1, to=1-3]
	\arrow["{f^{dR}}", from=1-3, to=3-5]
	\arrow[from=3-1, to=5-3]
	\arrow["{\pi_{S}}", from=5-3, to=3-5]
	\arrow["{\nu_{X}}", from=1-3, to=1-7]
	\arrow["f", from=1-7, to=3-9]
	\arrow["{F_S}", from=5-3, to=3-9]
	\arrow["{\nu_{S}}", from=3-5, to=3-9]
\end{tikzcd}\]
For a scheme $Y/S$ denote $\Hom(L_{Y/S}, \bG_a^{\sharp})$ to be the group stack on $Y$ sending $\eta: \Spec (R) \to Y$ to $\tau^{\le 0}\Hom_{R} (\eta^* L_{Y/S}, R^{\sharp}[1])$ computed in $\cD(R).$
\Cref{torsor-ring-stack} implies that the fiber product

\begin{equation}
\begin{tikzcd}
	{\Hom(L_{X'/S}, \bG_{a}^{\sharp}[1])} &&& {\Hom(L_{X}, \bG_{a}^{\sharp}[1])} \\
	\\
	\\
	{*} &&& {\Hom(L_{S}, \bG_{a}^{\sharp}[1])}
	\arrow[from=1-1, to=1-4]
	\arrow[from=1-1, to=4-1]
	\arrow[from=4-1, to=4-4]
	\arrow[from=1-4, to=4-4]
\end{tikzcd}
\label{action-for-torsor}
\end{equation}
of $X'$-group stacks acts on $(X/S)^{dR}.$
Indeed, it is enough to check that the fiber of $\Hom(L_{X}, \bG_{a}^{\sharp}[1]) \to \Hom(L_{S}, \bG_{a}^{\sharp}[1])$ is identified with $\Hom(L_{X'/S}, \bG_{a}^{\sharp}[1]).$ This amounts to showing that $(f \circ \rho)^* L_{X} \to \rho^* L_{X} \to L_{X'/S}$ is a fiber sequence where $\rho: X' \to X$ is the canonical projection. This follows from the base change $L_{X'/S} = \rho^* L_{X/S}$ and the transitivity triangle $f^{*}L_{S} \to L_{X} \to L_{X/S}.$
To construct a local splitting of $\nu_{X/S}: (X/S)^{dR} \to X'$ we proceed as in  \Cref{torsor-ring-stack}. First, we can assume that $S=\Spec (R)$ and $X = \Spec (A)$. Then write $$\Map_R(F_R^*A, W(F_R^*A) \otimes^L \bF_p) = (X/S)^{dR}(X') \to X'(X')=\Map_R(F_R^*A, F_R^*A).$$ In particular, finding a section of $\nu_{X/S}$ is equivalent to finding a section of $R: W(F_R^*A) \otimes^L \bF_p \to F_R^*A$. By \cite[Variant 4.33]{MR3949030} we can find a quasisyntomic cover $A \to B$ with $B$ being a quasiregular semiperfect $R$-algebra such that $\pi_1(L_{B/R})$ is a finite free $B$-module. In this case the map $W(F_R^*B) \xrightarrow{p} W(F_R^*B)$ is surjective and thus $W(F_R^*B) \otimes^L \bF_p \to F_R^*B$ is a square-zero extension with fiber $\bG_a^\sharp[1](F_R^*B)$. Thus, we obtain a map $L_{F_R^*B/R} \to \bG_a^{\sharp}[2]$ whose vanishing is sufficient for splitting this square-zero extension in $R$-algebras. Since $B$ is semiperfect, the claim follows.
\end{proof}


\begin{remark}\label{p-curv-equivariant-relative}
Let $X \to Y$ be a map of $S$-stacks. Then the induced map
\[
(X/S)^{dR} \to (Y/S)^{dR} \times_{Y'} X'
\]
is equivariant with respect to the map
\[
\Map(L_{X/S}, F_*\bG_a^{\sharp}[1]) \to \Map(f^*L_{Y/S}, F_*\bG_a^{\sharp}[1])
\]
of $Y'$-group stacks. This follows from \Cref{BTsharp-equivariant-map} and the definition of the action as in diagram \eqref{action-for-torsor}.
\end{remark}


\begin{remark}
    Note $X'$ is a classical scheme when $X/S$ is lci. This amounts to observing that for $f: A \to B=A/(f_1, \cdots, f_n),$ the $A$-modules $F_*A$ and $B$ are tor-independent. See \cite[Lemma 3.41]{bhatt2012padic} for the argument.
\end{remark}

\begin{remark}
    Let $f: X \to Y$ be a map of $S$-stacks. Recall $(X/Y)^{dR}$ is defined as $X^{dR} \times_{Y^{dR}} Y$. However, it is naturally isomorphic to $(X/S)^{dR} \times_{(Y/S)^{dR}} Y$. It is enough to observe that $(X/S)^{dR} \simeq X^{dR} \times_{Y^{dR}} (Y/S)^{dR}$ where the map $(Y/S)^{dR} \to Y^{dR}$ is the canonical projection. This follows since $$X^{dR} \times_{Y^{dR}} (Y/S)^{dR} \simeq X^{dR} \times_{Y^{dR}} (Y^{dR} \times_{S^{dR}} S) \simeq X^{dR} \times_{S^{dR}} S.$$
\end{remark}

\begin{lemma}\label{lem:dR-preserves-et-covers}
    If $f: X \to Y$ is an \'etale surjection of Noetherian algebraic stacks with a quasi-affine diagonal, then $f^{dR}: X^{dR} \to Y^{dR}$ is an \'etale  surjection.
\end{lemma}
\begin{proof}
Fix $y \in Y^{dR}(R).$ Endow $W(R)/p$ with its induced $V$-adic filtration. Then $$Y^{dR}(R) = \Map(\Spec(W(R)/p), X) \simeq \Map(\Spf(W(R)/p), Y)$$ by \cite[Corollary 1.5]{MR3672914}.
Consider $\bar{y}: \Spec(R/p) \to \Spf(W(R)/p) \to Y$ and choose an \'etale $\bar{g}: R/p \to \bar{S}$ such that $\bar{g}(\bar{y}): \Spec(\bar{S}) \to Y$ comes from $\bar{x} \in X(S/p).$ Note $\bar{S}$-corresponds to a unique \'etale algebra $g: R \to S$ and under
the equivalence $\Spf(W(R)/p)_{\acute{e}t} \simeq \Spec(R)_{\acute{e}t}$ sends $g$ to $\Spf(W(S)/p) \to \Spf(W(R)/p)$ which completes the proof.
\end{proof}

\begin{lemma}
    Let $G$ be an affine group scheme over $\bF_p.$ Then $(B_{\et}G)^{dR} \simeq B_{\et}G^{dR}$. Moreover, this equivalence sends $\pi_{B_{\et}G}: B_{\et}G \to (B_{\et}G)^{dR}$ to $B(\pi_{G}): B_{\et}G \to B_{\et}G^{dR}.$ 
\end{lemma}
\begin{proof}
    Consider an \'etale covering $f: \Spec(\bF_p) \to B_{\et} G$. Then $f^{dR}: \Spec(\bF_p) \to (B_{\et}G)^{dR}$ is an \'etale covering by \Cref{lem:dR-preserves-et-covers}. Its Cech nerve is obtained from the Cech nerve $G^{\bullet}$ of $f$ by applying $(-)^{dR}$, since $(-)^{dR}$ preserves products. Thus, it is identified with $(G^{dR})^{\bullet}$ the Cech nerve of $\Spec(\bF_p) \to B_{\et}G^{dR}.$ The map $B(\pi_G)$, by definition, is represented by a map of simplicial stacks $G^{\bullet} \xrightarrow{(\pi_{G}^{dR})^{\bullet}} (G^{dR})^{\bullet}$. To finish the proof, it is enough to observe that the map $B_{\et}G \to (B_{\et}G)^{dR} \simeq B_{\et} G^{dR}$ is pointed and the induced map $G \to G^{dR}$ is $\pi_{G}.$
\end{proof}

\begin{lemma}\label{torsor-for-stacks}
    Let $X \to S$ be a representable quasi-syntomic map of algebraic stacks over $\bF_p.$ Then $(X/S)^{dR} \to X'$ is a quasi-syntomic torsor for $L_{X'/S}^{\vee} \otimes_{\bG_a} \bG_a^{\sharp}$
\end{lemma}
\begin{proof}
    It is enough to show local non-emptiness. For this, choose a smooth cover $f: T \to S$ where $T/\bF_p$ is a scheme. Also, choose a scheme $Y/S$ with a smooth cover $g: Y \to X_T := X \times_S T$ of $S$-schemes. By \Cref{torsor-structure} we can choose an $h: U \to Y$ quasi-syntomic cover such that $\nu_{U/S}: (U/S)^{dR} \to U'$ admits a section $s.$ Define $U' \to (X/S)^{dR} \times_{X'} U'$ to be equal to the identity on the second component, and the first component is $$U' \xrightarrow{s} (U/T)^{dR} \xrightarrow{h^{dR}} (Y/T)^{dR} \xrightarrow{g^{dR}} (X_T/T)^{dR} \simeq (X/S)^{dR} \times_S T \to (X/S)^{dR}.$$
\end{proof}

\subsection{Affineness}
\begin{lemma}\label{absolute-dr-affine}
    For any scheme $X/\bF_p$ the derived stack $X^{dR}$ is relatively affine over $X$ and, moreover, is isomorphic to $\underline{\Spec}_{X} (F_{X, *}\dR_{X}).$
\end{lemma}
\begin{proof}
First, let us observe that $X^{dR} \to X$ is affine. Note $\bG_a^{dR}$ is affine by \cite[Proposition 2.34]{MR4718127}. Thus, it is true for $X=\bA^n$ since affine stacks are closed under products. Thus, we know it for any affine $X$ since affine stacks are closed under cosifted limits. To finish the proof, note that for an arbitrary $X$ and open subscheme $U \xrightarrow{j}X$, the natural map $U^{dR} \xrightarrow{(j^{dR}, \nu_U)} X^{dR} \times_{X} U$ is an isomorphism. 
    
Now, it is enough to show $\nu_{X, *} \cO_{X^{dR}} \simeq F_{X, *} \dR_{X}$ for any $X.$ By \cite[Corollary 2.7.2 (3)]{bhatt-lecture-notes} or \cite[Proposition 3.6]{petrov2025decompositionrhamcomplexquasifsplit} we have this identification for any smooth $X.$ This implies the statement for any affine $X$ since both functors commute with limits in $X.$ Finally, both functors satisfy Zariski descent and thus the statement holds for any $X.$
\end{proof}

\begin{lemma}\label{rem:sharp-of-vb-and-dR}
    Let $E$ be a vector bundle over $X.$ The $X$-group stack $(\bV(E)/X)^{dR}$ is naturally identified with the quotient $\Cone(\bV(E)^{\sharp} \to \bV(E)).$ 
\end{lemma}

\begin{proof}
    There is a natural map of $X$-group stacks $\pi_{\bV(E)/X}: \bV(E) \to (\bV(E)/X)^{dR} = \bV(E)^{dR} \times_{X^{dR}} X$ given by $(\pi_{\bV(E), h}, \pr)$ where $\pi_{\bV(E)}: \bV(E) \to \bV(E)^{dR}$ and $h$ is the natural homotopy $\pi_X\circ \pr \simeq \pr^{dR} \circ \pi_{\bV(E)}$. It is enough to conclude that $\fib(\pi_{\bV(E)/X}) \simeq \bV(E)^{\sharp}.$ For this we assume $X=\Spec(A)$ and $M$ is the $A$-module corresponding to $E.$ Then the map $\bV(E)(R) \to (\bV(E)/X)^{dR}(R)$ is identified with $M \otimes_A R \xrightarrow{\operatorname{Id} \otimes \pi_R} M \otimes_A \bG_a^{dR}(R),$ the claim follows since $\bV(E)^{\sharp}(R) = M \otimes_A \bG_a^{\sharp}(R).$
\end{proof}

\begin{remark}\label{rel-dR-of-BE}
    Let $E$ be a vector bundle on $S.$ Then $B_{S}(E^\sharp)$ is identified with the relative de Rham stack of $S \to B_{S}(E)$. Indeed, by definition, the relative de Rham stack is given by the fiber product of $S \simeq (S/S)^{dR} \to B_S(E)^{dR}$ and $B_{S}(E) \xrightarrow{\pi_{B_SE/S}} (B_S(E)/S)^{dR}$. Thus, it is isomorphic to $\fib(\pi_{B_SE/S}) \simeq B_{S}(E^\sharp)$ by \Cref{rem:sharp-of-vb-and-dR}
\end{remark}

\begin{remark}
    Let $X$ be a smooth scheme with an action of a smooth group scheme $G$. Then $\nu_{X/G}: (X/G)^{dR} \to X/G$ is affine by \Cref{rem:affine-torsor}. Indeed, $T_{X/G}$ is represented by a $G$-equivariant complex $g \boxtimes \cO_{X} \to T_{X}$ in $\Perf(X),$ where $T_X$ is placed in degree $0.$
\end{remark}

\begin{remark}
    The derived stack $X^{dR}$ is not classical in general, see \cite[Warning 7.4]{bhatt2022prismatization}. By \Cref{for-lci-dR-torsor} one sees that the derived stack $X^{dR}$ is classical for lci $X$. It also follows from \cite[Corollary 8.13]{bhatt2022prismatization}.
\end{remark}

\begin{lemma}\label{dR-is-affine-relative}
    For a map of schemes $X \to S$, the map $\nu_{X/S}: (X/S)^{dR} \to X'$ from \Cref{torsor-structure} is an affine map and exhibits the source as $\underline{\Spec}_{X'} (F_{X/S, *} \dR_{X/S})$.
\end{lemma}
\begin{proof}
\textit{Step 1. Reduction to the case of affine $S$.} It is enough to check that $(X/S)^{dR} \to X'$ is affine after changing the base along $X' \times_{S} V \to X'$ for any affine open $V \subset S.$ One computes $(X/S)^{dR} \times_{X'} (X' \times_{S} V) \simeq (X/S)^{dR} \times_{S} V \simeq (X \times_{S} V/V)^{dR}.$ Thus, we can assume that $S$ is affine.

\textit{Step 2. Reduction to the case of affine $X$.} 
For an open immersion $U \to X$ the map $U^{dR} \to X^{dR}$ is a representable open immersion by \cite[Theorem 2.5.6]{bhatt-lecture-notes} and one gets that the diagram
\[\begin{tikzcd}
	U &&& {U^{dR}} \\
	\\
	X &&& {X^{dR}}
	\arrow[from=1-1, to=1-4]
	\arrow[from=1-1, to=3-1]
	\arrow[from=3-1, to=3-4]
	\arrow[from=1-4, to=3-4]
\end{tikzcd}\]
is cartesian. To prove surjectivity $X \to (X/S)^{dR}$ it is enough to do this Zariski locally on $(X/S)^{dR},$ for an open immersion $U \to X$ we get the following diagram 
\[\begin{tikzcd}
	U && {(U/S)^{dR}} && {U^{dR}} \\
	\\
	X && {(X/S)^{dR}} && {X^{dR}} \\
	\\
	&& S && {S^{dR}}
	\arrow[from=1-3, to=3-3]
	\arrow[from=1-3, to=1-5]
	\arrow[from=3-3, to=3-5]
	\arrow[from=1-5, to=3-5]
	\arrow[from=3-5, to=5-5]
	\arrow[from=3-3, to=5-3]
	\arrow[from=5-3, to=5-5]
	\arrow[from=1-1, to=1-3]
	\arrow[from=1-1, to=3-1]
	\arrow[from=3-1, to=3-3]
\end{tikzcd}\]
where the vertical upper right arrow is an open immersion which implies that $(U/S)^{dR} \to (X/S)^{dR}$ is also an open immersion. Moreover, right two squares are cartesian which implies that the left upper square is also cartesian. Moreover, if $X$ is written as a colimit of a diagram $U^{\bullet}$ consisting of open affine subschemes of $X$, then the colimit of $(U/S)^{\bullet, dR}$ is $(X/S)^{dR}$, it also follows from the proof of \cite[Theorem 2.5.6]{bhatt-lecture-notes} since the analogous statement holds for $X^{dR}$ combining with the fact that the upper right square is cartesian. Thus, we can assume that $X$ is affine. In this case $(X/S)^{dR} = X^{dR} \times_{S^{dR}} S$ is an affine stack since all terms are affine.

\textit{Step 3. Computing $\nu_* \cO_{(X/S)^{dR}}$.} The argument is similar to the one in \Cref{absolute-dr-affine}. By Zariski descent it is enough to prove it for affine $X, S$ and, moreover, it is enough to assume that the map $X \to S$ is smooth. In that case we refer to \cite[Corollary 2.7.1. (3)]{bhatt-lecture-notes}.

\begin{remark}
    Let $X \to S$ be a representable quasi-syntomic map of $\bF_p$-algebraic stacks. Then $\nu_{X/S}: (X/S)^{dR} \to X'$ is an affine map. Indeed, by \Cref{torsor-for-stacks} it is a pullback of $X' \to B_{X'}(T_{X'/S}^{\sharp})$ which is affine by \Cref{rem:affine-torsor}. Moreover, $\nu_{X/S,*} \cO_{(X/S)^{dR}}$ is identified with $F_{X/S, *} \dR_{X/S}$ by the computation for smooth schemes.
\end{remark}

\begin{remark}
    Note that $(X/S)^{dR}$ is classical for lci $X/S$ since it admits a structure of a $L_{X'/S}^{\vee}[1] \otimes_{\bG_a} \bG_{a}^{\sharp}$-torsor over $X'.$ Indeed, fppf locally on $X'$ one can split $\nu_{X/S}$ by \Cref{torsor-structure}. Since $L_{X'/S}^{\vee}[1] \otimes_{\bG_a} \bG_a^{\sharp}$ is classical, the statement follows.
\end{remark}

\end{proof}

\begin{lemma}\label{lm:regular-affine-stack}
    Let $i: Z \to X$ be a regular closed immersion. Then $(Z/X)^{dR}$ is identified with the PD-envelope $D_{Z}(X).$
\end{lemma}
\begin{proof}
    First, assume that $X=\Spec (A)$ and $Z=\Spec (A/I)$ where $I$ is generated by a regular sequence $(f_1, \cdots, f_n).$ Then \cite[Corollary 3.40]{bhatt2012padic} provides a functorial isomorphism of $\bE_{\infty}-A/\phi(I)$-algebras $\dR_{(A/I)/A } \simeq D_{I}(A)$. By Zariski descent and \Cref{dR-is-affine-relative} it gives rise to $(Z/X)^{dR} \simeq D_{Z}(X)$.
\end{proof}

\begin{remark}\label{rem:two-actions}
    Given an lci map $X \to S,$  then \Cref{torsor-structure} provides us with a torsor structure on the relative de Rham stack. In particular, one can take $X \to S$ to be a regular closed immersion. In this case, by \Cref{lm:regular-affine-stack} the relative de Rham is identified with the PD-envelope $D_{X}(S).$ The twist $X'$ is identified with the Frobenius neighbourhood of $X$ in $S$ which we denote by $X^{\phi}.$ The group stack $BT_{X'/S}^\sharp$ is identified with the PD-envelope of the zero section in the normal bundle $\mathbf{N}_{X^\phi/S}^\sharp$. Indeed, the tangent complex of a regular immersion is given by $(I/I^2)^\vee[-1]$ where $I$ is the ideal of $X \to S.$   Therefore, one obtains an action of $\mathbf{N}_{X^{\phi}/S}^{\sharp}$ on $D_{X}(S) \to X^{\phi}$. Let us compare this action with the one coming from $D$-module theory which was discovered by Ogus and explained to us by Vologodsky. Denote $f: D_{X} (S) \to X^{\phi}$. Then $f_* \cO_{D_X (S)}$ is naturally a $D$-module on $X^{\phi}$. Then the desired action is given by the integration of the $p$-curvature of $f_* \cO_{D_{X}(S)}$. More explicitly, to define an action of $\mathbf{N}_{X^{\phi}/S}^{\sharp}$ on $D_X(S) \to X^{\phi},$ it is enough to define an action of the Lie algebra of $\bN_{X^{\phi}/S}$. In local coordinates the action is given by $\partial \cdot \gamma_n(x) = c'(\partial) \gamma_n(x)$ where $c': F_{X^{\phi}}^* T_{{X^{\phi}}'} \to D_{X^{\phi}}$ is the $p$-curvature map. It is easy to see that it factors through the quotient $T_{X^{\phi}} \to N_{X^{\phi}/S}$. 
\end{remark}

\begin{lemma}\label{lm:closed-immersion-action}
    For a regular immersion $X \to S$ of $\bF_p$-schemes, the two actions of $\bN_{X^{\phi}/S}^{\sharp}$ on $D_{X} (S)$ constructed in \Cref{rem:two-actions} coincide.
\end{lemma}

\begin{proof}
    The statement is local and we can assume that $X \to S$ comes from a surjection $A \to A/I$ with the regular $I = (f_1, f_2, \cdots, f_n).$ In this case $$A/I^p = A/f_{1}^p \otimes_A \cdots \otimes_{A} A/f_{n}^p \to D_I (A) = \otimes_i D_{f_i} A$$ is a torsor for $\bN_{A/I^p} (I/I^2)^{\vee}$ and it is a product of $\bN_{A/f_i^p} (f_i/f_{i}^2)^{\vee}$-torsors. Thus, it is enough to prove it in the case of $I=(f).$ Consider the map $\bF_p[x] \xrightarrow{x \to f} A,$ then $A/f \simeq A \otimes_{\bF_p[x]} \bF_p$ and by the base change we can assume that $A=\bF_p[x]$ and $f=x.$
    Thus, we are reduced to prove it for $X = \Spec (\bF_p) \xrightarrow{} \Spec (\bF_p[x]) = S.$ In this case the map $(X/S)^{dR} \to X'$ is identified with the canonical map $\bG_a^\sharp \to \alpha_p.$ Indeed, $$(\Spec(\bF_p)/\bG_a^{dR})= \Spec(\bF_p) \times_{\bG_a} \bG_a^{dR} \simeq \fib(\bG_a \to \bG_a^{dR}) \simeq \bG_a^\sharp$$
    and $X' = \Spec(\bF_p) \times_{\bG_a, F_{\bG_a}} \bG_a \simeq \fib(\bG_a \xrightarrow{F_{\bG_a}} \bG_a) \simeq \alpha_p.$ Recall that the action of $\Hom(L_{X'/S}, \bG_a^{\sharp}[1])$ comes from the diagram \eqref{action-for-torsor}.  Note that in our case $\Hom(L_{X}, \bG_a^{\sharp}[1]) = \Spec (\bF_p)$ as a stack since $L_{X} = 0.$ In particular, we learn that $\Hom(L_{X'/S}, \bG_a^\sharp[1]) \simeq F_*\bG_a^\sharp.$ Unwinding the definitions, the action is given by $F_*\bG_a^\sharp \times \bG_a^\sharp \to \bG_a^\sharp$ sending $(t, x)$ to $V(t)+x.$ In particular, if $R=\bF_p[\varepsilon],$ then for any $x \in \bF_p[\varepsilon]^{\sharp}$ one has $\gamma_n(\varepsilon \cdot x) = \gamma_{n-p}(x) \varepsilon$ for $n \ge p$ and $0$ otherwise. It is easy to see the action coming from $p$-curvature is given by the same formula since $\partial_{x}^p\gamma_n(x) = \gamma_{n-p} (x)$ where $x$ is a coordinate on $\bG_a$ and $\partial_x$ is the associated vector field.
\end{proof}

\begin{remark}\label{rem:p-curv=action}
    Another case where the action of $B_{X'}(T_{X'/S}^{\sharp})$ on $(X/S)^{dR}$ can be explicitly understood is when $X \to S$ is smooth.  Denote $a: B_{X'} (T_{X'/S}^{\sharp}) \times_{X'} (X/S)^{dR} \to (X/S)^{dR}$ to be the action morphism. Note $(X/S)^{dR} \times_{X'} B_{X'} (T_{X'/S}^{\sharp}) \simeq B_{(X/S)^{dR}} (F_{X/S}^* T_{X'/S}^{\sharp})$ and by Cartier duality the derived category of quasi-coherent sheaves on it is identified with the $\infty$-category of pairs $(M, \theta)$ where $M \in \cD_{qc}((X/S)^{dR})$ and $\theta: M \to M \otimes_{\cO_{X}} F_{X/S}^* \Omega_{X'/S}$ is a map in $\cD_{qc} ((X/S)^{dR})$ such that $\theta$ acts locally nilpotently on $\cH^i(M)$ for any $i.$ We claim that $a^*(M) = (M \xrightarrow{\theta_{M}} M \otimes F_{X/S}^* \Omega^1_{X'/S})$ where $\theta_M$ is the $p$-curvature map of $M$. Equivalently, it is map $F_{X/S}^* T_{X'/S} \to \underline{\End}(M)$ and for a local section $\partial$ the endomorphism $\theta_M(\partial)$ is multiplication by the corresponding central element of $F_{X/S, *} D_{X/S}.$ The statement is local and by standard reductions we can assume that $X = \bA^1.$ In this case the computation is the same as in \Cref{lm:closed-immersion-action}. Indeed, the pullback along $\bG_a^{dR} \to B_{\bG_a^{dR}} (F_* \bG_a^{\sharp})$ takes $(M, \theta_M)$ to $M$. In particular, the functor $a^*$ does not change the underlying crystal. Explicitly, $a$ is given by $\Cech(\bG_a \to B_{\bG_a^{dR}} (\pi^* F_* \bG_a^{\sharp})) \to \Cech(\bG_a \to \bG_a^{dR})$. In particular, we have a map $a_1: \bG_a \times \bG_a^{\sharp} \times F_* \bG_a^{\sharp} \to \bG_a \times \bG_a^{\sharp}$ given by $a_1(x, a, b) = (x, a + V(b)).$ Note that the map induced by  $V: F_* \bG_a^{\sharp} \to \bG_a^{\sharp}$ on Lie algebras takes the standard derivation $\partial_{x}$ to $\partial_{x}^{[p]}.$ Thus, the nilpotent operator $\theta_M$ on $a^* M$ is given by the action of the central element $\partial^p.$ 
\end{remark}

In \cite[Remark 1.8]{MR2373230}, it is explained how to compute the $p$-curvature under the inverse image functor. In the next remark, we note that this can be understood via the interpretation of the $p$-curvature from \Cref{rem:p-curv=action} in terms of the torsor action, together with the fact that any $S$-map $X \to Y$ induces a map $(X/S)^{dR} \to (Y/S)^{dR}$ which is equivariant with respect to this action; see \Cref{p-curv-equivariant-relative}.
\begin{remark}
    Let $h: X \to Y$ be a map of smooth $S$-schemes. Consider the map $$H: (X/S)^{dR} \times _{X'} B_{X'}(T_{X'/S}^{\sharp}) \to (Y/S)^{dR} \times_{Y'} B_{Y'}(T_{Y'/S}^{\sharp})$$ over $h': X' \to Y'$. The derived category of quasi-coherent sheaves on the target is equivalent to the $\infty$-category of pairs $(M, \psi_M)$ where $M \in \cD_{qc}((Y/S)^{dR})$ and $$\psi_M: M \to M \otimes F_{Y/S}^*\Omega^1_{Y'/S}$$ is a horizontal map. Then $H^*$ sends $M \xrightarrow{\psi_M} M \otimes F^*_{Y/S}\Omega^1_{Y'/S}$ to the composition
    
    \begin{equation}\label{p-curv}
        h^*M \xrightarrow{h^* \psi_M} h^*M \otimes h^*F_{Y/S}^* \Omega^1_{Y'/S} \to h^*M \otimes F_{X/S}^* \Omega^1_{X'/S}
    \end{equation}
        where the last map is induced by $h^*F_{Y/S}^* \Omega^1_{Y'/S} \simeq F_{X/S}^* h'^* \Omega^1_{Y'/S} \xrightarrow{dh'} F_{X/S}^* \Omega^1_{X'/S}.$  Thus, for $E \in \MIC^{\cdot}(X/S)$ the $p$-curvature is given by \eqref{p-curv}. Indeed, the diagram
\[\begin{tikzcd}
	{(X/S)^{dR} \times B_{X'}(T_{X'/S}^{\sharp})} &&& {(X/S)^{dR}} \\
	{(Y/S)^{dR}  \times_{Y'} B_{X'}(h^*T_{Y'/S}^{\sharp})} &&& {(Y/S)^{dR} }
	\arrow["{a_{X/S}}", from=1-1, to=1-4]
	\arrow["{h^{dR} \times dh}"', from=1-1, to=2-1]
	\arrow["{h^{dR}}", from=1-4, to=2-4]
	\arrow["{a_{Y/S}}"', from=2-1, to=2-4]
\end{tikzcd}\]
is commutative by \Cref{p-curv-equivariant-relative}, and the pullback along the action maps sends a $D$-module to its $p$-curvature by \Cref{rem:p-curv=action}.
\end{remark}

\begin{corollary}
    Let $(E, \nabla) \in \MIC^{\cdot}(X/S)$ for a representable smooth map $X/S$ of algebraic $\bF_p$-stacks. Then the Frobenius pullback of the complex $F_{X/S, *}\dR(E, \nabla) \in \cD_{qc}(X')$ is isomorphic to 
    \begin{equation}\label{complex}
        E \xrightarrow{\psi_E} E \otimes F_{X/S} \Omega^1_{X'/S} \xrightarrow{\psi_E} E \otimes F_{X/S} \Omega^2_{X'/S} \to \cdots
    \end{equation} 
    where $\psi_E$ is the $p$-curvature operator. Indeed, recall from \Cref{rem:p-curv=action} that the pullback of $(E, \nabla) \in \MIC^{\cdot}(X/S)$ under the action map $$a: B_{X'}(T_{X'/S}^{\sharp}) \times_{X'} (X/S)^{dR} \to (X/S)^{dR}$$
    is a pair $(E, \psi_E)$ where $\psi_E: E \to E \otimes F_{X/S}^* \Omega^1_{X'/S}$ is a map of crystals given by the $p$-curvature. Pulling back further to $B_{X'}(T_{X'/S}^{\sharp}) \times_{X'} X$ and pushing forward to $X$ gives the complex \eqref{complex}. On the other hand, by the base change this complex is isomorphic to $F_{X/S}^*\nu_{X/S, *}(E, \nabla)$ where $\nu_{X/S}: (X/S)^{dR} \to X'$, it completes the proof.

\end{corollary}

\begin{remark}\label{rem:higgs-field=action-of-sharp}
    Let $\pi: X \to S$ be smooth. Consider the multiplication $$m: B_{X}(T_{X/S}^{\sharp}) \times_{X} B_{X}(T_{X/S}^{\sharp}) \to B_{X}(T_{X/S}^{\sharp}).$$ Note the left-hand side is the classifying stack of $\pi^* T_{X/S}^\sharp$ on $B_{X}(T_{X/S}^{\sharp})$ where $\pi: B_{X}(T_{X/S}^{\sharp}) \to X$ is the projection. Therefore, its category of quasicoherent sheaves is the $\infty$-category of pairs $(E, \theta_E)$ where $E \in \HIG^\cdot(X/S)$ and $\theta_E: E \to E \otimes \Omega^1_{X/S}$ is a nilpotent map of Higgs modules. Analogously to \Cref{rem:p-curv=action}, we learn that $m^*$ takes a Higgs module $(E, \theta_E)$ to the same pair $(E, \theta_E)$ where $\theta_E$ is emphasized to be a map of Higgs modules and not merely an $\cO_X$-linear map.
\end{remark}

\subsection{Splittings of the relative de Rham gerbe}

It was observed by Bhatt in \cite{bhatt2012padic} that a compatible lift of Frobenii on $X$ and $S$ gives rise to the formality of the de Rham complex which gives rise to a section of $\nu_{X/S}: X' \to (X/S)^{dR}.$
\begin{remark}
    Let $X/S$ be smooth. If we are given $\tilde{X}, \tilde{S}$ with lifts of $F_{X}$ and $F_{S}$, then \cite[Proposition 3.17]{bhatt2012padic} provides an isomorphism $F_{X/S, *}\dR_{X/S} \simeq (\bigoplus \Omega^i_{X'/S}, 0)$ of $\bE_{\infty}$-algebras in $\cD_{qc}(X').$ The proof uses a lift of $F_{X/S}$ to $\tilde{X} \to \tilde{X} \times_{\tilde{S}, \tilde{F}_{S}} \tilde{S}$ to construct a map $L_{X'/S}[-1] \to F_{X/S, *}\dR_{X/S}$ which splits the conjugate filtration in degree $1$. Let us note that the proof only uses a lift of $F_{S}$ through a lift of $F_{X/S}$; i.e., it is enough to have a lift $\tilde{S}$ with flat lifts $\tilde{X}, \tilde{X'}$ and a lift $\tilde{F}_{X/S}: \tilde{X} \to \tilde{X'}.$ In the next lemma, we explain how to split the de Rham gerbe in the absence of a lift of $S$.
\end{remark}
If $X/S$ is a family with a lift $\tilde{X}/W_2(S)$, then the correct replacement of a lift of the relative Frobenius is given in the following definition.
\begin{definition}\label{def:strong-frob-lift}
    Let $X \to S$ be a map over $\bF_p.$ Let $\tilde{X}/W_2(S)$ be a lift. A strong Frobenius lift is a $W_2(S)$-map $W_2(X) \to \tilde{X}$ which reduces to the identity map on $X$, via the natural embeddings $X \mono W_2(X)$ and $X \mono \tilde{X}.$
\end{definition}
\begin{remark}
Let $X/S$ be a family admitting a lift $\widetilde{X}/W_2(S)$. Using the Frobenius on $W_2(S)$ we define
\[
\widetilde{X}' := \widetilde{X}\times_{W_2(S),F_{W_2(S)}} W_2(S).
\]
One can then consider lifts of the relative Frobenius $F_{X/S}\colon X\to X'$ to a $W_2(S)$-morphism 
$\widetilde{F}\colon \widetilde{X}\to \widetilde{X}'$.

We warn the reader that this is weaker than a strong Frobenius lift. Namely, for a local section 
$a$ of $\cO_{\widetilde{X}}$, an arbitrary Frobenius lift satisfies 
$\widetilde{F}(a)=a^p+i_a$ with $i_a\in\Ker(\cO_{\widetilde{X}}\to\cO_X)$, whereas a strong Frobenius lift requires 
$\widetilde{F}(a)=a^p+p\delta_a$ with $\delta_a\in\cO_{X'}$.
\end{remark}

\begin{lemma}\label{affine-splitting-strong-frob}
    Let $R \to A$ be a map of commutative rings. Let $W_2(R) \to \tilde{A}$ a be flat lift of $R \to A$. Assume we are given a strong Frobenius lift, i.e., a map $f: \tilde{A} \to W_2(A)$ such that $\tilde{A} \xrightarrow{f} W_2(A) \to A$ is the natural map $\tilde{A} \to A.$ Then $f$ gives rise to an isomorphism $\dR_{A/R} \simeq \bigoplus_k \Lambda^k L_{A'/R}[-k]$ of derived algebras.
\end{lemma}
\begin{proof}
    The proof is essentially given in \cite[Proposition 3.17]{bhatt2012padic}. Choose a free simplicial $W_2(A)$-algebra resolution $\tilde{P_{\bullet}} \to \tilde{B}$, then we get $\tilde{h}: \tilde{P}_{\bullet} \to W_2(P_\bullet)$ and therefore $\tilde{F}: \tilde{P}_{\bullet} \to \tilde{P}_{\bullet}$ as well as its linearized version $\tilde{F}_{\tilde{P}/W_2(A)}: \tilde{P'}_{\bullet} \to \tilde{P}_{\bullet}$. Note $\tilde{F}(x)$ can be canonically written as $x^p+p\delta_x$ for some $\delta_x \in P_{\bullet}.$\footnote{It is wrong for an arbitrary Frobenius lift to $W_2(R).$ It is crucial to have a map $\tilde{P}_{\bullet} \to W_2(P_{\bullet}).$} This gives rise to a well-defined map $$\frac{\Omega^1(\tilde{F})}{p}: \Omega^1_{\tilde{P'}_{\bullet}/W_2(A)} \to \Omega^1_{\tilde{P}_\bullet/W_2(A)}.$$
    Analogously, one obtains $$\frac{\Omega^k(\tilde{F})}{p}: \Omega^k_{\tilde{P'}_{\bullet}/W_2(A)} \to \Omega^k_{\tilde{P}_{\bullet}/W_2(A)}.$$
    These maps assemble into a map of bicomplexes $\Omega^1_{P_{\bullet}'/A}[-1] \to \Omega^\bullet_{P_\bullet/A}$. This totalizes to a map $L_{A'/R}[-1] \to \dR_{A/R}$ splitting the first term of the conjugate filtration. The standard argument shows it gives rise to the desired isomorphism.
\end{proof}

\begin{corollary}\label{affine-section-of-dR-relative}
     Let $X \to S$ be a quasi-syntomic map of $\bF_p$-schemes. Let $\tilde{X}/W_2(S)$ be a flat lift of $X$ and let $f: W_2(X) \to \tilde{X}$ be a strong Frobenius lift. Then $f$ gives rise to a splitting $s_f: X' \to (X/S)^{dR}$ of the de Rham torsor.
\end{corollary}

Let $s: X' \to (X/S)^{dR}$ be a section of the de Rham torsor as in \Cref{affine-section-of-dR-relative}. Although when $S = \Spec(\bF_p)$ one obtains a canonical homotopy $s \circ F_{X/S} \simeq \pi_{X/S}$, it is not clear whether the same holds for arbitrary $S$. In the next remark, we observe that in this particular case it does hold.

\begin{remark}
    Let $X/S$ be a smooth map over $\bF_p$. Let $C: (\Omega^{\bullet}_{X'/S}, 0) \simeq \dR_{X/S}$ be the chain quasi-isomorphism from \Cref{affine-splitting-strong-frob}. The splitting $s_f: X' \to (X/S)^{dR}$ is given by the composition $$\dR_{X/S} \xrightarrow{C^{-1}} (\Omega^\bullet_{X'/S}, 0) \to \cO_{X'} $$
    where the second map is the projection onto the first component. Note one has a canonical homotopy $\pi_{X/S} \simeq s \circ F_{X/S}.$ Indeed, it amounts to observe that the chain map $C$ is given by the Frobenius $\cO_{X'} \to F_{X/S, *}\cO_X$ in degree $0$.
\end{remark}

Generalizing \Cref{delta-absolute-splitting}, we obtain a relative version.
\begin{lemma}\label{delta-splitting}
    Let $X/S$ be a smooth family. Let $\tilde{X}/W(S)$ be a $\delta$-lift. It gives rise to a section $X' \to (X/S)^{dR}.$
\end{lemma}
\begin{proof}
    Let us define a map $s: \tilde{X'}:=\tilde{X} \times_{W(S), F_{W(S)}} W(S) \to \tilde{X}$. Namely, we define
$$\tilde{X}'(R)=\Map_{W(S)}(\Spec(F_*R), \tilde{X}) \to \Map_{W(S)}(\Spec(\bG_a^{dR}(R), \tilde{X}) = (\tilde{X}/W(S))^{dR}(R)$$
as follows. Let $x: \Spec(F_*R) \to \tilde{X}$. Then $\delta$-structure on $\tilde{X}$ gives rise to a map $\Spec(F_*W(R)) \to \tilde{X}$ and the composition $\Spec(W(F_*R)/p) \mono \Spec(W(F_*R)) \to \tilde{X}$ uniquely factors through some map $\Spec(W(F_*R)/p) \to X$ since $\Spec(W(F_*(R))/p)$ has a natural $S$-scheme structure. 
\end{proof}
\begin{remark}
    In \Cref{delta-splitting} we observed that a $\delta$-lift $\tilde{X}/W(S)$ gives rise to a map $s: \tilde{X}' \to (\tilde{X}/W(S))^{dR}$. Let us note that $s \circ \tilde{F} \simeq \pi_{\tilde{X}/W(S)}$ where $\tilde{F}: \tilde{X} \to \tilde{X'}$ is the Frobenius defined by the $\delta$-structure. Indeed, if $S=\Spec(R)$ and $\tilde{X}=\Spec(\tilde{A}),$ then this point $\pi_{\tilde{X}/W(S)} \in (\tilde{X}/W(S))^{dR}$ is given by a map $\tilde{A} \to \bG_a^{dR}(\tilde{A})$. Note this map is homotopic to the composition
\begin{equation}
\label{first-comp}
   \tilde{A} \xrightarrow{w_{\tilde{A}}} W(\tilde{A}) \xrightarrow{F_{W(\tilde{A})}} W(\tilde{A}) \to  W(\tilde{A})/p.
\end{equation}
    Now $s: \tilde{X'} \to (\tilde{X}/W(S))^{dR}$ sends a point $\tilde{F}: \tilde{A} \to \tilde{A}$ to the composition $$\tilde{A} \xrightarrow{\tilde{F}} \tilde{A}  \xrightarrow{w_{\tilde{A}}} W(\tilde{A}) \to W(\tilde{A})/p$$
    which is naturally isomorphic to \eqref{first-comp}. Indeed, the naturality of $w_{\tilde{A}}$ gives $w_{\tilde{A}} \circ \tilde{F} = W(\tilde{F}) \circ w_{\tilde{A}}$. It remains to check that $W(\tilde{F}) \circ w_{\tilde{A}} = F_{W(\tilde{A})} \circ w_{\tilde{A}}$ and it can be seen that the left-hand side is equal to $w_{\tilde{A}} \circ F_{W(\tilde{A})}$ which can be checked on ghost coordinates. Both maps are maps of $\delta$ rings and thus we can check that they are the same after post-composing with $W(\tilde{A}) \to \tilde{A}.$ It is easy to see that the resulting maps are equal to $\tilde{F}$.
\end{remark}
\begin{lemma}\label{lm:all-splitting-same}
    Let $\tilde{X}/W(S)$ be a $\delta$-lift of smooth $X/S$ with $\tilde{F}: \tilde{X} \to \tilde{X'}$ the corresponding Frobenius lift. Let $s_1, s_2$ be the maps $\tilde{X}' \to (\tilde{X}/W(S))^{dR}$ that reduce to sections of the de Rham stack after the base change along $S \to W(S).$ Assume there exist isomorphisms $h_i: s_i \circ \tilde{F} \simeq \pi_{\tilde{X}/W(S)}$ for $i=1, 2$. If $S$ is reduced, then there exists $s_1 \simeq s_2.$
\end{lemma}
\begin{proof}
    Note $h_i$ realizes $s_i$ as a map of simplicial objects $$\Cech(\tilde{X} \to \tilde{X'}) \to \Cech(\tilde{X} \to (\tilde{X}/W(S))^{dR})$$
    where the map on $1$-simplices is $D^{\phi}(\tilde{X} \times_{W(S)} \tilde{X}) \to D_{\Delta}(\tilde{X} \times_{W(S)} \tilde{X})$ splitting the natural projection. Such a map is equivalent to endowing $D^{\phi}(\tilde{X} \times_{W(S)} \tilde{X})$ with a PD-structure on the ideal of the diagonal. Since it is flat over $W(S)$, it is also flat over $\bZ_p$ because $S$ is reduced. Thus, this PD-structure is unique.
\end{proof}

\begin{lemma}\label{lm:affine-over-witt-splitting}
    Let $\tilde{X}/W(S)$ be a $\delta$-scheme. There is a natural map  $\tilde{X}' \to (\tilde{X}/W(S))^{dR} $ deforming the section from \Cref{affine-section-of-dR-relative}.
\end{lemma}
\begin{proof}
    \textit{Step 1.} Consider the category of maps $(Y, W(T), \pi_{Y, T})$ where $\pi_{Y, T}: Y \to W(T)$ is a flat map over the formal scheme $W(T),$ i.e., $\pi_{Y, T}$ is equivalent to a set of compatible flat maps $\pi^n_{Y, T}: Y_n \to W_n(T).$  Formally, it is defined as the fiber product of  $$\Fun^{fl}(\Delta^1, \FSch_{\bZ_p}) \xrightarrow{t} \FSch_{\bZ_p} \xleftarrow{W} \Sch_{\bF_p} $$
    where $\Fun^{fl}(\Delta^1, \FSch)$ is the full subcategory of $\Fun(\Delta^1, \FSch)$ consisting of objects $X \to Y$ where the map is flat. Denote this category to be $\cC.$ The relative de Rham stack restricts to a functor $(-/-)^{dR}: \cC \to \dSt_{\bZ_p}.$ Also, there is a functor $T: \cC \to \Sch_{\bZ_p}$ sending $Y \to W(T)$ to $Y' := Y \times_{W(T), F_{W(T)}} W(T).$ We want to construct a natural transformation $T \to (-/-)^{dR}.$ Both functors are Zariski sheaves. Let $\cC_0$ be a full subcategory of $\cC$ where $T$ and $Y$ are affine formal schemes. Now, if $F, G$ are Zariski sheaves on $\cC$, then any natural transformations of their restriction to $\cC_0$ uniquely extends to a natural transformation between $F$ and $G.$ Thus, we restrict ourselves to $\cC_0.$

    \textit{Step 2.} Assume $Y \to W(T)$ is an object of $\cC_0$ with $T=\Spec(R)$ and $Y=\Spec(A).$ We construct $Y' \to (Y/W(T))^{dR}$ in two steps. First, we use the Deligne-Illusie chain map $(\Omega^\bullet_{A/R}, pd) \to (\Omega_{A/R}^\bullet, d)$ which is in degree $n$ given as follows: the map $F^*: \Omega^n_{A'/W(R)} \to \Omega^n_{A/W(R)}$ is divisible by $p^n$ and thus defines $$F^*/p^n: \Omega^n_{A'/W(R)} \to \Omega^n_{A/W(R)} $$
    and by construction it gives rise to a map of complexes $$C_{A/W(R)}: (\Omega^\bullet_{A/W(R)}, pd) \to (\Omega_{A/W(R)}^\bullet, d)$$
    which is a quasi-isomorphism. The chain map $C_{A/W(R)}$ is natural in pairs $W(R) \to A$ with $A$ being a $\delta-W(R)$-algebra. The second step is to consider the composition $$F^0 \circ C_{A/W(R)}^{-1}:  (\Omega_{A/W(R)}^\bullet, d) \to (\Omega^\bullet_{A'/W(R)}, pd) \to A' $$
    which is natural as a composition of two natural transformations. This gives rise to the desired section $X' \to (X/S)^{dR}.$
\end{proof}
\begin{remark}
    One can observe that in the second step of the proof \Cref{lm:affine-over-witt-splitting} we produce, for a flat $\delta-W(S)$-scheme $\tilde{X}$, an isomorphism of stacks $(\tilde{X}/W(S))^{dR} \simeq (\tilde{X}/W(S))^{dR, p}$ where the latter is defined as the fiber of the Hodge-filtered de Rham stack $$(\tilde{X}/W(S))^{dR, +} \to \bA^1/\bG_m$$ at $t=p.$ To complete the argument, one has to check that for affine $S$ and $\tilde{X}$ the stack $(\tilde{X}/W(S))^{dR, p}$ is affine and represented by $(\Omega^\bullet_{\tilde{X}/W(S)}, pd)$. Since we will not need it, we do not include the details. We note however that the category of quasi-coherent sheaves on $(\tilde{X}/W(S))^{dR, p}$ is identified with the category of (topologically quasi-nilpotent) $p$-connections. In particular, when $S=\Spec(\bZ_p),$ this isomorphism of stacks is closely related to the work of Shiho on the deformation of the local Ogus-Vologodsky equivalence to $\bZ/p^n$, see \cite{MR3408075}.
    
\end{remark}

\begin{corollary}
    Let $X/S$ be a smooth family and $\tilde{X}/W(S)$ its $\delta$-lift. The two maps $s_1, s_2: \tilde{X'} \to (\tilde{X}/W(S))^{dR}$ obtained from \Cref{delta-splitting} and \Cref{lm:affine-over-witt-splitting} are isomorphic as points of $(\tilde{X}/W(S))^{dR}(\tilde{X'})$. Indeed, they are both equal to $\pi_{\tilde{X}/W(S)}$ after pre-composing with $\tilde{F}$ and \Cref{lm:all-splitting-same} concludes the statement. 
\end{corollary}

\subsection{Quasi-coherent sheaves on de Rham stack}
In this subsection, we recall that for smooth $X/\bF_p$ the category $\cD_{qc}(X^{dR})$ is identified with the $\infty$-category $\Crys(X/\bF_p)$ of crystals of quasi-coherent complexes on the $p$-torsion crystalline site $(X/\bF_p)_{\mathrm{crys}}$. For example, this follows from the more general result \cite[Theorem 6.5]{bhatt2022prismatization} by taking the crystalline prism. For the reader’s convenience, we give a (much less general) more elementary proof, which works in the case of a de Rham stack. We also explain that this equivalence holds for an lci $X/S$.
\begin{remark}
    Let $X \to S$ be a smooth map. Then by \Cref{covering-for-smooth} we get an fpqc covering $\pi_{X/S}: X \to (X/S)^{dR}$. Moreover, by \Cref{dR-commutes-colimits} one identifies $n$-th term of the Cech nerve $X^{\times_{(X/S)^{dR}} n}$ with $(X/X^{\times_{S} n})^{dR}.$ Moreover, by \Cref{lm:regular-affine-stack} the stack $(X/X^{\times_S n})^{dR}$ is represented by the PD-envelope $D_{I_{\Delta}} (X^{\times_S n})$ of the diagonal map $X \to X^{\times_S n}$. In particular, one gets $D_{qc}((X/S)^{dR}) \simeq \operatorname{Crys}(X/S)$ since $\Cech(X \to (X/S)^{dR})$ is identified with the simplicial scheme $D_{I_{\Delta}}^{\bullet/S}.$
\end{remark}

\begin{theorem}\label{sheaves-on-dR-lci}
    Let $X \to S$ be an lci map. Then $D_{qc}((X/S)^{dR}) \simeq \operatorname{Crys}(X/S)$.
\end{theorem}
\begin{proof}
    Recall from \cite[Th. 6.6]{44026ff5-d9db-3e28-a22f-70968eda7ea4} that for any closed immersion $Y \to Z$ with smooth $Z/S$, the category $\operatorname{Crys}(Y/S)$ is identified with the category of modules over $D_{Y}(Z)$ with an HPD-stratification compatible with the natural stratification on $D_{Y}(Z).$ Moreover, the latter category is identified with the category of sheaves on the Cech nerve of the covering $D_{Y}(Z) \to Y.$ Both $D_{qc}((-/-)^{dR})$ and $\operatorname{Crys}(-/-)$ are functors $\operatorname{Fun}(\Delta^1, \operatorname{Sch}^{op}) \to \Pr^{L}_{st}$. Moreover, they both satisfy Zariski descent and it is enough to prove that their restrictions to $\operatorname{Fun}(\Delta^1, \operatorname{Ring})$ are isomorphic. Assuming $X = \Spec(B)$ and $S = \Spec (A)$ we can use the canonical projection $A[x_b  | b \in B] := P \to B$ to identify $\Crys(B/A) = D_{qc}(\Cech(D \to B))$ where $D := D_{P} (J)$ for $J = \ker(P \to B)$ and the Cech nerve is computed in $(B/A)_{crys}$, i.e. its $n$-term is $D(n) := D_{P \otimes_A \cdots \otimes_{A} P} (J(n))$ for $J(n) := \ker(P \otimes_A \cdots \otimes_A P \to B)$ and both tensor products contain $n+1$ term. Note that $A[B]$ is functorial in $B$, therefore $\Crys(B/A)$ is functorially identified with $D_{qc}(\Cech(D \to B)).$ 
    Let us now check that $D_{qc}((B/A)^{dR})$ is also functorially identified with $D_{qc}(\Cech(D \to B)).$ We have a commutative diagram
\[\begin{tikzcd}
	{D \simeq (B/P)^{dR}} &&& {(B/A)^{dR}} \\
	\\
	\\
	{\Spec (P)} &&& {(P/A)^{dR}}
	\arrow[from=1-1, to=1-4]
	\arrow[from=1-1, to=4-1]
	\arrow[from=1-4, to=4-4]
	\arrow[from=4-1, to=4-4].
\end{tikzcd}\]
Note that the bottom horizontal arrow is a faithfully flat cover by \Cref{covering-for-smooth} since $A \to P$ is smooth, therefore $(B/P)^{dR} \to (B/A)^{dR}$ is a faithfully flat covering. Identify $(B/P)^{dR} \simeq D$ by \Cref{lm:regular-affine-stack}. Moreover, $$(B/P)^{dR} \times_{(B/A)^{dR}} \cdots \times_{(B/A)^{dR}} (B/P)^{dR} \simeq (B \otimes_B \cdots \otimes_B B/P \otimes_A \cdots \otimes_A P)^{dR} \simeq (B/P \otimes_A \cdots \otimes_A P)^{dR}$$
which is isomorphic to $D(n)$ by \Cref{dR-commutes-colimits} and \Cref{lm:regular-affine-stack}; all tensor products contain $n+1$ terms. Using this covering we identify $D_{qc}((B/A)^{dR})$ with $D_{qc}(\Cech(D \to B))$ and this identification is functorial. This completes the proof.
\end{proof}
Let us now recall that for smooth $X/S$ the category $\cD_{qc}((X/S)^{dR})$ is identified with certain $D$-modules. Denote
\[
\hat D_{X/S} := F_{X/S, *} D_{X/S} \otimes_{ST_{X'/S}} \hat ST_{X'/S} \in \QCoh(X').
\]
Let $\Mod^{\mathrm{ln}}(\hat D_{X/S})$ be the full subcategory of $\Mod_{X'}(\hat D_{X/S})$ consisting of those left $\hat D_{X/S}$-modules such that every local section is killed by $S^{\ge n}T_{X'/S} \subset \hat D_{X/S}$ for some $n>0$. Denote by $\cD^{\mathrm{ln}}(\hat D_{X/S})$ the full subcategory of $\cD(\hat D_{X/S})$ consisting of complexes $E$ such that $\cH^i(E) \in \Mod^{\mathrm{ln}}(\hat D_{X/S})$.
\begin{lemma}\label{lm:mod-xdr=d-mod}
    Let $X/S$ be smooth. Then $\cD^{ln}(\hat D_{X/S}) \simeq \cD_{qc}((X/S)^{dR})$.
\end{lemma}
\begin{proof}
Note the left $F_{X/S, *}D_{X/S}$-module $F_{X/S, *}\cO_X$ is compact. Moreover, the minimal $\cD_{qc}(X')$-linear stable subcategory of $\cD(D_{X/S})$ containing $F_{X/S, *}\cO_X$ coincides with $\cD^{ln}(D_{X/S}).$ Thus, $\cD^{ln}(D_{X/S})$  is identified with the category of $\End_{D_{X/S}}(F_{X/S, *}\cO_X)$-modules in $\cD_{qc}(X')$. Moreover, the Spencer resolution identifies this endomorphism algebra with $F_{X/S, *}\dR_{X/S}.$   Recall from \Cref{dR-is-affine-relative} that $\nu_{X/S}: (X/S)^{dR} \to X'$ is affine, thus $\cD_{qc}((X/S)^{dR})$ is identified with the left completion of $\Mod_{\cD_{qc}(X')}(F_{X/S, *} \dR_{X/S})$ by \cite[Theorem 1.4]{derivedstacks}. It completes the proof since this category is already left-complete.
\end{proof}

\section{Lifts and Cartier transform}
This section is devoted to a proof of \cite[Theorem 2.8]{MR2373230} as well as some variants of it. Note this theorem only uses a flat lift $\tilde{X'}/\tilde{S}$ of $X'/S$ which is not enough to split $(X/S)^{dR} \to X'$. However, one can define a modification of the de Rham gerbe denoted by $(X/S)^{dR, \gamma}$ which is split by the data of a flat lift of $X'/S$ and which can be used to deduce \cite[Theorem 2.8]{MR2373230}. The main result of this section is \Cref{rem:splittings-push=lifts-to-witt} where the gerbe of splittings of $(X/S)^{dR, \gamma}$ is identified with the gerbe of lifts of $X$ to $W_2(S)$. 
\newline
\begin{definition}\label{def:Gadrgamma}
    Denote $\bG_a^{dR, \gamma}$ to be the $\bG_a$-algebra stack $F_*\bG_a \otimes^L_{\bZ_p} \bF_p$ with $\bG_a$-algebra structure coming from a map of quasi-ideals 
\[\begin{tikzcd}
	{F_*W} &&& W \\
	{F_*\bG_a} &&& {F_*\bG_a}
	\arrow["V", from=1-1, to=1-4]
	\arrow["R"', from=1-1, to=2-1]
	\arrow["{F \circ R}", from=1-4, to=2-4]
	\arrow["0", from=2-1, to=2-4].
\end{tikzcd}\]
We denote $\pi_{\bG_a}^{\gamma}: \bG_a \to \bG_a^{dR, \gamma}.$
\end{definition}
\begin{remark}
    If $R$ is an $\bF_p$-algebra, then $\bG_a^{dR, \gamma}(R)$ is also an $\bF_p$-algebra since $\bG_a^{dR, \gamma}$ is a $\bG_a$-algebra stack. Explicitly, $\bG_a^{dR, \gamma}(R) = R \otimes^L_{\bZ_p} \bF_p$ has two natural $\bF_p$-algebra structures and the $\bF_p$-algebra structure coming from the second factor is the one provided by \Cref{def:Gadrgamma}.
\end{remark}
\begin{remark}\label{rem-two-maps}
    It is easy to see $\pi_0(\bG_a^{dR, \gamma}) = F_*\bG_a$ and $\pi_1(\bG_a^{dR, \gamma}) = F_*\bG_a$ as sheaves of $\bG_a$-modules. In particular, we have a map $\nu^{\gamma}: \bG_a^{dR, \gamma} \to F_*\bG_a$. We warn the reader that the section of $\nu^\gamma$ given by a map of quasi-ideals 
\[\begin{tikzcd}
	0 &&& {F_*\bG_a} \\
	{F_*\bG_a} &&& {F_*\bG_a}
	\arrow[from=1-1, to=1-4]
	\arrow[from=1-1, to=2-1]
	\arrow[equals, from=1-4, to=2-4]
	\arrow["0", from=2-1, to=2-4]
\end{tikzcd}\]
is not $\bF_p$-linear but only is $\bZ/p^2$-linear as a simple diagram chasing shows; we refer to \Cref{rem:obvious-splittig-not-R-linear} for this diagram chasing in a more general setting. In particular, $\nu^{\gamma}$ admits a $\bZ/p^2$-linear section and $\bG_a^{dR, \gamma} \simeq \bG_a \oplus B\bG_a$ as a $\bZ/p^2$-algebra stack. 
\end{remark}

\begin{remark}\label{witt-vectors}
    Let $A/\bF_p$ be a commutative algebra. In \Cref{rem-two-maps} we contemplate two maps $c_A, \pi_a^{\gamma}: A \to \bG_a^{\gamma}(A)$ of $\bZ/p^2$-algebras. Let us note that their fiber product $A \times_{\bG_a^{\gamma}(A)} A$ is canonically identified with $W_2(A),$ see \cite[Remark 2.5]{MR4502597}.
\end{remark}

\begin{definition}
    For an $\bF_p$-scheme $X$ we use $X^{dR, \gamma}$ to denote the prestack on $\bF_p$-algebras obtained by the transmutation of an $\bF_p$-algebra stack $\bG_a^{dR, \gamma}$. For $X/S$ denote $(X/S)^{dR, \gamma} = X^{dR, \gamma} \times_{S^{dR, \gamma}} S$ where $\pi_{S}^{\gamma}: S \to S^{dR, \gamma}$ is the transmutation of $\pi^{\gamma}$ from \Cref{def:Gadrgamma}. We denote $\pi_{X/S}^{\gamma}: X \to (X/S)^{dR, \gamma}$ to be the map induced by $\pi_{X}^{\gamma}:X \to X^{dR, \gamma}.$
\end{definition}

\begin{remark}
    Let $S=\Spec(R), X=\Spec(A).$ Then the functor of points of the $R$-stack $(X/S)^{dR, \gamma}$ sends $x: R \to B$ to $(X/S)^{dR, \gamma}(B)=\Map_{R}(A, B \otimes_{\bZ} \bF_p)$ where the $R$-algebra structure on $B \otimes \bF_p$ is given by 
\[\begin{tikzcd}
	{F_*W(R)} &&& {W(R)} \\
	B &&& B
	\arrow["V", from=1-1, to=1-4]
	\arrow["{x \circ R}", from=1-1, to=2-1]
	\arrow["{F(x) \circ R}", from=1-4, to=2-4]
	\arrow["0", from=2-1, to=2-4].
\end{tikzcd}\]
Note it is a map of quasi-ideals, thus defines a map of rings $R \to B \otimes F_p.$ To specify this $R$-algebra structure on $B \otimes \bF_p,$ we denote it by $\bG_a^{dR, \gamma}(B).$
\end{remark}

\begin{remark}\label{rem:obvious-splittig-not-R-linear}
    For $R \to B$, recall the $R$-algebra stack $\bG_a^{dR, \gamma}(B)$. We warn the reader that the canonical map of rings $B \to \bG_a^{dR, \gamma}(B)$ is not $R$-linear but is $W_2(R)$-linear. This can be seen from the following diagram:
\[\begin{tikzcd}
	{F_*^2W(R)} &&& {W(R)} \\
	{F_*W(R)} &&& {W(R)} \\
	B &&& B \\
	0 &&& B
	\arrow["{V^2}", from=1-1, to=1-4]
	\arrow["V", from=1-1, to=2-1]
	\arrow["0"', curve={height=30pt}, from=1-1, to=3-1]
	\arrow[equals, from=1-4, to=2-4]
	\arrow["V", from=2-1, to=2-4]
	\arrow["{x \circ R}", from=2-1, to=3-1]
	\arrow["{F(x) \circ R}", from=2-4, to=3-4]
	\arrow["0", from=3-1, to=3-4]
	\arrow[from=4-1, to=3-1]
	\arrow[from=4-1, to=4-4]
	\arrow[equals, from=4-4, to=3-4].
\end{tikzcd}\]

\end{remark}
\begin{remark}
    If $X/S$ is quasi-syntomic, then $\nu_{X/S}: (X/S)^{dR} \to X'$ is an  fppf $L_{X'/S}^{\vee}[1] \otimes_{\bG_a} \bG_a^{\sharp}$-torsor and its pushout along the map induced by $\bG_a^{\sharp} \to \bG_a$ is identified with $(X/S)^{dR, \gamma}.$ In particular, we get a map $\nu_{X/S}^{\gamma}: (X/S)^{dR, \gamma} \to X'$ exhibiting the source as $L_{X'/S}^{\vee}[1]$-torsor over the target. It also follows from \Cref{torsor-ring-stack} via an argument as in \Cref{for-lci-dR-torsor}.
\end{remark}

\subsection{Partial Crystalline Site}
In this subsection, we introduce a partial crystalline site which plays the same role for $(X/S)^{dR, \gamma}$ as the crystalline site does for $(X/S)^{dR}$. We prove that for lci $X/S$ the category $\cD_{qc}((X/S)^{dR, \gamma})$ is equivalent to the category of partial crystals; see \Cref{cor:partial-crystals-x-dr-gamma}. We also interpret this category as (continuous) modules over a suitable algebra of differential operators; see \Cref{gamma-crystals-are-gamma-d-modules}.

\begin{definition}
    Given a pair $(A, I),$ a partial PD-structure on it is a map $\gamma_p: I \to A$ such that $\gamma_p(x) p! = x^p.$ We denote the category of partial PD-rings by $\operatorname{PDAlg}^{\gamma}$.
\end{definition}

\begin{remark}
    If $(A, I, \gamma_p)$ is a partial PD-ring and $pA = 0,$ then $x^p =0$ for any $x \in I.$
\end{remark}

\begin{definition}
    Given a scheme $X$ with a sheaf of ideals $\cI$ we say that $\gamma_p: \cI \to \cI$ is a partial PD-structure on $\cI$ if for every affine Zariski open $U = \Spec (A) \subset X$ the triple $(A, \cI(U), \gamma_p(U))$ is a partial PD-ring.
\end{definition}

\begin{definition}\label{def:partial-site}
    Let $(S, \cJ, \gamma_p)$ be a partial PD-scheme. Define a big partial crystalline site $\operatorname{CRIS^{\gamma}(X/((S, \cJ, \gamma_p)))}$ for any $X \to V(\cJ)$ on which $p$ acts in a locally nilpotent way. The underlying category consists of partial PD-schemes $(T, \cI, \delta_p)$ over $(S, \cJ, \gamma_p)$ endowed with a map $V(\cI) \to X.$
\end{definition}

\begin{definition}
    For $X$ as in \eqref{def:partial-site} define the category of partial crystals $\operatorname{Crys}^{\gamma}(X/((S, \cJ, \gamma_p)))$ by $\lim_{(T, I, \gamma_p) \in \cC_X} \cD(T)$ where $\cC_X$ is the category of affine objects in $\operatorname{CRIS^{\gamma}(X/((S, \cJ, \gamma_p)))},$ i.e. the category of pairs $(B, J) \in \operatorname{PDAlg}_{/S}^{\gamma}$ and a map $\Spec (B/J) \to X.$
\end{definition}

\begin{remark}
    When $X$ is a scheme over a perfect field $k$ of characteristic $p$, we note that $(W(k), p)$ has a canonical partial PD-structure given by $\gamma_p (p) = p!/p$ and we will use $\operatorname{Crys}^{\gamma}(X/W(k))$ and $\operatorname{CRIS^{\gamma}}(X/W(k))$ instead of $\operatorname{Crys}^{\gamma}(X/((W(k), (p), \operatorname{can}))$ and instead of $\operatorname{CRIS}^{\gamma}(X/((W(k), (p), \operatorname{can})))$ respectively.
\end{remark}
\begin{lemma}\label{pd-envelope-exists}
    Let $(A, I, \gamma_p)$ be a partial PD-ring. Let $(A, I) \to (B, J)$ be a map of pairs, then there exists $(A, I, \gamma_p) \to (D, \bar{J}, \bar{\gamma}_p)$ satisfying $\Hom_{(A, I, \gamma_p)}((D, \bar{J}, \bar{\gamma}_p), (C, K, \delta_p)) = \Hom_{(A, I)}((B, J), (C, K)).$
\end{lemma}
\begin{proof}
    The proof is the same as in \cite[\href{https://stacks.math.columbia.edu/tag/07H7}{Tag 07H7}]{Stacks}.
\end{proof}
\begin{definition}
    Let $(A, I, \gamma_p)$ be a partial PD-ring and $(A, I) \to (B, J)$ a map of pairs. The partial PD-algebra over $(A, I, \gamma_p)$ provided by \Cref{pd-envelope-exists} will be denoted by $D^{\gamma}_{B}(J)$ and referred to as the partial PD-envelope of $J$ (or $B/J$) in $B.$
\end{definition}

\begin{corollary}\label{cor:partial_crystals_cech_pd}
    Let $X \to Y$ be a closed embedding of $S$-schemes with $Y/S$ being smooth. Then $\operatorname{Crys}^{\gamma}(X/S)$ identifies with modules over a cosimplicial scheme $\Cech(D^{\gamma}_X (Y) \to X).$
\end{corollary}
\begin{proof}
    Indeed, $D^{\gamma}_X (Y)$ is weakly final in $\operatorname{CRIS^{\gamma}(X/((S, \cJ, \gamma_p)))}$. After enlarging it we can assume it has a final object $*$ by \cite[TAG 03CI]{Stacks}. Thus, since $D_{X}^{\gamma} (Y) \to *$ is a surjective map of presheaves, it can be used to compute $\operatorname{Crys}^{\gamma}(X/S)$ which gives the result.
\end{proof}
\begin{lemma}\label{cech-nerve-comp-of-partial}
    Given a closed immersion $Y \to X$, the stack $(Y/X)^{dR, \gamma} $ is represented by $D_{Y}^{\gamma} (X)$ the partial PD-envelope of $Y$ in $X$.
\end{lemma}
\begin{proof}
    By definition $(Y/X)^{dR, \gamma}$ is the quotient of $(Y/X)^{dR} \times_{Y^{\phi}} \bN_{Y^{\phi}} (X) \simeq D_{Y} (X) \times_{Y^{\phi}} \bN_{Y^{\phi}} (X)$ by the anti-diagonal action of $\bN_{Y^{\phi}} (X)^{\sharp}$. Consider the map $D_{Y} (X) \to D_{Y}^{\gamma} (X),$ it is enough to observe that it is equivariant for $\bN_{Y^{\phi}}(X)^{\sharp} \to \bN_{Y^{\phi}} (X)$.
\end{proof}
\begin{corollary}\label{cor:dqc-dr-gamma-cech}
    For smooth $X \to S$ the Cech nerve of $X \to (X/S)^{dR, \gamma}$ is represented by a simplicial scheme $D_{I_n} (X^{\bullet/S})$ where $I_n$ is the ideal of the diagonal $X \to X^{\times n/S}.$ In particular, the category $\cD_{qc} ((X/S)^{dR, \gamma})$ is identified with the $\infty$-category of crystals in quasi-coherent complexes on $(X/S)_{crys}^{\gamma}$.
\end{corollary}
\begin{corollary}\label{cor:partial-crystals-x-dr-gamma}
    Combining \Cref{cor:dqc-dr-gamma-cech} and \Cref{cech-nerve-comp-of-partial} one observes that for an lci morphism $X \to S$ we obtain the equivalence $\cD_{qc}((X/S)^{dR, \gamma}) \simeq \operatorname{Crys}^{\gamma}(X/S)$ by the same argument as in \Cref{sheaves-on-dR-lci}.
\end{corollary}

    Recall from \cite[\S 2.3]{MR2373230} the category $\MIC^{\cdot}_{\gamma}(X/S)$ which is defined as the full subcategory of the abelian category of $F_{X/S, *}D_{X/S} \otimes_{S T_{X'/S}} \hat \Gamma T_{X'/S}$-modules, consisting of locally nilpotent objects.


\begin{theorem}\label{gamma-crystals-are-gamma-d-modules}
    Let $X$ be a smooth scheme over $S$. Then $\MIC_{\gamma}^{\cdot}(X/S)$ is equivalent to the category of partial crystals in quasi-coherent sheaves on $X$ relative to $S.$
\end{theorem}
\begin{proof}
  
Consider an fppf-covering $$C_1 := (X/S)^{dR} \times_{X'} B_{X'}(T_{X'/S}) \to (X/S)^{dR, \gamma} =:C_0.$$ A quasi-coherent sheaf on $(X/S)^{dR, \gamma}$ amounts to $E \in \QCoh(B_{(X/S)^{dR}}(F_{X/S}^*T_{X'/S}))$ with $p_1^* E \simeq p_2^*E$ satisfying the cocycle condition, where $p_1, p_2: C_1 \times_{C_0} C_1 \rightrightarrows C_1$ are the projections. By Cartier duality $E$ amounts to a quasicoherent sheaf $\cE \in \QCoh((X/S)^{dR})$ together with a continuous\footnote{With respect to the PD-filtration on the source.} horizontal map $\psi^{\gamma}: \Gamma_{\cO_X} (F_{X/S}^* T_{X'/S}) \to \cE n d(\cE)$, i.e., every local section of $\cE$ is annihilated by $\Gamma^{n}_{\cO_{X}}(F_{X/S}^*T_{X'/S})$ for some $n.$ Note 
\[\begin{tikzcd}
	{C_1 \times_{C_0} C_1 } & {B_{X'}T_{X'/S} \times_{X'} B_{X'} T_{X'/S}^{\sharp} \times_{X'} (X/S)^{dR}} && {B_{X'}(T_{X'/S}) \times_{X'}(X/S)^{dR}}
	\arrow["\mathrel{\simeq}", from=1-1, to=1-2]
	\arrow["{(id, a)}"', shift right, from=1-2, to=1-4]
	\arrow["{(a', \id)}", shift left, from=1-2, to=1-4]
\end{tikzcd}\]
where $a$ is the action map $B_{X'}(T_{X'/S}^{\sharp}) \times_{X'} (X/S)^{dR} \to (X/S)^{dR}$ and similarly $a'$ is the action of $B_{X'}(T_{X'/S}^{\sharp})$ on $B_{X'}(T_{X'/S}).$ Moreover, these maps $(a', \id)$ and $(\id, a)$ correspond to the projections $p_1, p_2: C_1 \times_{C_0} C_1 \to C_1$ under this identification. Note the category of quasi-coherent sheaves on $C_1 \times_{C_0} C_1$ is the category $(\cM, \xi^{\gamma}, \xi)$ where $\cM$ is a crystal, $\xi^{\gamma}: \Gamma_{\cO_X}(F_{X/S}^*T_{X'/S}) \to \cE nd(\cM)$ and $\xi: S_{\cO_X}(F_{X/S}^*T_{X'/S}) \to \cE nd(\cM) $ are horizontal maps. Under this identification, the object $(a', \id)^*(\cE, \psi^{\gamma})$ corresponds to $(\cE, \psi^{\gamma}, \psi)$ where $$\psi: S_{\cO_X}(F_{X/S}^*T_{X'/S}) \to \Gamma_{\cO_X}(F_{X/S}^*T_{X'/S}) \xrightarrow{\psi^{\gamma}} \cE nd(\cM).$$
Similarly, the object $(\id, a)^*(\cE, \psi^{\gamma}) = (\cE, \psi^{\gamma}, \psi_{\cE})$ where $\psi_{\cE}$ is the $p$-curvature of $\cM.$ In particular, supplying an isomorphism $\alpha: p_1^*(\cE, \psi^{\gamma}) \simeq p_2^*(\cE, \psi^{\gamma})$ amounts to a $D$-module automorphism $\alpha_0$ of $\cE$ commuting with two actions of $S_{\cO_{X}}(F_{X/S}^*T_{X'/S}).$ In particular, the two actions agree $\psi=\psi_{\cE}$ and $\alpha_0$ gives rise to an isomorphism of this descent datum to the one with the identity isomorphism between two pullbacks, which completes the proof. 
\end{proof}

\subsection{Gerbe of liftings}
Recall \cite[Theorem 2.8]{MR2373230} uses a lift of $X'/S$ in order to construct the global Cartier transform. Let us first show that a lift of $X'$ to $\bZ/p^2$ is enough.
\begin{corollary}\label{strengh-of-OV-only-lift-X'}
    For $X/S$ any flat lift of $X'$ to $\bZ/p^2$ produces a section of $(X/S)^{dR, \gamma} \to X'.$ 
\end{corollary}
\begin{proof}
By definition, $X'^{dR, \gamma} = X^{dR, \gamma} \times_{S^{dR, \gamma}, F_{S^{dR, \gamma}}} S^{dR, \gamma}$ and by \Cref{Frob-factors-1-truncated} we get a factorization of $F_{S^{dR, \gamma}}: S^{dR, \gamma} \to S \xrightarrow{F_{S^{dR, \gamma}}'} S^{dR, \gamma}$ which gives $X'^{dR, \gamma} \simeq (X/S)^{dR, \gamma} \times_{S} S^{dR, \gamma}.$ In particular, we obtain a map $q: X^{dR, \gamma} \to (X/S)^{dR, \gamma}$. Moreover, by the construction this map is over $X'.$ In particular, since a lift of $X'$ is equivalent to giving a section $s$ of $X'^{dR, \gamma} \to X'$, the map $q \circ s$ defines a section of $(X/S)^{dR, \gamma}$.
\end{proof}
\begin{remark}
    The proof of \Cref{strengh-of-OV-only-lift-X'} shows that there is a map of groupoids $$\{\text{Flat lifts of $X'$ to $\bZ/p^2$ }\} \to \{\text{Splittings of $(X/S)^{dR, \gamma}$}\}.$$
Let us explain how to see this map using elementary methods.\footnote{If $X=\Spec(B), S=\Spec(A)$ are smooth, then $f^*T_S \to T_X$ sends an $A$-derivation $\partial$ to a derivation of $B \otimes_A F_*A$ defined by $\partial(b \otimes a) = b \otimes \partial a$. This gives the splitting $T_{X'} = T_{X'/S} \oplus f^*T_S.$}     The transitivity triangle for $f': X'\to S \to \Spec(\bF_p)$ gives $f'^*L_S \to L_{X'} \to L_{X/S}$. Let us check that the connecting map $o: L_{X'/S} \to f'^*L_{S}[1]$ is canonically homotopic to $0.$ Note $L_{X'/S}=\nu^* L_{X/S}$ and by adjunction $o$ corresponds to a map $L_{X/S} \to \nu_*f'^*L_{X/S}.$ Since $\nu_*f'^*L_{S}=f^*F_{S, *}L_S$, where $\nu: X' \to X$ is the canonical map, we learn that $o$ corresponds to $L_{X/S} \to f^*F_{S, *}L_S[1].$ By functoriality this map factors as $L_{X/S} \to f^*L_{S}[1] \to f^*F_{S, *}L_S[1],$ where the first map is the connecting homomorphism in the transitivity triangle for $f: X \to S$ and the second map is obtained by applying $f^*$ to $dF_S[1]: L_S[1] \to F_{S, *}L_S[1]$, which is canonically $0.$ Therefore, we obtain a section of $L_{X'} \to L_{X'/S}.$ 
\end{remark}

\begin{remark}\label{rem:lift-to-witt-is-stronger}
    Given a flat lift $\tilde{X}/W_2(S)$ of $X/S$, any flat lift of $S$ provides a splitting of $(X/S)^{dR, \gamma}$. Indeed, for $\tilde{S}/\bZ/p^2$ there is a natural map $\tilde{S} \to W_2(S)$ and pullback $\tilde{X} \times_{W_2(S)} \tilde{S}$ provides a flat lift of $X'/S.$
\end{remark}
Let $S/\bF_p$ be a scheme. Recall the closed immersion $S \to W_2(S)$ is a square-zero deformation with an ideal $F_*\cO_{S}.$ Therefore, for $f: X\to S$ the groupoid of flat lifts to $W_2(S)$ is a quasi-torsor for $\Map_{\cO_{X}}(L_{X/S}, f^*F_{S, *}\cO_{S}[1]) =\Map_{\cO_{X}}(L_{X/S}, \nu_*\cO_{X'}[1])$ where $\nu: X'=X \times_{S, F_S} S \to X$ is the projection map. This space is $\Map_{\cO_{X'}}(L_{X'/S}, \cO_{X'}[1])$ by adjunction and base change for cotangent complex. In particular, the obstruction class $\ob_{X/W_2(S)}$ of deforming $X$ to $W_2(S)$ defines a point in $\Map_{\cO_{X'}}(L_{X'/S}, \cO_{X'}[2]).$
\begin{remark}
    Let $X/S$ be lci. Consider the Zariski sheaf on $X$ that assigns to $U \subset X$ a groupoid of lifts $\tilde{U} \to W_2(S)$ of $U/S$ to $W_2(S).$ Recall the obstruction $\ob_{U/W_2(S)}: L_{U/S} \to f^*L_{S/W_2(S)}[1] \simeq \nu_* \cO_{U'}[2]$ where $\nu: U' \to U$ is the canonical map. Assume $U$ is affine, then $\ob_{U/W_2(S)}=0$. Indeed, $U'$ is affine and $\cH^i(T_{U'/S})=0$ for $i>1$, thus $H^2(U', T_{U'/S})=0.$
\end{remark}
\begin{lemma}
    For any flat lift $\tilde{S} \to \bZ/p^2$ of $S$, the obstruction class $\ob_{X'/\tilde{S}}$ is equal to $\ob_{X/W_2(S)}$.
\end{lemma}
\begin{proof}
    Recall obstruction class $\ob_{X'/\tilde{S}}: L_{X'/S} \to f'^*L_{S/\tilde{S}}[1]$, which corresponds to $L_{X'/S} \to \nu_*f'^*L_{S/\tilde{S}}[1] \simeq f^*F_{S, *}L_{S/\tilde{S}}[1]$ by adjunction for $\nu: X' \to X.$ Note the canonical map $\tilde{S} \to W_2(S)$ induces $dF_{S, \tilde{S}/W_2(S)}:L_{S/W_2(S)} \to F_{S, *} L_{S/\tilde{S}}$, the local computation shows it is an isomorphism on $\cH^{-1}.$ Therefore, composing $\ob_{X/W_2(S)}: L_{X/S} \to f^*L_{S/W_2(S)}[1]$ with $dF_{S, \tilde{S}/W_2(S)}[1]$ we obtain $\ob_{X'/S}$. 
\end{proof}

The main result of this subsection is the following.
\begin{lemma}\label{rem:splittings-push=lifts-to-witt}
    Let $X/S$ be a representable quasi-syntomic map of algebraic stacks over $\bF_p$. The fppf $T_{X'/S}[1]$-torsor $\nu_{X/S}^{\gamma}: (X/S)^{dR, \gamma} \to X'$ is equivalent to the torsor of liftings of $X$ to $W_2(S).$
\end{lemma}
\begin{proof}
  Given $\tilde{X}/W_2(S),$ our goal is to produce $$X'(R)=\Map_{S}(\Spec (F_*R), X) \xrightarrow{s_{\tilde{X}}} \Map_S(\Spec(\bG_a^{\gamma}(R)), X) = (X/S)^{dR, \gamma}(R).$$ Note $\Map_S(\Spec(\bG_a^{\gamma}(R)), X) = \Map_{W_2(S)}(\Spec (\bG_a^{\gamma}(R)), \tilde{X})$ and $\Spec(\bG_a^{\gamma}(R))$, when viewed as $W_2(S)$-scheme, is the split square-zero deformation of $\Spec(F_*R)$ with the ideal $F_*R[1].$ Thus, $\tilde{X}/W_2(S)$ identifies the space $(X/S)^{dR}(R)$ with $\Map_{S}(\Spec (F_*R \oplus F_*R[1]), X)$ and we declare $s_{\tilde{X}}$ to be the map sending $x: \Spec(F_*R) \to X$ to $\pi \circ x$ where $\pi: \Spec (F_*R \oplus F_*R[1]) \to \Spec(F_*R)$ is the canonical map. It is clearly the section since $F_*R \xrightarrow{\pi} F_*R \oplus F_*R [1] \to F_*R$ is the identity.
\end{proof}

\begin{corollary}
    Let $X/S$ be a smooth morphism of algebraic stacks over $\bF_p.$ Splitting of $\cO_{X} \to \tau^{\le 1}\dR_{X/S} \to \Omega^1_{X'/S}[-1]$ in $\cD_{qc}(X')$ is equivalent to lifting $X/S$ to $W_2(S).$
\end{corollary}
\begin{proof}
    Indeed, from \Cref{rem:splittings-push=lifts-to-witt} we learn that a groupoid of lifts of $X/S$ to $W_2(S)$ is equivalent to the groupoid of splitting of $(X/S)^{dR} \to X'.$ The groupoid of splittings of $\tau^{\le 1}\dR_{X/S}$, when non-empty, is a torsor for $T_{X'/S}$, any splitting of $\nu_{X/S}^\gamma$ gives rise to a splitting $\Omega^1_{X'/S}[-1] \to \tau^{\le 1}\dR_{X/S}$ and thus gives rise to an equivalence of groupoids.
\end{proof}

Now \Cref{rem:splittings-push=lifts-to-witt} allows us to construct an $F_{X/S}^*T_{X'/S}$-torsor on $X$ with a flat connection, which plays a central role.
\begin{construction}\label{cons:torsor}
    Let $X/S$ be an lci family and $\tilde{X}/W_2(S)$ a lift of $X/S.$ Then $$X \xrightarrow{\pi_{X/S}^{\gamma}} (X/S)^{dR, \gamma} \simeq B_{X'}(T_{X'/S})$$ defines a $F_{X/S}^*T_{X'/S}$-torsor on $X.$ We will denote it by $\cL_{\tilde{X}}.$ Note it comes equipped with a flat connection since $\pi_{X/S}^{\gamma}$ naturally factors as $X \xrightarrow{\pi_{X/S}} (X/S)^{dR} \to (X/S)^{dR, \gamma}.$ We denote by $\cL^c_{\tilde{X}}$ the corresponding $F_{X/S}^*T_{X'/S}$-torsor on $(X/S)^{dR}.$ 
\end{construction}

First, we identify the underlying torsor of $\cL_{\tilde{X}}$ with a torsor of strong Frobenius lifts.

\begin{remark}\label{lifts-of-id}
    Let $f: X \to S$ be an lci family. A lift $\tilde{X}/W_2(S)$ is equivalent to a map $c_{\tilde{X}}: L_{X/W_2(S)} \to \nu_*\cO_{X'}[1]$ where $\nu: X \to X'$ is the canonical map. Then $F_{X/S}: X \to X'$ induces $\nu_*\cO_{X'} \to F_{X, *}\cO_{X}$ and we get $L_{X/W_2(X)} \xrightarrow{c_{\tilde{X}}} \nu_*\cO_{X'}[1] \to F_{X, *}\cO_{X}[1]$. Note the square-zero extension $X \to W_2(X)$ has an ideal $F_{X, *}\cO_{X}$ and thus defines $c_{W_2(X)}: L_{X/W_2(S)} \to F_{X, *}\cO_{X}[1].$ Note $c_{\tilde{X}} - c_{W_2(X)}: L_{X/W_2(S)} \to F_{X, *}\cO_X[1]$ naturally factors through $L_{X/W_2(S)} \to L_{X/S}.$ Thus, we obtain $F_{X/S}^*T_{X'/S}$-torsor which measures the difference between two square zero extensions. In other words, it is the torsor of deformations of the identity map on $X$ to a $W_2(S)$-linear map $W_2(X) \to \tilde{X}$.
\end{remark}

\begin{lemma}\label{lm:torsor=lifts-of-frob}
    Let $X/S$ be an lci family of $\bF_p$-schemes. A lift $\tilde{X}/W_2(S)$ defines a section $s_{\tilde{X}}: X' \to (X/S)^{dR, \gamma}$ which defines a $F_{X/S}^*T_{X'/S}$-torsor on $X$. This torsor is identified with the torsor from \Cref{lifts-of-id}.
\end{lemma}

\begin{proof}
    Recall from \Cref{cons:torsor} the $F_{X/S}^*T_{X'/S}$-torsor $\cL_{\tilde{X}} \simeq X \times_{(X/S)^{dR, \gamma}} X'$. In particular, the space of splittings is identified with the path space between $\pi_{X/S}^{\gamma}$ and $s_{\tilde{X}} \circ F_{X/S}$ in $(X/S)^{dR}(X).$ Let us now identify it with the set of maps $\tilde{X} \to W_2(X)$ deforming the identity. Let $S=\Spec(R), X=\Spec(A).$ Let $s: X' \to (X/S)^{dR, \gamma}$ be a section. Explicitly, $(X/S)^{dR}(X')=\Map_{R}(A, \bG_a^{\gamma}(A')),$ i.e., $s$ corresponds to a map $A \to \bG_a^{\gamma}(A')$ of $R$-algebras whose post-composition with $\bG_a^{\gamma}(A') \to F_*A'$ is the canonical map $A \to F_*A'.$ Recall from \Cref{witt-vectors} that a map $\tilde{A} \to W_2(A)$ is the same as a homotopy filling in the square
\[\begin{tikzcd}
	& A && {\bG_a^{\gamma}(A)} \\
	{\tilde{A}} && A
	\arrow["{\pi_{A}^{\gamma}}", from=1-2, to=1-4]
	\arrow["{\operatorname{pr}}", from=2-1, to=1-2]
	\arrow["{F_A \circ \operatorname{pr}}"', from=2-1, to=2-3]
	\arrow["{c_A}"', from=2-3, to=1-4].
\end{tikzcd}\]
Using the factorization $F_A: A \xrightarrow{\nu_A} A' \xrightarrow{F_{A/R}} A,$ we obtain $c_A \circ F_A \circ \operatorname{pr} \simeq \bG_a^{\gamma}(F_{A/R}) \circ c_{A'} \circ \nu_A \circ \operatorname{pr}$. Recall $\tilde{A}$ comes equipped with a canonical homotopy $s \circ \operatorname{pr} \simeq c_{A'} \circ \nu_{A} \circ \operatorname{pr}$. Thus, we learn that giving a map $\tilde{A} \to W_2(A)$ is equivalent to giving a homotopy $\pi_{A}^{\gamma} \circ \operatorname{pr} \simeq \bG_a^{\gamma}(F_{A/R}) \circ s \circ \operatorname{pr}$. Since $\Map_{W_2(R)}(\tilde{A}, \bG_a^{\gamma}(A)) \simeq \Map_{R}(A, \bG_a^{\gamma}(A)),$ the claim follows.
\end{proof}
\begin{remark}
    Another way to see the torsor of strong Frobenius lifts is the following. Given $\tilde{X}/W_2(S),$ one defines its Greenberg transform $\tilde{X}^{W_2}$ which is an $S$-scheme sending $T \to S$ to $\Map_{W_2(S)}(W_2(T), \tilde{X}).$ There is a natural map $\tilde{X}^{W_2} \to X$ and, by deformation theory, it exhibits the source as a $F_{X/S}^*T_{X'/S}$-torsor over the target. Note if $s: X \to \tilde{X}^{W_2}$ is a section, then it gives rise to a point in $\tilde{X}^{W_2}(X),$ i.e., to a map $W_2(X) \to \tilde{X}$ which is automatically a strong Frobenius lift.
\end{remark}

We would like to compare the torsor from \Cref{cons:torsor} with the torsor of Ogus–Vologodsky constructed in \cite[Theorem 1.1]{MR2373230}, denoted by $\cL_{\cX}$. First, they depend on different data: $\cL_{\tilde{X}}$ only depends on a lift $\tilde{X}$ of $X$ to $W_2(S)$. The torsor $\cL_{\cX}$ depends on a lift $\tilde{S}$ of $S$ to $\bZ/p^2$ and a further lift of $X'$ to $\tilde{S}$. However, any such lift $\tilde{S}$ identifies the groupoids
\[
\begin{tikzcd}
	{\{\text{Lifts of } X \text{ to } W_2(S)\}} &&& {\{\text{Lifts of } X' \text{ to } \tilde{S}\}}
	\arrow[dashed, tail reversed, from=1-1, to=1-4]
\end{tikzcd}
\]
since both are torsors for $F_{X/S}^*T_{X'/S}$, and there is a natural map which sends a lift $\tilde{X}/W_2(S)$ to $\tilde{X'} := \tilde{X} \times_{\tilde{S}} W_2(S)$, where $\tilde{S} \to W_2(S)$ is the natural map.

Thus, one may ask whether the torsors $\cL_{\tilde{X}}$ and $\cL_{\cX}$ are isomorphic. The next remark explains that the underlying torsors are isomorphic. Later, we will explain that they are isomorphic as torsors with flat connections; see \Cref{lem:connection-of-global-stacky}.

\begin{remark}\label{iso-of-torsor-deformation}
    Let $X/S$ be a representable quasi-syntomic map of algebraic $\bF_p$-stacks. Let $\tilde{X}/W_2(S)$ be a lift of $X.$ In particular, we obtain a $F_{X/S}^*T_{X'/S}$-torsor on $(X/S)^{dR}$ and we denote its underlying torsor by $\cL_{\tilde{X}}$. Assume we are given $\tilde{S} \to \bZ/p^2$ a lift of $S$. This defines $\tilde{X}':= \tilde{X} \times_{W_2(S)} \tilde{S}$ a lift of $X'/S$ to $\tilde{S}.$ Let $\cL_{\cX}$ be the $F_{X/S}^*T_{X'/S}$-torsor associated with $\tilde{X'}/\tilde{S}$ as in \cite[Theorem 1.1]{MR2373230}. Assume we are given a lift $\tilde{X}_S \to \tilde{S}$ of $X$. Then $\cL_{\cX}$ is identified with the torsor of lifts of $F_{X/S}: X \to X'$ to a map $\tilde{X}_S \to \tilde{X}' := \tilde{X} \times_{W_2(S)} \tilde{S}$, see \cite[Theorem 1.1]{MR2373230}. Note there is a canonical map $\cL_{\tilde{X}} \to \cL_{\cX}$ of $F_{X/S}^*T_{X'/S}$-torsors. In particular, any such lift $\tilde{S}$ identifies $\cL_{\tilde{X}}$ with $\cL_{\cX}.$
\end{remark}

First, we compute the $p$-curvature of the torsor from \Cref{cons:torsor}. We need the following auxiliary lemma.

\begin{lemma}\label{lm:vertical-vector-fields-p-class-stack}
    Let $s: X \to B_X(E)$ be a map classifying an $E$-torsor $q: Y \to X.$ Denote $t: B_X(E^{\sharp}) \to B_X(E)$ to be the map induced by the canonical map $E^{\sharp} \to E.$ Then $t^*s_* \cO_X \in \cD_{qc}(B_X(E^\sharp))$ is identified with the pair\footnote{We recall that $\cD_{qc}(B_XE^{\sharp})$ is identified with the $\infty$-category of pairs: $(\cM, \theta)$ where $\cM \in \cD_{qc}(X)$ and $\theta: \cM \to \cM \otimes E^{\vee}$ a locally nilpotent map.} $(f_*\cO_Y, d)$ where $d: f_*\cO_Y \to f_*\cO_Y \otimes E^{\vee}$ is the de Rham differential under $f_*\cO_Y \otimes E^{\vee} \simeq f_* \Omega^1_{Y/X}.$ 
\end{lemma}
\begin{proof}
Consider the following diagram
\[\begin{tikzcd}
	{(Y/X)^{dR}} && {(Y/S)^{dR}} && (X/S)^{dR} \\
	\\
	X && {(X/S)^{dR}} && {(B_X(E)/S)^{dR}} \\
	&& {B_X(E)}
	\arrow[from=1-1, to=1-3]
	\arrow[from=1-1, to=3-1]
	\arrow["\ulcorner"{anchor=center, pos=0.125}, draw=none, from=1-1, to=3-3]
	\arrow[from=1-3, to=1-5]
	\arrow["{f^{dR}}", from=1-3, to=3-3]
	\arrow["\ulcorner"{anchor=center, pos=0.125}, draw=none, from=1-3, to=3-5]
	\arrow["0"', from=1-5, to=3-5]
	\arrow["{\pi_{X/S}}"', from=3-1, to=3-3]
	\arrow["s"', from=3-1, to=4-3]
	\arrow["{s^{dR}}", from=3-3, to=3-5]
	\arrow["{\pi_{BE}}"', from=4-3, to=3-5]
\end{tikzcd}\]
where both squares are cartesian. Namely, the right square is cartesian because $(-/S)^{dR}$ preserves limits and the left square is cartesian by definition. Using the factorization of the bottom composition, we learn that
\[
(Y/X)^{dR} \simeq X \times_{B_X(E)} B_X(E^{\sharp}),
\]
where the maps are $s$ and $t$. Indeed, the fiber product of $\pi_{BE}$ and $0$ is $B_X(E^{\sharp})$, with the map $B_X(E^{\sharp}) \to B_X(E)$ being the canonical map $t$. In particular, $t^*s_*\cO_X$ is identified with $g_* \cO_{(Y/X)^{dR}}$ for $g \colon (Y/X)^{dR} \to B_X(E^{\sharp})$. Note that the $\bV(E^{\sharp})$-torsor classified by $g$ is $\pi_{Y/X} \colon Y \to (Y/X)^{dR}$. In particular, it follows that the underlying $\cO_X$-module of $g_* \cO_Y$ is $f_* \cO_X$. By descent along $X \to B_X(E^{\sharp})$, the locally nilpotent endomorphism on $f_*\cO_X$ comes from the $\bV(E)^{\sharp}$-action on $Y$, and it is the canonical action coming from the $\bV(E)$-action. Since $D_{\Delta}(Y \times_X Y) \simeq Y \times \bV(E)^{\sharp}$, exchanging the groupoid action on the left with the canonical action on the right, we obtain $[Y/\bV(E)^{\sharp}] \simeq (Y/X)^{dR}$, which is also the desired result.
\end{proof}

A crucial input to the theory of Ogus--Vologodsky is the fact that the crystal of Frobenius lifts has a rich $p$-curvature, see \cite[Proposition 1.5]{MR2373230}. It corresponds to the fact that the map $(X/S)^{dR} \to B_{X'}(T_{X'/S})$, which is obtained from the lift $\tilde{X}/W_2(S)$, is equivariant with respect to the action of $B_{X'}(T_{X'/S}^{\sharp})$, which we make precise in the following lemma.
\begin{lemma}\label{p-curvature-rich}
    Let $X/S$ be a representable map of $\bF_p$-stacks. Let $\tilde{X}/W_2(S)$ be a lift. It gives rise to a $F_{X/S}^*T_{X'/S}$-torsor with a flat connection $\cL_{\tilde{X}}^c \to (X/S)^{dR}$ coming from $(X/S)^{dR} \to (X/S)^{dR, \gamma} \simeq B_{X'}(T_{X'/S}) $. The $p$-curvature  of $q_* \cO_{\cL_{\tilde{X}}^c}$ is given by the formula  $q_*\cO_{\cL_{\tilde{X}}/X} \xrightarrow{d} q_*\Omega^1_{\cL_{\tilde{X}}/X} \simeq q_*q^*F_{X'/S}^*\Omega^1_{X'/S} \simeq q_* \cO_{\cL_{\tilde{X}}} \otimes F_{X/S}^*\Omega^1_{X'/S}.$
\end{lemma}
\begin{proof}
    By \Cref{rem:p-curv=action} it is enough to compute $a^* q_* \cO_{\cL_{\tilde{X}^c}}$ where $a: (X/S)^{dR} \times B_{X'}(T_{X'/S}^{\sharp}) \to (X/S)^{dR}$ is the action morphism.
    Consider the composition $$B_{(X/S)^{dR}}(F_{X'/S}^*T_{X'/S}^{\sharp}) \simeq (X/S)^{dR} \times_{X'} B_{X'}(T_{X'/S}^{\sharp}) \xrightarrow{a} (X/S)^{dR} \to (X/S)^{dR, \gamma} \simeq B_{X'}(T_{X'/S}) $$
    and note the resulting map $B_{(X/S)^{dR}}(F_{X'/S}^*T_{X'/S}^{\sharp}) \to B_{X'}(T_{X'/S})$ lives over $\nu_{X/S}: (X/S)^{dR} \to X'$ and the map $F_{X/S}^*T_{X'/S}^{\sharp} \to F_{X/S}^* T_{X'/S}$ is the canonical map since the map $(X/S)^{dR} \to B_{X'}(T_{X'/S})$ is equivariant for $B_{X'}(T_{X'/S}^\sharp)$. The claim follows from \Cref{lm:vertical-vector-fields-p-class-stack}.
\end{proof}
Now we will compute the connection of the torsor from \Cref{cons:torsor}. First, we recall what it means.

\begin{remark}\label{rem:torsor-on-XdR}
    Let $E \in \Vect((X/S)^{dR})$ and $f: (X/S)^{dR} \to B_{(X/S)^{dR}}E$ an $E$-torsor on $(X/S)^{dR}.$ Specifying $f$ is equivalent to specifying a map $\cO_{(X/S)^{dR}} \to E[1]$ in $\cD_{qc}((X/S)^{dR}) \simeq \Crys(X/S)$. Discussion in \cite[\S 1.1, (1.0.9)]{MR2373230} explains that this extension is equivalent to giving an $E$-torsor $\cL$ over $X$ together with $\nabla_{\cL}: \cL \to E \otimes \Omega^1_{X/S}$ such that $\nabla_{\cL}(l+e) = \nabla_{\cL}(l) + \nabla_{E}(e)$ for all local sections $l \in \cL, e \in E$ and $\nabla_{E} \circ \nabla_{\cL} = 0.$ 
\end{remark}

\begin{remark}\label{rem:interpret-torsor-on-XdR}
    Let $E$ be a locally free sheaf on $(X/S)^{dR}.$ Let $(X/S)^{dR} \xrightarrow{f} B_{(X/S)^{dR}} (E)$ be an $E$-torsor on $X$ with a flat connection. This is equivalent to giving an $E$-torsor $\cL$ over $X$ together with $\nabla_{\cL}: \cL \to E \otimes \Omega^1_{X/S}$ satisfying some properties. In this remark we explain how to compute $\nabla_{\cL}(s)$ for a local section $s \in \cL(X).$ Note that $s$ provides a commutativity datum for the diagram 
\[\begin{tikzcd}
	{(X/S)^{dR}} &&& {B_{(X/S)^{dR}}(E)} \\
	\\
	X &&& X
	\arrow["f", from=1-1, to=1-4]
	\arrow["{\pi_{X/S}}", from=3-1, to=1-1]
	\arrow[equals, from=3-1, to=3-4]
	\arrow["0"', from=3-4, to=1-4]
\end{tikzcd}\]
which gives rise to a map of Čech nerves $\Cech(X \to (X/S)^{dR}) \to \Cech(X \to B_{(X/S)^{dR}}(E))$. The map on $1$-simplices is some map $f: D_{\Delta}(X^2) \to \bV(E)$. By construction, this map restricts to $X \xrightarrow{\operatorname{Id}} X$ with $X \mono D_{\Delta}(X^2)$ and $X \mono \bV(E)$. Equivalently, the map $E^\vee \to \cO(D_{\Delta}(X^{\times_S 2}))$ lands in the ideal of the diagonal $I_{\Delta}$. In particular, we obtain a map $E^\vee \to I_{\Delta}/I_{\Delta}^2 \simeq \Omega^1_{X/S}$, which is the same as a section of $E \otimes \Omega^1_{X/S}$, and by construction it is $\nabla_{\cL}(s)$.

\end{remark}

Now we are ready to compute the connection on $\cL_{\tilde{X}}.$

\begin{lemma}\label{lem:connection-of-global-stacky}
    Let $X/S$ be a representable map of $\bF_p$-stacks. Let $\tilde{X}/W_2(S)$ be a flat lift. One obtains a $F_{X/S}^*T_{X'/S}$-torsor on $X$ with a flat connection corresponding to $(X/S)^{dR} \to (X/S)^{dR, \gamma} \simeq B_{X'}(T_{X'/S}).$ The corresponding connection $\cL_{\tilde{X}} \to F_{X/S}^*T_{X'/S} \otimes_{\cO_X} \Omega^1_{X/S}$ sends a local section $f: W_2(X) \to \tilde{X}$ to $d\tilde{F}/p \in \cH om(F_{X/S}^*\Omega^1_{X'/S}, \Omega^1_{X/S})$ where $\tilde{F}$ is the Frobenius lift defined by $f.$
\end{lemma}
\begin{proof}
    The claim is Zariski local. We assume $X=\Spec(A),$ $S=\Spec(R)$, and $f \in \cL_{\tilde{X}}(X)$ a local section. Recall $f$ is equivalent to a homotopy $h_f: \pi_{X/S}^{\gamma} \simeq  s_{\tilde{X}} \circ F_{X/S}$. In particular, $h_f$ can be used to represent a map $(X/S)^{dR} \to (X/S)^{dR, \gamma}$ using their Cech nerves along $\pi_{X/S}: X \to (X/S)^{dR}$ and $s_{\tilde{X}}: X' \to (X/S)^{dR, \gamma.}$ The map on $1$-simplices is $$r_2: D_{\Delta}(X \times_S X) \simeq X \times_{(X/S)^{dR}}  X \to X' \times_{(X/S)^{dR, \gamma}} X' \simeq \bT_{X'/S}.$$
It is a map of $X'$-schemes and corresponds to a Frobenius-linear derivation $A \to D_{\Delta}(A^{\otimes_R 2})$. We claim it is given by the formula $$\theta_f(a) := \dfrac{1 \otimes \tilde{F}(\tilde{a})-\tilde{F}(\tilde{a}) \otimes 1}{p} \mod p$$
where $\tilde{a}$ is any lift of $a$ along $\tilde{A} \to A$. 

\textit{Step 1. Identifying the derivation.}
Let $(x, y, h_{\gamma})$ be a $B$-point of $X \times_{(X/S)^{dR}} X$ where $B$ is a test $R$-algebra; here $x, y: A \to B$ and $h_{\gamma}$ is a homotopy between two compositions in $\Map_{R}(A, \bG_a^{dR}(B)).$ Then $X \times_{(X/S)^{dR}} X \to X' \times_{(X/S)^{dR, \gamma}} X'$ sends this $B$-point to $(F(x), F(y), h)$. Here, $h$ is a homotopy obtained from $h_{\gamma}$ by using $h_{f}: s_{\tilde{X}} \circ F_{X/S} \simeq \pi_{X/S}^{\gamma}.$ Let us make $h_{f}$ explicit to compute the map.  Note the points $s_{\tilde{X}} \circ F_{X/S}, \pi_{X/S}^{\gamma} \in (X/S)^{dR, \gamma}(X)=\Map_{R}(A, F_*A/p)$ are represented by 
\[\begin{tikzcd}
	{F_*A'} & {\tilde{A}} && {F_*A} & {W_2(A)} \\
	{F_*A'} & {F_*A'} && {F_*A} & {F_*A} \\
	{F_*A} & {F_*A}
	\arrow["i", from=1-1, to=1-2]
	\arrow[from=1-1, to=2-1]
	\arrow["{\nu \circ pr}", from=1-2, to=2-2]
	\arrow["V", from=1-4, to=1-5]
	\arrow[equals, from=1-4, to=2-4]
	\arrow["F", from=1-5, to=2-5]
	\arrow["0", from=2-1, to=2-2]
	\arrow["{F_{A/R}}"', from=2-1, to=3-1]
	\arrow["{F_{A/R}}", from=2-2, to=3-2]
	\arrow["0"', from=2-4, to=2-5]
	\arrow[from=3-1, to=3-2]
\end{tikzcd}\]
respectively.
 The first diagram defines a map $A \to F_*A/p$ which is also obtained from $$\tilde{A} \xrightarrow{f} W_2(A) \to F_*\tilde{A} \to \Cone(F_*\tilde{A} \xrightarrow{p} F_*\tilde{A})$$
via the base change along $W_2(R) \to R.$ Similarly, the second diagram is obtained from a similar map $\Cone(F_*\tilde{A} \xrightarrow{V} W_2(A)) \to \Cone(F_*\tilde{A} \xrightarrow{p} F_*\tilde{A}).$ Now we note that $f$ gives rise to a homotopy between these two points in $$(\tilde{X}/W_2(S))^{dR, \gamma}(\tilde{X}) \simeq \Map_{W_2(R)}(\tilde{A}, \Cone(F_*\tilde{A} \xrightarrow{p} F_*\tilde{A})).$$
Explicitly, under $\pi_{\tilde{X}/W_2(S)}$ an elements $\tilde{a}$ is sent to $\tilde{a}^p.$ Under the other map, an element $\tilde{a}$ is sent to $\tilde{a}^p+p\delta_{\tilde{a}}.$ By construction, these two elements are canonically homotopic and $\delta_{\tilde{a}}$ is the desired homotopy.

 \textit{Step 2.} Take $B=D_{\Delta}(A^{\otimes_R 2})$ and $(x, y, h_{\gamma})$ to be the canonical $B$-point of $X \times_{(X/S)^{dR}} X,$ i.e., $x, y: A \to D_{\Delta}(A^{\otimes_R 2})$ the two canonical maps and $h_{\gamma}$ is the homotopy between two compositions $$A \to D_{\Delta}(A^{\otimes_R 2}) \to \bG_a^{dR}(D_{\Delta}(A^{\otimes_R 2})) $$
 coming from divided powers on the ideal $\Ker(D_{\Delta}(A^{\otimes_R 2}) \to A).$ The map $X \times_{(X/S)^{dR}} X \to X' \times_{(X/S)^{dR, \gamma}} X'$ sends this point to some map $\Sym \Omega^1_{A'/R} \to D_{\Delta}(A^{\otimes_R 2})$ which we can compute using the first step. Namely, it is the homotopy between two compositions $$A \to D_{\Delta}(A^{\otimes_R 2}) \xrightarrow{F} F_* D_{\Delta}(A^{\otimes_R 2}) \xrightarrow{\pi} \bG_a^{dR, \gamma}(D_{\Delta}(A^{\otimes_R 2}))$$
 which comes from a homotopy between two compositions 
 $$\tilde{A} \to D_{\Delta}(\tilde{A}^{\otimes_{W_2(R)}2}) \xrightarrow{\tilde{F}} \tilde{F}_* D_{\Delta}(\tilde{A}^{\otimes_{W_2(R)}2})  \xrightarrow{s} \bG_a^{dR, \gamma}(D_{\Delta}(\tilde{A}^{\otimes_{W_2(R)}2}) ).$$
 Explicitly, for any $\tilde{a} \in \tilde{A}$, this homotopy provides a path between $1 \otimes \tilde{F}(a)$ and $\tilde{F}(a) \otimes 1$ in $\bG_a^{dR, \gamma}(D_{\Delta}(\tilde{A}^{\otimes_{W_2(R)}2}) )$. Thus, the resulting $D_{\Delta}(A^{\otimes_R 2})$-derivation of $A$ is given by the desired formula.

\textit{Step 3.} Note the image of the map $\theta_f: \Omega^1_{A'} \to D(A^{\otimes_R 2})$ lands in the ideal of the diagonal $I_{\Delta}$ by construction. Therefore, one obtains $\Omega^1_{A'/R} \to I_{\Delta}/I_{\Delta}^2 \simeq \Omega^1_{A/R}$ which is equal to $\zeta_{\tilde{F}} = p^{-1}d\tilde{F}$ from the explicit formula of $\theta_f.$
\end{proof}

\begin{corollary}
    Let $X/S$ be a smooth family of $\bF_p$-schemes. Let $\tilde{X}/W_2(S)$ be a flat lift of $X$ and let $\tilde{S} \to \bZ/p^2$ be a flat lift of $S.$ Denote $\tilde{X}' := \tilde{X} \times_{W_2(S)} \tilde{S}$ be the corresponding lift of $X'$ to $\tilde{S}$ induced by $\tilde{X}.$ The torsor $\cL_{\tilde{X}}$ associated with $\tilde{X}$ is isomorphic to the torsor $\cL_{\cX}$ studied by Ogus and Vologodsky associated to $\tilde{X}'/\tilde{S}$ compatibly with the connections.
\end{corollary}
We need the following auxiliary remark in order to compare the splitting of $F_{X/S}D_{X/S}$ we obtain from the splitting of the de Rham gerbe with the splitting of Ogus-Vologodsky.

\begin{remark}\label{rem:splitting-module}
    Let $s: X \to (X/S)^{dR, \gamma}$ be a section of $\pi_{X/S}^{\gamma}: (X/S)^{dR, \gamma} \to X'.$ It extends to an isomorphism $B_{X'}T_{X'/S} \simeq (X/S)^{dR, \gamma}$ of $T_{X'/S}$-gerbes. In particular, one obtains 
\begin{equation}\label{splitting-module}
    \cE nd_{(X/S)^{dR, \gamma}}(s_*\cO_{X'}, s_* \cO_{X'}) \simeq \cE nd_{(B_{X'}(T_{X'/S}))} (i_*\cO_{X'}, i_*\cO_{X'}) \simeq \hat \Gamma_{X'} T_{X'/S}
\end{equation}
where $i: X' \to B_X'(T_{X'/S})$ is the zero-section.
Note $s_* \cO_{X'}$, as a $D_{X/S}^{\gamma}$-module, is $q_*\cO_{\cL_{\tilde{X}}}$ endowed with its $D_{X/S}^{\gamma}$-module structure. Under the equivalence $\QCoh((X/S)^{dR, \gamma}) \simeq \MIC^{\cdot}(X/S)$, the left term in \eqref{splitting-module} corresponds to $\cE nd_{D_{X/S}^{\gamma}}(s_*\cO_{X'}^*, s_*\cO_{X'}^*)$ where $s_*\cO_{X'}^*$ is the $\cO_{X}$-linear dual of $q_*\cO_{\cL_{\tilde{X}}}$ with its flat connection. In particular, the equivalence induced by $s$ is the same as the one induced by the splitting of the Azumaya algebra $D_{X/S}^{\gamma}$ given by the $D_{X/S}^{\gamma}$-module $\cH om_{\cO_X}(q_*\cO_{\cL_{\tilde{X}}}, \cO_X).$

\end{remark}

Now we are ready to recover the Cartier transform.

\begin{theorem}\label{main:global-ov}
    Let $X/S$ be a representable quasi-syntomic morphism of algebraic $\bF_p$-stacks. Any lift $\tilde{X}/W_2(S)$ gives rise to a symmetric monoidal equivalence $C^{-1}_{\tilde{X}}:  \HIG_{\gamma}^{\cdot}(X'/S) \to \MIC_{\gamma}^{\cdot}(X/S)$. Moreover, if $X/S$ is a smooth map of $\bF_p$-schemes and we are given a lift $\tilde{S} \to \bZ/p^2$ of $S$, then $C^{-1}_{\tilde{X}}$ is isomorphic to the global Cartier transform from \cite[Theorem 2.8]{MR2373230} applied to the data of a flat lift $\tilde{X} \times_{W_2(S)} \tilde{S} \to \tilde{S}$ of $X'/S.$
\end{theorem}
\begin{proof}
\textit{Step 0.} A lift $\tilde{X}/W_2(S)$ provides a splitting of $\nu_{X/S}^{\gamma}: (X/S)^{dR, \gamma} \to X'$ by \Cref{rem:splittings-push=lifts-to-witt}. Therefore, it defines a symmetric monoidal equivalence of categories $\MIC^{\cdot}_{\gamma}(X/S) \simeq \HIG^{\cdot}_{\gamma}(X'/S).$

 \textit{Step 1.} Now assume $X \to S$ is a smooth map of $\bF_p$-schemes. Let us compare the equivalences in the presence of $\tilde{S}$. Since $\tilde{X'} := \tilde{X} \times_{W_2(S)} \tilde{S}$ is a flat lift of $X'/S$ to $\tilde{S},$ thus \cite[Theorem 2.8]{MR2373230} provides us with a symmetric monoidal equivalence $C^{-1}_{\cX}: \MIC^{\cdot}_{\gamma}(X/S) \to \HIG^{\cdot}_{\gamma}(X'/S)$. Let us prove that it is isomorphic to $C^{-1}_{\tilde{X}}.$ The composition $(X/S)^{dR} \to (X/S)^{dR, \gamma} \to B_{X'} (T_{X'/S})$ defines a $F_{X/S}^* T_{X'/S}$-torsor on $X$ with a flat connection. By \Cref{rem:splitting-module} it is enough to prove that this torsor is isomorphic to the torsor from \cite[Theorem 1.1]{MR2373230}. Recall the torsor $\cL_{\cX}$ from \cite[Theorem 2.8]{MR2373230} is the torsor of liftings of $F_{X/S}$ to $X'/S$ with the connection as in \cite[Theorem 1.1]{MR2373230}.  By \Cref{iso-of-torsor-deformation} the underlying torsor of $ (X/S)^{dR} \to (X/S)^{dR, \gamma} \simeq B_{X'}(T_{X'/S})$ is isomorphic to $\cL_{\cX}.$ Thus, it is enough to compare connections. We recall that the connection of Ogus-Vologodsky sends a local section $\tilde{F}$ to a map $d\tilde{F}/p: F_{X/S}^*\Omega^1_{X'/S} \to \Omega^1_{X/S}.$ Then \Cref{lem:connection-of-global-stacky} identifies the two connections. This finishes the proof.
\end{proof}

\section{Lifts of Frobenii and Local Cartier transform}
\begin{definition}\label{def:GadRf}
    Denote $\bG_a^{dR, F}$ to be a $W$-algebra stack given by $\Cone(F_*\bG_a \xrightarrow{p} F_*W_2).$
\end{definition}
\begin{remark}\label{rem:Ga-alg-on-alpha-p-version}
    Note that $W$-algebra structure of $\bG_a^{dR, F}$ refines to a $\bG_a$-algebra structure via 
\[\begin{tikzcd}
	{F_*\bG_a} &&& {W_2} \\
	{F_*\bG_a} &&& {F_*W_2}
	\arrow["V", from=1-1, to=1-4]
	\arrow[equals, from=1-1, to=2-1]
	\arrow["F", from=1-4, to=2-4]
	\arrow["p", from=2-1, to=2-4].
\end{tikzcd}\]
\end{remark}
\begin{remark}\label{rem:pi-of-alphap-version}
Note $\pi_1(\bG_a^{dR, F})=\alpha_p$ and $\pi_0(\bG_a^{dR, F})=F_*\bG_a$ as sheaves of $\bG_a$-modules. Note \Cref{rem:Ga-alg-on-alpha-p-version} gives a map $\pi^{F}: \bG_a \to \bG_a^{dR, F}$ whose post-composition with $\nu^F: \bG_a^{dR, F} \to \pi_0(\bG_a)^{dR, F}) = F_*\bG_a$ is the Frobenius. This is also the map guaranteed by \Cref{Frob-factors-1-truncated}.
\end{remark}
\begin{definition}\label{dR1-def}
    For $X$ denote $X^{dR, F}$ to be the transmutation of the $\bG_a$-algebra stack $\bG_a^{dR, F}$ from \Cref{rem:Ga-alg-on-alpha-p-version}. For $X/S$ denote $(X/S)^{dR, F} = X^{dR, F} \times_{S^{dR, F}} S$ where $\pi_{S}^F: S \to S^{dR, F}$ is the transmutation of $\pi^F$ from \Cref{rem:pi-of-alphap-version}.
\end{definition}
\begin{remark}
    Note for $X/S$ there is a map $(X/S)^{dR, F} \to X'$ as in \Cref{for-lci-dR-torsor} which we denote by $\nu^F_{X/S}$ and the map $\pi_{X}^F$ induces $X \to (X/S)^{dR, F}$ which we denote by $\pi_{X/S}^F.$
\end{remark}
\begin{remark}
    If $X/S$ is lci, then the pushout of the $L_{X'/S}^{\vee}[1] \otimes_{\bG_a} \bG_a^{\sharp}$-torsor $\nu_{X/S}: (X/S)^{dR} \to X'$ along the map induced by $\bG_a^{\sharp} \to \alpha_p$ is identified with $(X/S)^{dR, F}.$ In particular, we get a $L_{X'/S}^{\vee}[1] \otimes_{\bG_a} \alpha_p$-torsor $\nu_{X/S}^F: (X/S)^{dR} \to X'.$ This also follows from \Cref{torsor-ring-stack} and \Cref{rem:pi-of-alphap-version}.
\end{remark}

Recall from \Cref{dR1-def}
the stack $(X/S)^{dR, F} := X^{dR, F} \times_{S^{dR, F}} S$.
\begin{remark}
    Recall that $(X/S)^{dR, F}$ is the pushout of the quasi-torsor $(X/S)^{dR} \to X'$ along $T_{X'/S}^{\sharp} \to T_{X'/S} \otimes_{\bG_a} \alpha_p$. In particular, if $X \to S$ is smooth, then this map exhibits the source as an fppf $B_{X'}(T_{X'/S} \otimes_{\bG_a} \alpha_p)$-torsor over the target. Consider $\varepsilon_{X/S}: (X/S)_{\text{fppf}} \to (X/S)_{\text{Zar}},$ then $R\varepsilon_{X/S, *} \alpha_p = \fib(\cO_{X'} \to F_{X', *}\cO_{X'}) =: \cB_{X'}$. Thus, we get a torsor for $T_{X'/S} \otimes_{\cO_{X'}} \cB_{X'}[1].$ Equivalently, this Picard stack can be written as $\Cone(T_{X'/S} \to F_{X', *}F_{X'}^*T_{X'/S}).$
\end{remark}

\begin{remark}\label{frob-lifts-zariski}
    Let $X/S$ be smooth. Consider the Zariski torsor which assigns to $U \to X$ a groupoid consisting of flat $\tilde{U}/W_2(S)$ together with a lift $\tilde{F}_{U/S}: \tilde{U} \to \tilde{U'}.$ Note it is a torsor for $\Cone(T_{X'/S} \to F_{X', *}F_{X'}^* T_{X'/S}).$ 
\end{remark}

\begin{remark}\label{dR-F-refines-to-frob}
Let $X/S$ be smooth. \Cref{frob-lifts-zariski} provides us with $\delta_{F_{X/S}}: L_{X'/S} \to \cB_{X'}[1]$ in $\cD_{qc}(X').$ Recall $\cB_{X'}=\cofib(\cO_{X'} \to F_{X', *} \cO_{X}),$ thus it comes equipped with $\cB_{X'} \to \cO_{X'}[1].$ Note the composition $L_{X'/S} \xrightarrow{\delta_{F_{X/S}}} \cB_{X'}[1] \to \cO_{X'}[2]$ is the obstruction of lifting $X$ to $W_2(S).$ Thus, given a lift $\tilde{X}/W_2(S)$, the map $\delta_{F_{X/S}}$ refines to a map $L_{X'/S} \to F_{X', *}\cO_{X'}[1].$ Equivalently, to a map $F_{X'}^*L_{X'/S} \to \cO_{X'}[1].$ Next lemma shows that the resulting $F_{X'}^*T_{X'/S}$-torsor on $X'$ is the torsor of lifts of $F_{X/S}$ to $\tilde{X}.$
\end{remark}

\begin{lemma}\label{A1-split-scheme-lift-frob}
    An fppf-torsor $(X/S)^{dR, F}$ is equivalent to the gerbe of lifts of $F_{X/S}$ to $W_2(S)$.
\end{lemma}
\begin{proof}
Given $\tilde{X}/W_2(S)$ with $\tilde{F}_{X/S}: \tilde{X} \to \tilde{X}',$ let us construct a section $X' \to (X/S)^{dR, F}.$ Supplying a lift of $\tilde{F}_{X/S}$ is equivalent to giving $W_2(X') \to \tilde{X}$ which deforms $\nu: X' \to X$, see \Cref{frob-lifts-append}. Define $s_{\tilde{F}}: X' \to (X/S)^{dR, F}$ as follows: for $x \in X'(R),$ i.e., $\Spec(R) \to X',$ we canonically obtain from $\Spec(W_2(R)) \xrightarrow{W_2(x)} W_2(X') \xrightarrow{g} \tilde{X}$ a point $\Spec(\bG_a^{dR, F}(R)) \to X.$ Indeed, to show that \[\Spec(\bG_a^{dR, F}(R)) \mono \Spec(W_2(R)) \xrightarrow{W_2(x)} W_2(X') \to \tilde{X}\]
factors through $X \mono \tilde{X}$, it is enough to assume $S=\Spec(R), X=\Spec(A)$. In this case, let us note that the composition $\tilde{A} \to W_2(A') \to \Cone(R \xrightarrow{p} W_2(R))$ admits a canonical factorisation $\tilde{A} \to A \to \Cone(R \xrightarrow{p} W_2(R))$. Recall the square-zero extension $A' \to \tilde{A} \to A$ is obtained from $F_*R \to W_2(R) \to R$ by applying $- \otimes_{W_2(R)} \tilde{A}.$ Consider the composition $F_R \otimes_{W_2(R)} \tilde{A} \to \tilde{A} \xrightarrow{f} W_2(A').$ It sends $V[r] \tilde{a}$ to $F(V([r])\tilde{a}=p[r]\tilde{a}$ since $f$ is $W_2(R)$-semilinear, which completes the proof.
\end{proof}

\begin{remark}
    Let $X/S$ be smooth with reduced $S.$ Then a lift $\tilde{X}/W_2(S)$ gives rise to two torsors.
    \begin{itemize}
        \item[(1)] The composition $X \to (X/S)^{dR} \to (X/S)^{dR, \gamma} \simeq B_{X'} T_{X'/S}$ gives a $F_{X/S}^*T_{X'/S}$-torsor on $X.$
        \item[(2)] An fppf $T_{X'/S} \otimes_{\bG_a} \alpha_p$-gerbe $(X/S)^{dR, F} \to X'$ corresponds to a Zariski $T_{X'/S} \otimes_{\cO_{X'}} F_{X', *}\cO_{X'}/\cO_{X'}$-torsor which can be refined to a $F_{X'}^*T_{X'/S}$-torsor using $\tilde{X}/W_2(S)$ and \Cref{dR-F-refines-to-frob}.
    \end{itemize}
\end{remark}
Note the first torsor is sent to the second one under the map induced by the natural map $F_{X/S}^*T_{X'/S} \to F_{X'}^*T_{X'/S}$. This follows from description of the corresponding obstructions, see \Cref{iso-of-torsor-deformation} and \Cref{dR-F-refines-to-frob}. Roughly, the second torsor splits if and only if $F_{X/S}$ lifts to $\tilde{X},$ while the first splits if and only if there exists a strong Frobenius lift. We note that when $S=\Spec(\bF_p),$ or, more generally, any perfect scheme, there is no difference between two notions.

\subsection{Local Cartier Equivalence}

\begin{remark}\label{connection-preserves-pd-envelope}
    Let $E$ be a locally free sheaf on $(X/S)^{dR}$ and $\cY \to (X/S)^{dR}$ an $E$-torsor. Assume we are given a section $s: X \to Y := \cY \times_{(X/S)^{dR}} X$ of the underlying $E$-torsor on $X.$ The connection on $\cY$ might not preserve the section but it preserves the PD-envelope of this section $D_{X}(Y)$. This defines an $E^{\sharp}$-torsor on $X$ with a flat connection, see \cite[Remark 2.4]{MR2373230}.
\end{remark}

To give a stacky interpretation of \Cref{connection-preserves-pd-envelope} the following observation is crucial.

\begin{lemma}\label{dR-torsor-on-dR}
    Let $f: (X/S)^{dR} \to B_{(X/S)^{dR}}(E)$ be an $E$-torsor on $X$ with a flat connection. Denote $\bar{f}: X \to B(E)$ to be $f \circ \pi_{X/S}$ the underlying $E$-torsor on $X.$ 
 The composition $$(X/S)^{dR} \xrightarrow{f} B_{(X/S)^{dR}}(E) \xrightarrow{\pi_{B E}} B_{(X/S)^{dR}}(E/E^{\sharp}) $$
    is naturally homotopic to $\bar{f}^{dR}: (X/S)^{dR} \to (B_X(E)/S)^{dR} \simeq B_{(X/S)^{dR}}(E/E^{\sharp}).$  
\end{lemma}
\begin{proof}
Using naturality of $\pi_{-/-}$ we obtain the following diagram
\[\begin{tikzcd}
	{(X/S)^{dR}} &&& {B_{(X/S)^{dR}}(E)} \\
	\\
	{((X/S)^{dR}/S)^{dR}} &&& {B_{(X/S)^{dR}}(E/E^\sharp)}
	\arrow["f", from=1-1, to=1-4]
	\arrow["{\pi_{(X/S)^{dR}}}"', from=1-1, to=3-1]
	\arrow["{\pi_{BE}}", from=1-4, to=3-4]
	\arrow["{f^{dR}}"', from=3-1, to=3-4]
\end{tikzcd}\]
where the left vertical arrow is homotopic to $\pi_{X/S}^{dR}$ by \Cref{rem:dRdRvsdR}. In particular, we learn that the composition $(X/S)^{dR} \to B_{(X/S)^{dR}}(E/E^\sharp)$ is homotopic to $(f \circ \pi_{X/S})^{dR} = \bar{f}^{dR} $ as desired.
\end{proof}

\begin{corollary}\label{cor:obtain-E-sharp-torsor}
    Let $f: (X/S)^{dR} \to B_{(X/S)^{dR}}(E)$ be an $E$-torsor on $X$ with a flat connection. Any trivialization of the composition $X \to (X/S)^{dR} \to B_{(X/S)^{dR}}(E)$ gives rise to a map $f^\sharp: (X/S)^{dR} \to B_{(X/S)^{dR}}(E^\sharp)$ refining $f.$ Moreover, the underlying $E^\sharp$-torsor on $X$ is given by the PD-envelope $D_{X}(Y)$ of $s: X \to Y$. 
\end{corollary}
\begin{proof}
    The first assertion follows from \Cref{dR-torsor-on-dR}. To compute the underlying $E^\sharp$-torsor on $X$, we have to compute the fiber product of $\bar{f}^\sharp: X \xrightarrow{\pi_{X/S}} (X/S)^{dR} \xrightarrow{f} B_{X}(E^\sharp)$ with the zero map $X \to B_X(E^\sharp)$. For this, recall from \Cref{rel-dR-of-BE} that $B_{X}(E^\sharp) \simeq (X/B_{X}(E))^{dR}$. Under this identification, the map $\bar{f}^\sharp: (X/X)^{dR} \to (X/B_{X}(E))^{dR}$ is given by a map of pairs 
\[\begin{tikzcd}
	X &&& X \\
	X &&& {B_{X}(E)}
	\arrow["{\operatorname{Id}}", from=1-1, to=1-4]
	\arrow[equals, from=1-1, to=2-1]
	\arrow["{\bar{f}}", from=1-4, to=2-4]
	\arrow["0", from=2-1, to=2-4]
\end{tikzcd}\]
with the commutativity datum provided by the trivialization of $\bar{f}.$ The zero map $X \to B_{X}(E^\sharp)$ corresponds to a similar map of pairs where the map with the right vertical map being $0$. Recall from \Cref{dR-commutes-colimits} that $(-/-)^{dR}$ preserves products. Thus, we compute $$(X/X)^{dR} \times_{(X/B_{X}(E))^{dR}} (X/X)^{dR} \simeq (X/Y)^{dR} $$
where the map $X \to Y$ is equal to $s.$
\end{proof}

Now, we will prove that the connection on the obtained $E^\sharp$-torsor is the same one as in \Cref{connection-preserves-pd-envelope}.
\begin{lemma}
    Let $E \in \Vect((X/S)^{dR})$ and $\cY \to (X/S)^{dR}$ an $E$-torsor. Recall that it amounts to an $E$-torsor $q: Y \to X$  together with a flat connection. Let $s: X \to Y$ be a section of the underlying torsor on $X$. By \Cref{cor:obtain-E-sharp-torsor} one obtains an $E^\sharp$-torsor on $X$ with a flat connection. This torsor with flat connection is isomorphic to the one from \Cref{connection-preserves-pd-envelope}.
\end{lemma}

\begin{proof}
By \Cref{cor:obtain-E-sharp-torsor} the underlying $E^\sharp$-torsors are isomorphic and given by the PD-envelope of $s: X \to Y.$ Thus, it is enough to compare the connections.
    Note $s$ identifies $q_* \cO_Y \simeq S^{\bullet}(E^\vee).$ The connection from \Cref{connection-preserves-pd-envelope} is determined by $\nabla(f^{[n]})=f^{[n-1]}\nabla(f)$. Let us compute the second connection. 
    Note we have a diagram
\[\begin{tikzcd}
	{(X/S)^{dR}} && {B_{(X/S)^{dR}}(E)} && {(B_{(X/S)^{dR}}(E)/X)^{dR}} \\
	X && X && X \\
	{D_{\Delta}(X^{\times_S 2})} && {\bV(E)} && {(\bV(E)/X)^{dR}}
	\arrow["f", from=1-1, to=1-3]
	\arrow["{\pi_{BE}}", from=1-3, to=1-5]
	\arrow[from=2-1, to=1-1]
	\arrow[equals, from=2-1, to=2-3]
	\arrow[from=2-3, to=1-3]
	\arrow[equals, from=2-3, to=2-5]
	\arrow[from=2-5, to=1-5]
	\arrow[shift right, from=3-1, to=2-1]
	\arrow[shift left, from=3-1, to=2-1]
	\arrow["{f_1}", from=3-1, to=3-3]
	\arrow[shift right, from=3-3, to=2-3]
	\arrow[shift left, from=3-3, to=2-3]
	\arrow["g", from=3-3, to=3-5]
	\arrow[from=3-5, to=2-5]
	\arrow[shift left=2, from=3-5, to=2-5]
\end{tikzcd}\]
where commutativity data of two upper squares are induced by $s$ and lower row is the first level of the map on Cech nerves.
Say $X=\Spec(A)$ and $S=\Spec(R).$ Denote $D := D_{\Delta}(A^{\otimes_ R2})$ and $M$ is an $A$-module corresponding to $E$. Then $f_1$ corresponds to a map $d: E^\vee \to D$ whose image lands in $I_{\Delta}.$ The map $g$ sends this $D$-point of $\bV(E)$ to a point of $(\bV(E)/X)^{dR}(D)=\Map_A(E^\vee, \bG_a^{dR}(D))$ which is given by $E^\vee \xrightarrow{d} D \to \bG_a^{dR}(D).$ Recall the section $s$ provides a null-homotopy of this map. Explicitly, this null-homotopy comes from the fact that elements of $I_\Delta$ have canonical divided powers in $D$ and the image of $E^\vee \xrightarrow{d} D$ lands in $I_\Delta.$ We note that this descent data gives rise to the desired connection.
\end{proof}

\begin{theorem}\label{construction-of-triv}
    Let $X/S$ be a representable quasi-syntomic map of $\bF_p$-algebraic stacks. Let $\tilde{X}/W_2(S)$ be a lift. Then any strong Frobenius lift to $\tilde{X}$ gives rise to a splitting of the de Rham gerbe.
\end{theorem}

\begin{proof}
Let $\alpha_{dR}: X' \to B^2_{X'}(T_{X'/S}^{\sharp})$ the map classifying the de Rham torsor. Recall its composition with $B^2_{X'}(T_{X'/S}^{\sharp}) \to B^2_{X'}(T_{X'/S})$ classifies the torsor of lifts of $X/S$ to $W_2(S).$ Then $\tilde{X}/W_2(S)$ gives rise to a map $$X' \to \fib(B^2_{X'}(T_{X'/S}^{\sharp}) \to B^2_{X'}(T_{X'/S})) \simeq B_{X'}((T_{X'/S}/X')^{dR})$$
which we denote by $\alpha_1.$ Note $(X/S)^{dR} \xrightarrow{\nu_{X/S}} X' \xrightarrow{\alpha_1} B_{X'}((T_{X'/S}/X')^{dR})$ can be canonically refined to a map $(X/S)^{dR} \to B_{X'}(T_{X'/S}).$ Indeed, the composition $$(X/S)^{dR} \xrightarrow{\alpha_1 \circ \nu_{X/S}} B_{X'}((T_{X'/S}/X')^{dR}) \to B^2_{X'}(T_{X'/S}^{\sharp})$$ is canonically trivialized and we denote the corresponding map $(X/S)^{dR} \to B_{X'}(T_{X'/S})$ by $\alpha_{dR}^r.$ This is summarized in the following diagram.
\begin{equation}
\begin{tikzcd}
	X && {(X/S)^{dR}} && {X'} \\
	\\
	{B_{X'}(T_{X'/S})} && {(B_{X'}(T_{X'/S})/X')^{dR}} && {B^2_{X'}(T_{X'/S}^{\sharp})} && {B^2_{X'}(T_{X'/S})}
	\arrow["{\pi_{X/S}}", from=1-1, to=1-3]
	\arrow["\widetilde{\alpha_{c}^r}", from=1-1, to=3-1]
	\arrow["{\nu_{X/S}}", from=1-3, to=1-5]
	\arrow["{\alpha_{c}^r}", dotted, from=1-3, to=3-1]
	\arrow["\widetilde{\alpha_1}", from=1-3, to=3-3]
	\arrow["{\alpha_1}", dotted, from=1-5, to=3-3]
	\arrow["{\alpha_{dR}}"', from=1-5, to=3-5]
	\arrow["{\alpha_{dR}^{\gamma}}", from=1-5, to=3-7]
	\arrow["", from=3-1, to=3-3]
	\arrow[from=3-3, to=3-5]
	\arrow[from=3-5, to=3-7]
\end{tikzcd}
\label{diagram}
\end{equation}

\textit{Step 2.}
Applying \Cref{cor:obtain-E-sharp-torsor} to $\alpha_c^r$ with the trivialization of the underlying $F_{X/S}^*T_{X'/S}$-torsor on $X$  provided by the lift $\tilde{X}/W_2(S)$, we obtain $\alpha^\sharp: (X/S)^{dR} \to B_{X'}(T_{X'/S}^\sharp)$ refining $\alpha_c^r.$ Denote $q: \cL_{\tilde{X}} \to X$ to be the $F_{X/S}^* T_{X'/S}$-torsor on $X$ classified by $\widetilde{\alpha_c^r}.$ Then \Cref{cor:obtain-E-sharp-torsor} also implies $q^{dR}: (\cL_{\tilde{X}}/S)^{dR} \to (X/S)^{dR}$ is the $(T_{X'/S}/S)^{dR}$-torsor on $(X/S)^{dR}$ classified by $\widetilde{\alpha_1}.$

 \textit{Step 3. The map $\alpha^{\sharp}$ is an isomorphism.} Denote $\cL_{\tilde{X}}^c \to (X/S)^{dR}$ to be the $F_{X/S}^*T_{X'/S}$-torsor classified by $\alpha_c^r$. Explicitly, we have a diagram
\[\begin{tikzcd}
	{\cL_{\tilde{X}}} && {\cL_{\tilde{X}}^c} && {X'} \\
	\\
	X && {(X/S)^{dR}} && {(X/S)^{dR, \gamma}}
	\arrow[from=1-1, to=1-3]
	\arrow["q", from=1-1, to=3-1]
	\arrow[from=1-3, to=1-5]
	\arrow["{q^c}"', from=1-3, to=3-3]
	\arrow["s_{\tilde{X}}"', from=1-5, to=3-5]
	\arrow["{\pi_{X/S}}", from=3-1, to=3-3]
	\arrow[from=3-3, to=3-5]
\end{tikzcd}\]
where both squares are cartesian. Now let $s$ be a section of $\cL_{\tilde{X}} \to X$, i.e., a strong Frobenius lift to $\tilde{X}.$ Then $s$ gives rise to a section $s^{dR}: (X/S)^{dR} \to (\cL_{\tilde{X}}/S)^{dR}$ of the torsor classified by $\widetilde\alpha_1$. Moreover, by \Cref{p-curv-equivariant-relative} the map $s^{dR}$ is equivariant for the action of $B_{X'}(T_{X'/S}^\sharp)$. In particular, $s^{dR}$ descends to a map $X' \to [(\cL_{\tilde{X}}/S)^{dR}/B_{X'}(T_{X'/S}^\sharp)]$ on quotients. We will identify the right-hand side with $\cL_{\tilde{X}}^c.$ First, note the canonical map $\cL_{\tilde{X}}^c \to X'$ is the $T_{X'/S}^{dR}$-torsor classified by the map $\alpha_1$ in \eqref{diagram}. Since $\alpha_1 \circ \nu_{X/S} \simeq \widetilde{\alpha_1}$, we obtain a cartesian square
\[\begin{tikzcd}
	{(\cL_{\tilde{X}}/S)^{dR}} && {\cL_{\tilde{X}}^c} \\
	\\
	{(X/S)^{dR}} && {X'}
	\arrow[from=1-1, to=1-3]
	\arrow[from=1-1, to=3-1]
	\arrow[from=1-3, to=3-3]
	\arrow["{\nu_{X/S}}", from=3-1, to=3-3]
\end{tikzcd}\]
which realizes the top horizontal arrow as an $B_{X'}(T_{X'/S}^\sharp)$-torsor. Therefore, the map $s^{dR}$ descends to a map $X' \to \cL_{\tilde{X}}^c$ giving rise to a trivialization of $\alpha_1,$ and, in particular, to a trivialization of the de Rham gerbe.

\end{proof}

\begin{remark}
    During the proof of \Cref{construction-of-triv} we observed that a section $f \in  \cL_{\tilde{X}}(X)$ gives rise to a trivialization $(X/S)^{dR} \xrightarrow{f^{dR}} (\cL_{\tilde{X}}/S)^{dR}$ and upon reducing by $B_{X'}(T_{X'/S}^\sharp)$ it gives $s': X' \to \cL_{\tilde{X}}^c.$ In particular, the homotopy $f \circ \pi_{X/S} \simeq \pi_{\cL_{\tilde{X}/S}}\circ f$ gives rise to a homotopy $s' \circ F_{X/S} \simeq g \circ f$ where $g: \cL_{\tilde{X}} \to \cL_{\tilde{X}}^c$ is the canonical map. Since the composition $X \xrightarrow{f} \cL_{\tilde{X}} \xrightarrow{g} \cL_{\tilde{X}}^c \xrightarrow{} (X/S)^{dR}$ is naturally homotopic to $\pi_{X/S},$ we obtain a homotopy $\pi_{X/S} \simeq s \circ F_{X/S}$ where $s$ is the composition $X' \xrightarrow{s'} \cL_{\tilde{X}}^c \to (X/S)^{dR}.$
\end{remark}

\begin{remark}
    Let $X/S$ be a family with a lift $\tilde{X}/W_2(S)$ together with a strong Frobenius lift. The splitting of the de Rham gerbe from \Cref{construction-of-triv} gives rise to a map $X' \to B_{X'}T_{X'/S}^{\sharp}$. Indeed, this is because it gives a null-homotopy of the composition $$X' \xrightarrow{\alpha_1} (B_{X'}(T_{X'/S})/X')^{dR} \to B^2T_{X'/S}^\sharp $$
    where $\alpha_1$ is defined in \eqref{diagram}. We note it provides a Frobenius descent of the $F_{X/S}^*T_{X'/S}$-torsor on $X$ obtained from $(X/S)^{dR, \gamma} \simeq B_{X'}(T_{X'/S})$. We also note it is necessarily trivial. Indeed, there exists a (relative to $S$) Frobenius splitting of $\cO_{X'} \to F_{X/S, *}\cO_{X}.$ Therefore, for any $E$ on $X'$ the map $H^i(X', E) \to H^i(X, F_{X/S}^*E)$ is injective. 
\end{remark}


\begin{corollary}\label{local-OV}
    Let $X/S$ be a representable quasi-syntomic map of $\bF_p$-stacks. Given a lift $\tilde{X}/W_2(S)$ with a strong Frobenius lift, one obtains a symmetric monoidal equivalence $C_{f}: \MIC^{\cdot}(X/S)  \simeq \HIG^{\cdot}(X'/S).$ Moreover, given a lift $\tilde{S} \to \bZ/p^2$ of $S$, any lift $\tilde{X}_S$ of $X/S$ to $\tilde{S}$ provides an equivalence $C_f \simeq C_{\tilde{F}}$ where $C_{\tilde{F}}$ is from \cite[Theorem 2.11]{MR2373230} applied to $\tilde{F}: \tilde{X}_S \to \tilde{X} \times_{\tilde{S}} W_2(S)$ provided by \Cref{iso-of-torsor-deformation}.
\end{corollary}

\begin{remark}
    Let us list some properties of the equivalence $C_f: \MIC^{\cdot}(X/S) \simeq \HIG^{\cdot}(X'/S)$ from \Cref{local-OV}. 
    \begin{itemize}
        \item[1.] Recall the equivalence $\MIC^{\cdot}(X/S) \simeq \HIG^{\cdot}(X'/S)$ is obtained from the isomorphism of $X'$-stacks $(X/S)^{dR} \simeq B_{X'}(T_{X'/S}^{\sharp})$. In particular, for any $\cE \in \HIG^{\cdot}(X'/S)$ we obtain $\Hig(\cE) \simeq \dR(C_f(\cE))$ in $\cD_{qc}(X').$ 
        \item[2.] Recall that one has $s_f \circ F_{X/S} \simeq \pi_{X/S}$ where $s_f$ is a section of $\nu_{X/S}: (X/S)^{dR} \to X'$ defined by $f.$ In particular, for $\cM \in \HIG^{\cdot}(X'/S)$, the underlying $\cO_X$-module of $C^{-1}_f(\cM)$ is $F_{X/S}^* \cM.$
        \item[3.] For $\cM \in \MIC^{\cdot}(X/S)$, the Higgs field of $C_{f}(M)$ is the image of the $p$-curvature of $\cM.$ Indeed, the isomorphism $C: B_{X'}(T_{X'/S}^{\sharp}) \simeq (X/S)^{dR}$ is $B_{X'}(T_{X'/S}^{\sharp})$-equivariant, i.e., one has a commutative square 
\[\begin{tikzcd}
	{(X/S)^{dR} \times_{X'} B_{X'}(T_{X'/S}^{\sharp})} &&&& {(X/S)^{dR}} \\
	{B_{X'}(T_{X'/S}^{\sharp}) \times_{X'} B_{X'}(T_{X'/S}^{\sharp})} &&&& {B_{X'}(T_{X'/S}^{\sharp})}
	\arrow["a", from=1-1, to=1-5]
	\arrow["{C \times \id}", from=2-1, to=1-1]
	\arrow["m"', from=2-1, to=2-5]
	\arrow["C"', from=2-5, to=1-5].
\end{tikzcd}\]
Thus, combining \Cref{rem:p-curv=action} and \Cref{rem:higgs-field=action-of-sharp}, the claim follows.
    \end{itemize}
\end{remark}

\subsection{Logarithmic Local Cartier equivalence}
For this subsection let $X$ be a smooth scheme over $S/\bF_p$ and $D$ a relative simple normal crossing divisor with irreducible components $D_1, \cdots, D_n.$ This gives rise to a map $X \to (\bA_S^1/\bG_m)^n =: S_n.$ The goal of this subsection is to apply the Cartier transform to the family $X \to S_n$. 
\begin{remark}
    Assume $(X, D)$ is $F$-liftable in the sense of \cite[\S 3]{MR4269423}. Such a lift gives rise to a lift of $X \to S_n.$ Moreover, compatibility of Frobenii on $X$ and $S_n$ is equivalent to $\tilde{F}^* (\tilde{D})= p \tilde{D}$ where $\tilde{F}: \tilde{X} \to \tilde{X}$ is a lift of the absolute Frobenius. Therefore, one can apply \Cref{local-OV}.
\end{remark}
Assuming $(X, D)$ is $F$-liftable, we obtain $\MIC^{\cdot}(X/S_n) \simeq \HIG^{\cdot}(X'/S_n).$ Let us interpret both categories in terms of the pair $(X, D).$ Let us first note that $X'$ is the $p$-th multiroot stack of $(X, D).$

\begin{remark}
   The category $ \cD_{qc}((X/S_n)^{dR})$ is identified with the derived category of logarithmic connections, i.e., its objects are $\cE \in \cD_{qc}(X)$ with a connection $\cE \to \cE \otimes \Omega^1_X(\log D)$ such that $\cH^i(\cE)$ has nilpotent $p$-curvature, see \cite[Theorem 5.11]{barz2025logarithmicrhamstacksnonabelian}. We denote this category by $\MIC^{\cdot}(X, D).$
\end{remark}
Now let us interpret the category $\HIG^\cdot(X'/(\bA^1/\bG_m)).$ First, we recall that the category $\QCoh(X')$ has a concrete interpretation in terms of (multi)-$1/p$-parabolic sheaves on $X$ with respect to $D$, see \cite{MR2504408}, or \cite[Theorem 6.1]{MR2964607}. For example, if $X$ has only one smooth irreducible component, then a $1/p$-parabolic sheaf structure on a vector bundle $E$ is a filtration $$E(-D)= E_1 \mono E_{(p-1)/p} \mono \cdots \mono E_{1/p} \mono E_0 = E.$$
We refer to \cite[Example 5.14]{barz2025logarithmicrhamstacksnonabelian} which explains how one can see the necessity of parabolic structures for the logarithmic Cartier descent.

\begin{remark}
    Since $X' = X \times_{S_n, F_{S_n}} S_n$ is the $p$-multiroot stack of $(X, D)$, the category $\QCoh(X')$ is identified with the category of $1/p$-parabolic quasi-coherent sheaves on $X.$ Therefore, $\HIG(X'/(\bA^1/\bG_m))$ is the category consisting of pairs $(E, \theta_E)$ where $E$ is a $1/p$-parabolic sheaf on $X$ with respect to $D$ and $\theta_E: E \to E \otimes_{\cO_{X}} \Omega^1_{X/S}(\log(D))$ a map of $1/p$-parabolic sheaves such that $\theta_E \wedge \theta_E=0$; here $\Omega^1_{X/S}(\log(D))$ is endowed with the trivial structure of a $1/p$-parabolic sheaf. Indeed, note $\Omega^1_{X'/(\bA^1/\bG_m)} = \nu^* \Omega^1_{X/(\bA^1/\bG_m)}$ where $\nu: X' \to X$ is the canonical projection. Recall $\Omega^1_{X/(\bA^1/\bG_m)} \simeq \Omega^1_{X}(\log D)$ and, under the equivalence of $\QCoh(X')$ with $1/p$-parabolic sheaves, $\nu^*$ endows a quasicoherent sheaf with the trivial $1/p$-parabolic structure. 
    We denote $\operatorname{ParHIG}^{1/p, \cdot} (X, D)$ to be the category of nilpotent Higgs modules on $X'/S_n.$
\end{remark}
\begin{corollary}\label{log-F-lift}
        Let $(X, D)$ be an snc pair with smooth $X$. Any $F$-lift in the sense of \cite[\S 3]{MR4269423} gives rise to a symmetric monoidal equivalence $\MIC^{\cdot}(X, D) \simeq \operatorname{ParHIG}^{1/p, \cdot} (X, D)$ between the category of logarithmic connections with nilpotent $p$-curvature and the category of Higgs bundles in $1/p$-parabolic sheaves with nilpotent Higgs field.
\end{corollary}
\begin{proof}
    If $D$ has $n$ smooth components $D_i,$ we obtain a representable smooth map $X \to (\bA^1/\bG_m)^n$. By \Cref{torsor-for-stacks}, its relative de Rham stack is a gerbe over $X'$ which is the $p$-multiroot of $D$. The $F$-lift gives rise to a splitting of this gerbe by \Cref{construction-of-triv}. Then \Cref{local-OV} gives rise to the desired equivalence.
\end{proof}

\begin{corollary}\label{log-just-lift}
    Let $(X, D)$ be an snc pair. Then any snc lift $(\tilde{X}, \tilde{D})$ to $\bZ/p^2$ gives rise to an equivalence $\MIC_{\le p-1} (X, D) \simeq \operatorname{ParHIG}^{1/p, \cdot}_{\le p-1} (X, D)$.
\end{corollary}

\subsection{Equivariant Cartier Equivalence}
For this subsection let $X$ be a smooth scheme over $\bF_p$ equipped with an action of an affine group scheme $G$. In this case, there are two $D$-modules categories of interest: the category $\MIC_G(X)$ of \textit{weakly} $G$-equivariant $D$-modules and the category $\MIC(X/G)$ of \textit{strongly} $G$-equivariant $D$-modules. The former category is identified with quasi-coherent sheaves on $X^{dR}/G$ and the latter with quasi-coherent sheaves on $X^{dR}/G^{dR}.$
\begin{remark}
    Recall from \Cref{lem:dR-preserves-et-covers} that $(X/G)^{dR}=X^{dR}/G^{dR}$ as \'etale sheaves. In particular, the category of weakly $G$-equivariant $D$-modules is the category of $D$-modules on $X/G,$ i.e., the category $\cD_{qc}((X/G)^{dR}).$
\end{remark}
We note that the category of strongly $G$-equivariant $D$-modules is identified with the category of quasi-coherent sheaves on the relative de Rham stack of $X/G \to BG.$ This follows from the next lemma.
\begin{lemma}
    Consider the natural map $f: X/G \to BG.$ Then there is a natural identification $(X/G \xrightarrow{f} BG)^{dR} \simeq X^{dR}/G.$
\end{lemma}
\begin{proof}
    Recall the relative de Rham stack of $f$ is the fiber product of 
\[\begin{tikzcd}
	&&& {(X/G)^{dR}} \\
	\\
	BG &&& {(BG)^{dR}}
	\arrow["f", from=1-4, to=3-4]
	\arrow["{\pi_{BG}}", from=3-1, to=3-4].
\end{tikzcd}\]
Moreover, by \Cref{lem:dR-preserves-et-covers} one has $(BG)^{dR} \simeq BG^{dR}$ and $\pi_{BG} \simeq B\pi_{G}.$  Thus, the fiber product is identified with the quotient of $X^{dR}$ by $G$ with the action defined along the map $\pi_{G}: G \to G^{dR}$, which completes the proof.
\end{proof}
\begin{remark}
     Note the sheaf $\Omega^1_{f}$ of relative K\"ahler differentials of $f: X/G \to BG$ is given by $\Omega^1_{X}$ with its natural $G$-equivariant structure. In particular, a relative Higgs module for $X/G \to BG$ amounts to a weakly $G$-equivariant sheaf $E$ on $X$ with an $\cO_X$-linear map $E \to E \otimes \Omega^1_{X}$ of $G$-equivariant sheaves on $X$ satisfying $\theta_E \wedge \theta_E=0.$ We denote this category by $\HIG^\cdot_{G}(X).$
\end{remark}
As above, let $G$ be a group acting on $X.$ Consider the twisted action given by $$G \times X \xrightarrow{F_G \times \operatorname{Id}_X} G \times X \xrightarrow{a} X $$
and denote $\HIG^\cdot_{\phi(G)}(X)$ to be the category of nilpotent Higgs modules on the family $X/G \to BG$ for this twisted action.
\begin{corollary}\label{cor:G-equiv-local}
    Assume both $X$ and $BG$ are $F$-liftable. Moreover, assume the action of $G$ on $X$ also lifts to $\bZ/p^2.$ Then one obtains a symmetric monoidal equivalence $\MIC^{\cdot}_G(X) \simeq \HIG^{\cdot}_{\phi(G)}(X).$
\end{corollary}
\begin{proof}
    By \Cref{torsor-for-stacks} the relative de Rham stack of $X/G \to BG$ is a gerbe over $X'$ which is identified with the quotient $X/\phi(G)$. The assumptions imply it admits a section by \Cref{construction-of-triv}. Thus, the statement follows from \Cref{local-OV}.
\end{proof}
\begin{corollary}\label{cor:G-equiv-global}
    Assume the action of $G$ on $X$ lifts to $\bZ/p^2$. Then one obtains the equivalence $\MIC_{G, \le p-1}(X) \simeq \HIG_{\phi(G), \le p-1}(X).$
\end{corollary}
\begin{remark}
The appearance of Higgs modules with respect to the Frobenius-twisted action of $G$ in \Cref{cor:G-equiv-local} may be explained by the following observation.
   Let $a: G \times X \to X$ be the action map and denote $a^F: G \times X \xrightarrow{F_G \times \id} G \times X \xrightarrow{a} X$ to be the Frobenius-twisted action. Let $E$ be a vector bundle endowed with the $G$-equivariant structure for $a^F.$ That is, we have $(a^F)^* E \simeq \pr^*E.$ This gives rise to the canonical $G$-equivariant structure on $F_{X}^*E$ for the action $a.$ Indeed, this follows from $F_X \circ a = a \circ (F_G \times F_X) = a^F \circ (\id_G \times F_X).$ In particular, in the assumptions of \Cref{cor:G-equiv-local}, to define a strongly $G$-equivariant nilpotent $D$-module, it is enough to have a nilpotent Higgs module which only has a $\phi(G)$-equivariant structure. Namely, the functor sends such $(E, \theta_E)$ to $(F^*E, \nabla_{can}+\zeta_{\tilde{F}}(\theta_E))$ where $F^*E$ is endowed with its natural $G$-equivariant structure.
\end{remark}

\subsection{Decompleted Cartier transform}

\begin{definition}
    Denote $\bG_a^{dR, n}$ to be the $\bG_a$-algebra stack obtained from the square-zero extension $F_*\bG_a^{\sharp}[1] \to \bG_a^{dR} \to F_*\bG_a$ by pushing out along the map $\bG_a^{\sharp} \to \alpha_{p^n}.$
\end{definition}
\begin{definition}
\begin{itemize}
    \item[(1)] For $X/\bF_p$ denote $X^{dR, n}$ to be the $\bF_p$-stack obtained by the transmutation of $\bG_a^{dR, n},$ i.e., $X^{dR, n}(R)=X(\bG_a^{dR, n}(R)).$
    \item[(2)] For $X/S$ over $\bF_p$ denote $(X/S)^{dR, n} = X^{dR, n} \times_{S^{dR, n}} S$ in the category of derived stacks where $S \to S^{dR, n}$ is obtained by transmutation of $\bG_a \xrightarrow{\pi} \bG_a^{dR} \to \bG_a^{dR, n}.$
\end{itemize}
\end{definition}
Let us give an explicit quasi-ideal model for $\bG_a^{dR, n}.$
\begin{remark}
    Define $A_n=\Cone(F_*\bG_a \xrightarrow{VF^n} F_*^{n}W_2)$ a $W$-algebra stack. Note $\pi_0(A_n)=F_*^n\bG_a$ and $\pi_1(A_n)=F_*\alpha_{p^n}$ as sheaves of $W$-modules. Note there is a map $f_{n}: A_n \to F^{n-1}_*\bG_a^{dR, F}$ given by a map of quasi-ideals:
\[\begin{tikzcd}
	{F_*\bG_a} &&& {F_*^nW_2} \\
	\\
	{F^n_*\bG_a} &&& {F^n_*W_2}
	\arrow["{VF^n}", from=1-1, to=1-4]
	\arrow["{F^{n-1}}"', from=1-1, to=3-1]
	\arrow[equals, from=1-4, to=3-4]
	\arrow["p"', from=3-1, to=3-4]
\end{tikzcd}\]
which induces an isomorphism on $\pi_0$ and on $\pi_1$ is $F^{n-1}: F_*\alpha_{p^n} \to F_*\alpha_p$. 
    
\end{remark}

\begin{definition}
Let $\bG_a^{dR, n,*}$ be the fiber product of $f_n: A_n \to F_*\bG_a^n$ and $F^{n-1}: F_*\bG_a \to F^n_*\bG_a.$
\end{definition}
\begin{lemma}
    The $\bG_a$-algebra stacks $\bG_a^{dR, n}$ and $\bG_a^{dR, n, *}$ are isomorphic.
\end{lemma}

\begin{proof}
    Consider the map $\bG_a^{dR} \to A_n$ of $\bG_a$-algebra stacks given by 
\[\begin{tikzcd}
	{F_*W} &&& {F_*W} \\
	\\
	{F_*\bG_a} &&& {F_*^nW_2}
	\arrow["p", from=1-1, to=1-4]
	\arrow["R"', from=1-1, to=3-1]
	\arrow["{F^{n-1}\circ R}", from=1-4, to=3-4]
	\arrow["{VF^{n}}"', from=3-1, to=3-4]
\end{tikzcd}\]
and note that on $\pi_0$ it induces $F^{n-1}:F_*\bG_a \to F_*^n\bG_a$. Therefore, it gives rise to a map $f:\bG_a^{dR} \to \bG_a^{dR, n, *} = A_n \times_{F^n_*\bG_a} F_*\bG_a$ where the map $\bG_a^{dR} \to F_*\bG_a$ comes from $\pi_0(\bG_a^{dR})=F_*\bG_a.$ Moreover, the map $\pi_1(\bG_a^{dR}) \to \pi_1(\bG_a^{dR, n, *})$ is identified with the canonical map $F_*\bG_a^\sharp \to F_*\alpha_{p^n}$ since it is true for $\bG_a^{dR} \to A_n$ and $\pi_1(\bG_a^{dR, n, *}) = \pi_1(A_n)$ under the natural map. Therefore, $f$ factors through the map $\bG_a^{dR, n} \to \bG_a^{dR, n, *}$ which is an isomorphism from the previous discussion.
\end{proof}

\begin{remark}
    If $X/S$ is lci, then $(X/S)^{dR, n}$ is the pushout of $\bG_a^{\sharp} \otimes_{\bG_a} L_{X'/S}^{\vee}[1]$-torsor on $X'$ along the map induced by $\bG_a^{\sharp} \to \alpha_{p^n}.$ In particular, there exists a map $\nu_{X/S}^{dR, n}: (X/S)^{dR, n} \to X'$ exhibiting the source as an fppf-$ L_{X'/S}^{\vee}[1] \otimes_{\bG_a} \alpha_{p^n}$-torsor over the target.
\end{remark}
\begin{remark}\label{XdRn-vs-frobenius-power}
    Let $X/S$ be smooth and $S$ is reduced. Then pushforward of $\alpha_{p^n}$ from the fppf site to Zariski is given by $F_{X', *}^{n}\cO_{X'}/\cO_{X'}[-1]$. Thus, $(X/S)^{dR, n} \to X'$ defines a Zariski torsor for $T_{X'/S} \otimes F_{X', *}^n \cO_{X'}/\cO_{X'}.$ Under the canonical map $T_{X'/S} \otimes F_{X', *}^n \cO_{X'}/\cO_{X'} \to T_{X'/S}[1]$ this torsor is sent to the gerbe of lifts of $X$ to $W_2(S)$. In particular, any lift $\tilde{X}/W_2(S)$ gives rise to a torsor for $T_{X'/S} \otimes_{\cO_{X'}} F_{X', *}^n \cO_{X'}$-torsor on $X'$ refining $(X/S)^{dR, n} \to X'.$
\end{remark}
\begin{remark}\label{torsor-of-lifts-of-power-of-frob}
    Let $f:X \to S$ be a representable map of smooth algebraic stacks over $\bF_p.$ Let $\tilde{X}/W_2(S)$ be a lift. Deforming $F_{X/S}^n: X \to X^{(n)}$ to a $W_2(X)$-linear map $\tilde{X} \to \tilde{X}^{(n)}$ is governed by a map $F_{X/S}^{n, *}L_{X^{(n)}/S} \to f^*F_{S, *}\cO_{S}[1] \simeq \nu_* \cO_{X'}[1],$ see \Cref{def:thry}. Note $\nu^*F_{X/S}^{n ,*}L_{X^{(n)}/S} \simeq F_{X'}^{n ,*} L_{X'/S}.$ Therefore, we obtain a torsor for $T_{X'/S} \otimes_{\cO_{X'}} F^n_{X', *}\cO_{X'}.$
\end{remark}
\begin{lemma}
    Let $X/S$ be smooth with reduced $S.$ Let $\tilde{X}/W_2(S)$ be a flat lift of $X.$ Then Zariski $T_{X'/S} \otimes_{\cO_{X'}} F_{X', *}^n \cO_{X'}$-torsor on $X'$ from \Cref{XdRn-vs-frobenius-power} is isomorphic to the torsor of lifts of $F_{X/S}^n$ to $\tilde{X}/W_2(S)$ from \Cref{torsor-of-lifts-of-power-of-frob}. 
\end{lemma}

\begin{proof}
    The case $n=1$ is \Cref{A1-split-scheme-lift-frob}. The general case follows since both classes in $H^1(X', F_{X'}^{n, *}T_{X'/S})$ are obtained from the corresponding classes for $n=1$ via the pullback along $F_{X'}^{n-1}.$ Let us check this for both classes.

    \textit{Step 1.} Denote $\ob_n$ to be the obstruction class of deforming $F_{X/S}^n: X \to X^{(n)}$ to $\tilde{X}.$ Explicitly, it is a map $$F_{X/S}^{n, *} L_{X^{(n)}/S} \to \pi^*F_{S,*}\cO_S[1] \simeq \nu_* \cO_{X'}[1]$$ where $\nu: X' \to X$ is the canonical projection. By adjunction, it corresponds to a map $$F_{X/S}^{n-1, *}L_{X^{(n)}/S} \to F_{X', *}\cO_{X'}[1] $$
    and we note that the left-hand side is canonically identified with $F_{X'}^{n-1, *}L_{X^{(n-1)}/S}.$ Thus, $\ob_n$ defines a class in $H^1(X', F_{X'}^{n, *}T_{X'/S}).$ Moreover, it is equal to $F_{X/S}^* \ob_{n-1}'$ where $\ob_{n-1}' \in H^1(X'', F_{X''}^*T_{X''/S})$ is the obstruction of deforming $F_{X/S}^{n-1}: X' \to X^{(n)}$ to $\tilde{X}'.$ Indeed, this follows from functoriality of obstruction classes and $dF_{X/S}=0.$ Therefore, we learn that $F_{X'}^{n-1, *}\ob_1 = \ob_n.$

    \textit{Step 2.} Denote $\widehat \ob_n \in H^1(X', F_{X'}^{n, *}T_{X'/S})$ to be the class obtained from the $B_{X'}(T_{X'/S} \otimes_{\bG_a} \alpha_{p^n})$ as in \Cref{XdRn-vs-frobenius-power}. By construction, one has $F_{X'}^{n-1, *}(\widehat \ob_{1})=\widehat \ob_{n}.$ Indeed, pushing forward the map $\alpha_p \to \alpha_{p^{n+1}}$ of fppf sheaves to the Zariski site one obtains the map $$F_{X', *}\cO_{X'}/\cO_{X'}[-1] \to F^n_{X', *}\cO_{X'}/\cO_{X'}[-1]$$
    induced by $F_{X'}^{n-1}: F_{X', *}\cO_{X'} \to F^n_{X', *}\cO_{X'}.$ Since $(X/S)^{dR, n}$ is obtained from $(X/S)^{dR, 1}$ by pushing forward along the map induced by $\alpha_{p} \to \alpha_{p^n},$ the claim follows.

\end{proof}

Now we will identify the category of quasi-coherent sheaves on $\widehat{(X/S)^{dR}}$ with modules over the algebra $D_{X/S}^{\gamma}$ defined as  $$D_{X/S}^{\gamma} := F_{X/S, *} D_{X/S} \otimes_{S T_{X'/S}} \Gamma T_{X'/S}.$$
We denote the category of modules over $D_{X/S}^{\gamma}$ by $\MIC^{\gamma}(X/S).$
In particular, the category $\MIC_{\gamma}(X/S)$ may be viewed as the decompleted version of $\MIC_{\gamma}^{\cdot}(X/S).$
\begin{lemma}
    For smooth $X/S$ the category $\QCoh(\widehat{(X/S)^{dR}})$ is identified with $\MIC_{\gamma}(X/S)$, the category of $D_{X/S}^{\gamma}$-modules.
\end{lemma}
\begin{proof}
The proof is analogous to \Cref{gamma-crystals-are-gamma-d-modules} with the only exception: a quasicoherent sheaf on $B_{(X/S)^{dR}}(F_{X/S}^*\hat T_{X'/S})$ is also a pair $\Gamma_{\cO_X} (F_{X/S}^* T_{X'/S}) \to \cE n d(\cE)$ but there is no nilpotence condition on the map. Indeed, the Cartier dual to $B_{X'}( \hat T_{X'/S})$ is $\bT^{*, \gamma}_{X'/S}$. The rest of the proof is verbatim.
\end{proof}

    Denote $\HIG_{\gamma}(X'/S)$ to be the category of PD-Higgs modules, i.e., to be $\QCoh(B_{X'}(\hat T_{X'/S})).$ By Cartier duality, as an abelian category it is equivalent to $\QCoh(\bT^{*, \gamma}_{X'/S})$. That is, giving an object of $\MIC_{\gamma}(X'/S)$ amounts to giving $\cE \in \QCoh(X')$ with an action of the divided power algebra $\Gamma T_{X'/S}$. The monoidal structure on $\HIG_{\gamma}(X'/S)$ corresponds to the convolution monoidal structure on $\bT_{X'/S}^{*, \gamma}$ which is explicitly given in \cite[\S 2.3]{MR2373230}. Namely, if $\xi$ is a local section of $T_{X'/S},$ then $\cE$ inherits an action by operators $\psi_{\xi^{[n]}}$. Then, for $\cE_1, \cE_2 \in \MIC_{\gamma}(X'/S)$ the action of $\psi_{\xi^{[n]}}$ on  $\cE_1 \otimes \cE_2$ is given by $\sum_{i=0}^n \psi_{\xi^{[i]}} \otimes \psi_{\xi^{[n-i]}}$.

\begin{corollary}\label{cor:rel-dR-n-splits-provided-lift-of-nth-power}
    For a quasi-syntomic $X/S$ we have a $B\hat T_{X'/S}$-torsor $\hat \nu_{X/S}:\widehat{(X/S)^{dR}} \to X'$. By construction, one has an $X'$-linear map $(X/S)^{dR, n} \to \widehat{(X/S)}^{dR}$ for any $n.$ In particular, any splitting of $(X/S)^{dR, n}$ gives rise to a splitting of $\widehat{(X/S)}^{dR}.$ Thus, any flat lift of $\tilde{X}/W_2(S)$ of $X/S$ together with some power of $F_{X/S}$ induces a symmetric monoidal equivalence $\MIC^{\gamma}(X/S) \simeq \HIG^{\gamma}(X'/S).$
\end{corollary}

Although $F$-liftability is extremely rare, there are more examples of $X/S$ for which $F_{X/S}^n$ admits a lift for some $n$. For example, any abelian variety $X/\mathbf{F}_p$ admits a lift to $\mathbf{Z}/p^2$ together with a lift of some power $F_A^n$. More generally, we have the following result.

\begin{lemma}\label{ab-var-lifts}
    Given an abelian scheme $\cA/S$. Assume $S$ is regular, $R^1f_*\cO_{\cA}$ is a trivial vector bundle on $S$, $H^i(S, \cO_S)$ is finite-dimensional for $i=0$ and $i=1$, and $H^2(S, \cO_S)=0.$ There exists a flat lift $\tilde{\cA}/W_2(S)$ such that $F_{\cA/S}^k$ lifts to $\tilde{\cA}$ for some $k > 0.$
\end{lemma}
\begin{proof}
A standard argument shows that there exists a deformation of $\cA/S$ to $W_2(S).$ Denote the canonical map $\cA^{(n)} \to \cA^{(n-1)}$ by $\nu_{n}.$ Denote the obstruction of lifting of $F_{\cA/S}^n: \cA \to \cA^{(n)}$ by $\ob_n$ and the obstruction of lifting of $F_{A/S}: \cA^{(n-1)} \to \cA^{(n)}$ by $\ob_1^{(n)}.$ Explicitly, it is a map $\ob_n: F_{\cA/S}^{n, *}L_{\cA^{(n)}/S} \to \nu_{1, *}\cO_{A'}[1]$ and $\ob_{1}^{(n)}: F_{\cA/S}^* L_{\cA^{(n)}/S} \to \nu_{n, *}\cO_{\cA^{(n)}}[1]$. By functoriality, one has $F_{A/S}^{n-1, *}\ob_1^{(n)} = \ob_n$ under $F_{\cA/S}^{n-1, *}\nu_{n, *}\cO_{\cA^{(n)}} = \nu_{1, *}\cO_{\cA'}.$ Consider $$F_{\cA/S}^{n-1, *}: H^1(\cA^{(n-1)}, F_{\cA/S}^*T_{\cA^{(n)}/S} \otimes \nu_{n, *} \cO_{A^{(n)}}) \to H^1(\cA, F_{\cA/S}^{n, *} T_{\cA/S} \otimes \nu_{1, *} \cO_{\cA'}).$$
Using projection formula, we rewrite it as
$$F_{\cA/S}^{n-1, *}: H^1(\cA^{(n)}, F_{\cA^{(n)}}^* T_{\cA^{(n)}/S}) \to  H^1(\cA', F^{n, *}_{\cA'} T_{\cA'/S}) .$$
Recall that if $\widehat{\cA}$ is another lift, then the obstruction class of deforming $F_{\cA/S}^n: \cA \to \cA^{(n)}$ to $\widehat{\cA}$ differs from $\ob_n$ by $F_{\cA/S}^{n, *}t$ where $t: L_{\cA^{(n)}/S} \to \nu_* \cO_{\cA^{(n+1)}}[1]$ is the element measuring the difference between $\widetilde{\cA}^{(n)}$ and $\widehat{\cA}^{(n)}.$ Recall $T_{\cA/S} = e^* T_{\cA/S} \otimes_{\cO_S} \cO_{\cA} \simeq \cO_{\cA}^{\oplus g}$. Note the images of $F_{\cA/S}^{n-1}: H^1(\cA^{(n)}, \cO_{\cA^{(n)}}) \to H^1(\cA', \cO_{\cA'})$ stabilize. Indeed, by the assumption the Leray spectral sequence gives $H^1(\cA, \cO_{\cA}) = H^1(S, \cO_S) \oplus H^0(S, \cO_S)^{\oplus g}$ which is finite dimensional by the assumption. By Kunz's theorem $F_S$ is flat, thus flat base change implies $H^1(\cA^{(n)}, \cO_{\cA^{n}})$ is isomorphic to $H^1(S, \cO_S) \oplus H^0(S, \cO_S)^{\oplus g}$.  Thus, we can find $x \in H^1(\cA^{(n)}, T_{\cA^{(n)}/S} \otimes \nu_{n, *}\cO_{\cA^{(n+1)}})$ such that $F_{A/S}^{n ,*}x = F_{\cA/S}^{n-1, *} \ob_1^{(n)}.$ To finish the proof, replace $\widetilde{\cA}$ by $\widehat{\cA}$ which is obtained via such $x$.

\end{proof}

\begin{question}
   Let $X/\bF_p$ be a smooth projective variety that admits a lift to $\bZ/p^2$ together with some power of Frobenius. Is it necessarily true that, after a finite \'etale cover, it is a toric fibration over its Albanese variety? This question is motivated by the main conjecture of \cite{MR4269423}.

\end{question}

\subsection{Cohomology comparison}

Given $X/S$ with a lift $\tilde{X}/W_2(S),$ we obtain a splitting $(X/S)^{dR, \gamma} \simeq B_{X'} (T_{X'/S})$. For  $E \in \MIC^{\cdot}_{\le p-1}((X/S)^{dR})$, we obtain the corresponding Higgs module $C_{\tilde{X}}(E) \in \HIG_{\le p-1}(X'/S) $ and we would like to compare the complexes $\dR(E)$ and $\Hig(C_{\tilde{X}}(E))$ in $\cD_{qc}(X')$. The main ingredient is the conjugate-filtered version of $(X/S)^{dR, \gamma}.$ Thus, we recall its definition.
\begin{remark}
    We recall the ring stack $\bG_a^{dR, c}$ from \cite[Construction 2.7.8]{bhatt-lecture-notes}. By transmutation, for any stack $X/\bF_p$ one obtains $X^{dR, c}.$  Let us recall the main properties for a smooth $X/k.$
    \begin{itemize}
        \item There exists a map $\nu_{X/k}^c: X^{dR, c} \to \bA^1_+/\bG_m \times X'$ such that fiber over $t=1$ is $X^{dR}$ and fiber over $t=0$ is $ B_{X'}(T_{X'}^\sharp).$ Moreover, when $X/k$ is smooth, the map $\nu_{X/k}^c$ exhibits the source as a gerbe for $T_{X'/k}^\sharp\{1\} = T_{X'/k} \otimes_{\bG_a} \bG_a^\sharp\{1\}$ over the target. Here $\bG_a^\sharp\{1\}$ is a group scheme over $\bA^1_+/\bG_m$ whose underlying group scheme is $\bG_a^\sharp$ and the action of $\bG_m$ is the standard one.
        \item There exists a map $\pi_{X}^c: X \times \bA^1_+/\bG_m \to X^{dR, c}$ which is an fpqc cover for smooth $X.$
        \item The pushforward of $\cO_{X^{dR, c}}$ along $\pi_{X/k}^c$ is identified with $\Fil_*^{\conj} R\Gamma(X, \Omega^\bullet_{X/k})$ as a filtered object.
        \item  The category of vector bundles on $X^{dR, c}$ is identified with the category of triples $(E, \Fil_{\bullet}, \nabla)$ where $(E, \nabla)$ is a vector bundle with a flat connection on $X$ and $\Fil_\bullet$ is a filtration by vector subbundles preserved by the connection such that $\gr_i$ has vanishing $p$-curvature.
    \end{itemize}
\end{remark}

First, we define a relative version of the conjugate-filtered de Rham stack.
\begin{definition}
   Let $f: X \to S$ be a map over $\bF_p$. Define $(X/S)^{dR, c}$ to be the fiber product of $f^{dR, c}: X^{dR, c} \to S^{dR, c}$ and $\pi_{S}^{c}: S \times \bA^1/\bG_m \to S^{dR, c}.$
\end{definition}

\begin{definition}\label{df:weights-for-filtered-conj}
    Define $\cD_{qc, [a, b]} ((X/S)^{dR, c})$ to be the full subcategory of $\cD_{qc} ((X/S)^{dR, c})$ spanned by objects $M$ such that under the identification $\pi_{X/S, c}^* M = (\cdots \to M^i \to M^{i+1} \to \cdots)$ the map $M^{i} \to M^{i+1}$ is a quasi-isomorphism for any $i \ge b$ and $M^{i}$ is acyclic for $i < a.$
\end{definition}

We also define a conjugate-filtered analogue of $(X/S)^{dR, \gamma}.$
\begin{definition}\label{df:gamma-filtered}
    Define $(X/S)^{dR, \gamma, c}$ to be the pushout of $X^{dR, c} \to X' \times \bA^1/\bG_m$ via $B_{X'} (T_{X'/S}^{\sharp}\{1\}) \to B_{X'} (T_{X'/S}\{1\}).$
\end{definition}
\begin{remark}
    The fiber of $\nu_{X/S}^{\gamma, c}: (X/S)^{dR, \gamma, c} \to X' \times \bA^1_+/\bG_m$ over $t=1$ is $(X/S)^{dR, \gamma}.$
\end{remark}

\begin{definition}\label{df:weights-for-gamma-filtered}
    Define $\cD_{qc, [a,b]}((X/S)^{dR, \gamma, c})$ to be the full subcategory of $\cD_{qc}((X/S)^{dR, \gamma, c})$ spanned by those objects that under the pullback $\cD_{qc}((X/S)^{dR, \gamma, c}) \to \cD_{qc}((X/S)^{dR, c})$ land in the subcategory $\cD_{qc, [a, b]}((X/S)^{dR, c})$.
\end{definition}
We will show that the pullback along $(X/S)^{dR, c} \to (X/S)^{dR, \gamma, c}$ induces an equivalence of subcategories with weights $[0, p-1].$ Essentially, it amounts to showing a similar statement for the $\bA^1_+/\bG_m$-map $B(\bG_a^\sharp\{1\}) \to B(\bG_a\{1\})$. To prove it, we need a better understanding of $\cD_{qc}(B\bG_a\{1\}).$ We warn the reader that the Cartier equivalence fails for the unbounded derived category even for $B\bG_a$. The following is true.

\begin{theorem}\label{thm:BG_a-loaclly-small-modules}(\cite[Theorem C.1]{MR3995721})
\begin{enumerate}
    \item One has the equivalence $$\QCoh(B\bG_a) = \Mod^{ln} \Big(\bF_p [x_1, x_2, ..]/(x_i^p =0)\Big)$$ where the category on the right is a full subcategory of $ \Mod \Big(\bF_p [x_1, x_2, ..]/(x_i^p =0)\Big)$ consisting of locally nilpotent modules, i.e., modules such that every element is killed by all but finitely many variables.
    \item 
    The functor $\cD(\QCoh(B\bG_a)) \to \cD_{qc}(B\bG_a)$ exhibits the source as the left completion of the target. 
\end{enumerate}
\end{theorem}
Let $E$ be a vector bundle on an algebraic stack $X$.
\begin{definition}
    Define a category $\QCoh(\bV(E))^{ln}$ to be a full subcategory of $\QCoh(\bV(E)) \simeq \Mod_{\QCoh(X)} (\Gamma_{X} (E^{\vee}))$ which consists of those modules $M$ such that every element is killed by $\Gamma^{\ge n} (E^{\vee})$ for some $n$.
\end{definition}
The following lemma is a straight generalization of \Cref{thm:BG_a-loaclly-small-modules}.
\begin{lemma}
    There is an equivalence $\QCoh(B_{X}(E)) \simeq \QCoh(\bV(E)^{\sharp})^{ln}$. Moreover, the natural functor $\cD(\QCoh(B_{X} (E))) \to \cD_{qc} (B_{X}(E))$ realizes the target as the left-completion of the source.
\end{lemma}


\begin{lemma}\label{gamma-and-dr-equiv-in-small-range}
    For smooth $X/S$ the map $(X/S)^{dR, c} \to (X/S)^{dR, \gamma, c}$ induces an equivalence $\cD_{qc, [a, b]} ((X/S)^{dR, \gamma, c}) \to \cD_{qc, [a, b]} ((X/S)^{dR, c})$ for $b-a < p.$
\end{lemma}
\begin{proof}
    We have to check that the natural functor $\cD_{qc, [a, b]}((X/S)^{dR, \gamma, c}) \to \cD_{qc, [a, b]}((X/S)^{dR, c})$ is fully faithful and essentially surjective. Both statements are Zariski local, therefore we can assume that $X$ and $S$ are affine. In this case, by smoothness, we can find an \'etale map $X \to \bA_{S}^n$. Now there is a map of cartesian squares 
  
\[
\begin{tikzcd}
  & {X \times \mathbf{A}^1/\mathbf{G}_m} &&& {(X/S)^{dR, c}} \\
  {\mathbf{A}_S^n \times  \mathbf{A}^1/\mathbf{G}_m} &&& {(X/S)^{dR, \gamma, c}} \\
  & {X \times \mathbf{A}^1/\mathbf{G}_m} &&& {(\mathbf{A}_{S}^n)^{dR, c}} \\
  {\mathbf{A}_S^n \times  \mathbf{A}^1/\mathbf{G}_m} &&& {(\mathbf{A}_{S}^n)^{dR, \gamma, c}}
  \arrow[equals, from=1-2, to=3-2]        
  \arrow[from=1-2, to=1-5]
  \arrow[crossing over, from=2-1, to=2-4] 
  \arrow[from=3-2, to=3-5]
  \arrow[from=4-1, to=4-4]
  \arrow[from=1-2, to=2-1]
  \arrow[from=1-5, to=2-4]
  \arrow[from=3-2, to=4-1]
  \arrow[from=3-5, to=4-4]
  %
  \arrow[crossing over, from=1-5, to=3-5]              
  \arrow[crossing over, from=2-4, to=4-4]              
  %
  \arrow[equals, from=2-1, to=4-1]
\end{tikzcd}
\]
which shows that to prove $\cD_{qc, [a, b]}((X/S)^{dR, \gamma, c}) \to \cD_{qc, [a, b]}((X/S)^{dR, c})$ is an equivalence, it is enough to assume that $X=\bA^n_{S}.$
Moreover, by base change we can assume $X=\bA^n_{\bF_p}.$ Further, it is enough to assume that $X=\bA^1.$ Indeed, note for any $X, Y$ over $\bF_p$ one has the composition $$\cD_{qc, 0, [p-1]}(X^{dR, c} \times_{\bA^1_+/\bG_m} Y^{dR, c}) \mono \cD_{qc}(X^{dR, c} \times Y^{dR, c}) \simeq \cD_{qc}(X^{dR, c}) \otimes_{\cD_{qc}(\bA^1_+/\bG_m)} \cD_{qc}(Y^{dR, c}) $$
which lands in the full subcategory $\cD_{qc, [0, p-1]}(X^{dR, c}) \otimes_{\cD_{qc}(\bA^1_+/\bG_m)} \cD_{qc, [0, p-1]}(Y^{dR, c}).$ This implies a similar statement for $(-)^{dR, \gamma, c}.$ Therefore, we assume $X=\bA^1.$

\textit{Step 2.} 
Choose the standard $\delta$-lift of $\bG_a$. This gives rise to the splitting of the $F_*\bG_a^\sharp\{1\}$-gerbe $\bG_a^{dR, c} \to \bG_a \times \bA^1_+/\bG_m$. This identifies $\bG_a^{dR, c}$ with the classifying stack $B(\bG_a^\sharp\{1\})$ of a group scheme over $\bA^1_+/\bG_m \times \bG_a$. Thus, $\cD_{qc}(\bG_a^{dR, c})$ gets identified with $\cD_{gr, D^p-nilp}(\bF_p[x, v_+, D^p])$, which is a full subcategory of $\cD_{gr}(\bF_p[x, v_+, D^p])$ where $p \cdot \deg(v_+)=-\deg(D^p)=p.$ Namely, it is a full subcategory consisting of objects $\cM$ such that the action of $D^p$ on $H^*(\cM)$ is locally nilpotent. Under this equivalence, $\cD_{qc, [0, p-1]}(\bG_a^{dR, c})$ is identified with a full subcategory of $\cD_{gr, D^p-nilp}(\bF_p[x, v_+, D^p])$ consisting of objects $\cM = \bigoplus_i \cM_i$ such that $\cM_i$ is acyclic for $i \le 0$ and $\Cone(\cM_i \xrightarrow{v_+} \cM_{i+1})$ is acyclic for $i \ge p-1$.  Similarly, $\cD_{qc, [0, p-1]}(\bG_a^{dR, \gamma, c})$ is identified with a full subcategory of graded $\bF_p[x, v_+, \gamma_i(D^p)]/(\gamma_i(D^p)^p=0)$-modules with a similar constraint on weights, where $\deg(v_+)=1, \deg(\gamma_i(D^p))=-p^i$ for $i \ge 1.$ Under this equivalence, the desired functor corresponds to the restriction of scalars along the map $$ \bF_p[x, v_+, D^p] \to \bF_p[x, v_+, \gamma_i(D^p)]/(\gamma_i(D^p)^p=0) $$
sending $D^p$ to $\gamma_1(D^p).$ Let us first show that the pullback along the map $f: B(F_*\bG_a^\sharp\{1\}) \to B\bG_{m, \bA^1}$ of $\bA^1_+/\bG_m \times \bG_a$-stacks induces an equivalence on objects of weights $[0, p-1].$ Let $\cM \in \cD_{qc}(\bG_{m, \bA^1})$ with only one non-zero component $\cM_k$ in degree $k$. It gets sent to $\cM_k[v_+]$ with the trivial action of $D^p.$ To show that $f^*$ is fully faithful in the desired range, it is enough to prove that for any $0 \le k, n \le p-1$ and $\cM_k, \cM_n \in \cD_{qc}(\bA^1),$ the map $$\operatorname{RHom}_{\cD_{qc}(B\bG_m)}(\cM_k, \cM_n) \to \operatorname{RHom}_{\cD_{qc}(B\bG_a^\sharp\{1\})}(\cM_k[v_+], \cM_n[v_+]$$ 
is an equivalence. Note the right-hand side is computed by $$\fib(\operatorname{RHom}_{\cD_{gr}(\bF_p[x, v_+])} (\cM_k[v_+], \cM_n[v_+]) \xrightarrow{D^p} \operatorname{RHom}_{\cD_{gr}(\bF_p[x, v_+])}(\cM_k[v_+], \cM_n[v_+](-p)))$$
and note the right complex vanishes for weight reasons.  Now we will show that the pullback along the map $g: B(\bG_a\{1\}) \to B(\bG_{m})$ of $\bA^1_+/\bG_m \times \bG_a$-stacks induces an equivalence on subcategories of weights $[0 ,p-1].$  Denote  $A = \bF_p[x, v_+, \gamma_i(D^p)]/(\gamma_i(D^p)^p=0)$ and $A_k =\bF_p[x, v_+, \gamma_i(D^p)]/(\gamma_i(D^p)^p=0)$ where $1 \le i \le k.$ Note $$A_{n+1} = A_n \otimes_{\bF_p[v_+, x]} \bF_p[\gamma_{n+1}(D^p)]/(\gamma_{n+1}(D^p)^p=0).$$ A similar computation shows that $\cD_{gr, [0,p-1]}(A_n) \simeq \cD_{qc, [0, p-1]}(B\bG_{m, \bA^1})$ since $\deg(\gamma_i(D^p)) \le -p$. This finishes the proof.

\end{proof}

\begin{corollary}
    Let $X/S$ be a smooth map of schemes over $\bF_p.$ Let $\tilde{X}/W_2(S)$ be a lift of $X.$ Using \Cref{gamma-and-dr-equiv-in-small-range}, we contemplate the following diagram
\[\begin{tikzcd}
	{(X/S)^{dR, c}} && {B_{X' \times \bA^1_+/\bG_m}(T_{X'/S}^{\sharp}\{1\})} && {\MIC^{c}_{[k, l]}(X/S)} && {\HIG^c_{[k, l]}(X'/S)} \\
	\\
	{(X/S)^{dR, \gamma, c}} && {B_{X' \times \bA^1_+/\bG_m}(T_{X'/S}\{1\})} && {\MIC^{\gamma, c}_{[k, l]}(X/S)} && {\HIG^{\gamma, c}_{[k, l]}(X'/S)}
	\arrow[from=1-1, to=3-1]
	\arrow[from=1-3, to=3-3]
	\arrow["\simeq"{description}, no head, from=3-1, to=3-3]
	\arrow["\simeq"{description}, from=3-5, to=1-5]
	\arrow["\simeq"{description, pos=0.1}, from=3-7, to=1-7]
	\arrow["\simeq"{description}, no head, from=3-7, to=3-5]
\end{tikzcd}\]
of $X'$-stacks, where $l-k < p.$  It gives rise to an equivalence $\MIC^c_{[k, l]}(X/S) \simeq \HIG^c_{[k, l]}(X'/S).$ In particular, for any $\cM \in \HIG_{\le p-1}(X'/S)$ one obtains $\Hig(\cM) \simeq \dR(C_f(\cM)) \in \cD_{qc}(X').$
\end{corollary}

\section{Appendix}

\subsection{Deformation theory}
We refer to \cite{MR491680} and recall some facts that will be useful for us.
\begin{remark}\label{def:of:schemes}
    Let $X/S$ be a flat lci family and $S \to \tilde{S}$ a square-zero deformation with an ideal $I \in \cD_{qc}^{\le 0}(S)$. Then the transitivity triangle for $X \xrightarrow{\pi_{X}} S \to \tilde{S}$ gives $\pi_{X}^*L_{S/\tilde{S}} \to L_{X/\tilde{S}} \to L_{X/S}$. Note it stays a fiber sequence after truncating $\tau_{\le 1}$ since $\cH^{-2}L_{X/S} =0.$ Moreover, $\tau_{\le 1}\pi_{X}^* L_{X/S} = \pi_{X}^*I[1].$ Thus, we obtain the boundary map $L_{X/S} \to \pi_{X}^*I[2]$. It is equal to $\ob_{X/S}$ the obstruction of deforming $X$ to $\tilde{S}.$ In other words $\ob_{X/S}$ is the obstruction of splitting of $L_{X/\tilde{S}} \to L_{X/S}.$
\end{remark}

\begin{remark}
    Let $X/S$ be a flat lci family. Since the ideal of $S \to W_2(S)$ is $F_{S, *}\cO_{S}$, we learn from \Cref{def:of:schemes} that the obstruction of deforming $X$ to $W_2(S)$ is given by $L_{X/S} \to \pi_{X}^*F_{S, *}\cO_S[2].$ Note $\pi_{X}^*F_{S, *}\cO_{S}=\nu_{X,*}\cO_{X'}$ by contemplating the cartesian diagram for $X',$ here $\nu_{X}: X' \to X$ is the canonical map. By adjunction and the base change for the cotangent complex, we note that $\ob_{X/W_2(S)}$ gives rise to a map $L_{X'/S} \to \cO_{X'}[2]$ whose vanishing is equivalent to vanishing of $\ob_{X/W_2(S)}.$
\end{remark}

\begin{remark}\label{def:thry}
    Let $S/\bF_p$ be a scheme with square-zero deformation $S \to \tilde{S}$ with an ideal $I \in \cD_{qc}^{\le 0}(S).$ Let $f: X \to Y$ be a map over $S.$ Assume we are given lifts $\tilde{X}, \tilde{Y}$ of $X, Y$ to $\tilde{S}.$ Let us recall the obstruction to lifting $f$ to a map $\tilde{f}: \tilde{X} \to \tilde{Y}.$ Consider the diagram
\[\begin{tikzcd}
	{f^*L_{Y/S}} && {L_{X/S}} \\
	\\
	{f^*L_{Y/\tilde{S}}} && {L_{X/\tilde{S}}} \\
	{f^*L_{Y/S}} && {L_{X/S}}
	\arrow["df"', from=1-1, to=1-3]
	\arrow["{f^*s_{\tilde{Y}}}", from=1-1, to=3-1]
	\arrow[curve={height=30pt}, equals, from=1-1, to=4-1]
	\arrow["{s_{\tilde{X}}}"', from=1-3, to=3-3]
	\arrow[curve={height=-30pt}, equals, from=1-3, to=4-3]
	\arrow["{df_{/\tilde{S}}}", from=3-1, to=3-3]
	\arrow["{f^*h_{Y/\tilde{S}}}", from=3-1, to=4-1]
	\arrow["{h_{X/\tilde{S}}}"', from=3-3, to=4-3]
	\arrow["df", from=4-1, to=4-3]
\end{tikzcd}\]
and the map $$\ob'_{f, \tilde{X}, \tilde{Y}} := s_{\tilde{X}} \circ df - df_{/\tilde{S}} \circ f^*s_{\tilde{Y}}: f^*L_{Y/S} \to L_{X/\tilde{S}}.$$ The homotopies $f^*h_{Y/\tilde{S}} \circ f^*s_{\tilde{Y}} \simeq \id$ and $h_{X/\tilde{S}} \circ s_{\tilde{X}} \simeq \id$ provide for us a null-homotopy of the composition $f^*L_{Y/S} \xrightarrow{\ob'_{f, \tilde{X}, \tilde{Y}}} L_{X/\tilde{S}} \xrightarrow{h_{X/\tilde{S}}} L_{X/S}$, therefore we get a map $$\ob_{f, \tilde{X}, \tilde{Y}}: f^*L_{Y/S} \to \fib(h_{X/\tilde{S}}) \simeq \pi_{X}^*L_{S/\tilde{S}}.$$ Assume $Y/S$ is lci. Since $L_{Y/S}$ is $1$-truncated, we obtain a map $f^*L_{Y/S} \to \tau_{\le 1}\pi_{X}^*L_{S/\tilde{S}} \simeq \pi_{X}^* I[1]$ which is the obstruction of lifting $f$ to a map $\tilde{f}: \tilde{X} \to \tilde{Y}.$ Note that the set of lifts of $Y$ to $\tilde{S}$ is a torsor for $\pi_0\Map_Y(L_{Y/S}, \pi_{Y}^*I[1]).$ In particular, if we replace $\tilde{Y}$ by another lift $\check{Y},$ the obstruction $\ob_{f, \tilde{X}, \tilde{Y}}$ changes by $f^* t_{\tilde{Y}, \check{Y}}: f^*L_{Y/S} \to \pi_{X}^*I[1],$ where $t_{\tilde{Y}, \check{Y}}: L_{Y/S} \to \pi_Y^*I[1]$ measures the difference between $\tilde{Y}$ and $\check{Y}.$
\end{remark}

\begin{remark}\label{rem:any-two-lifts-of-X-same-frobs}
    Let $\tilde{S}$ be a flat scheme over $\bZ/p^2$ and $S$ its special fiber. Let $X/S$ be an lci family and $\tilde{X'}$ a lift of $X'/S$ to $\tilde{S}.$ Taking $Y=X'$ and $f=F_{X/S}$ in \Cref{def:thry}, we learn that for a fixed lift of $X'$, the torsors of lifts of Frobenii to two different lifts $\tilde{X}, \tilde{\tilde{X}}$ are canonically identified. Denote $F_{X/S}=F$ for simplicity, then $\ob'_{F, \tilde{X}, \tilde{X'}} = -dF_{/\tilde{S}} \circ F^*s_{\tilde{X'}}$ which does not depend on a lift of $X$ to $\bZ/p^2.$ Moreover, after post-composing with $F^*h_{X'/\tilde{S}}$ we get $-dF \circ F^*h_{X'/\tilde{S}} \circ F^*s_{\tilde{X'}}$ which is canonically homotopic to $0$ since $dF=0.$
\end{remark}

\begin{remark}
    Let $X/S$ be a flat lci family with a lift $\tilde{X}/W_2(S)$. In particular, we obtain a lift of $X'$ to $W_2(S)$ as $W_2(S)$ has Frobenius. From \Cref{def:thry} we learn that the obstruction of deforming $F_{X/S}$ to a map $\tilde{F}_{X/S}: \tilde{X} \to \tilde{X}^{'}$ is governed by a map $F_{X/S}^*L_{X'/S} \to \nu_* \cO_{X'}[1]$. 
\end{remark}

\subsection{Relative delta-rings.}
We refer to \cite[\S 2.1]{MR4157014} for a discussion of relative $\delta$-structures, and recall the basic definitions here. Let $A$ be a $\delta$-ring.
\begin{definition}
    A $\delta$-structure on an $A$-algebra $B$ is an $A$-linear section of $R:W_2(B) \to B$. Here the $A$-algebra structure on $W_2(B)$ is given by $A \to W_2(A) \to W_2(B)$ where the first map comes from the $\delta$-structure on $A.$
\end{definition}

\begin{lemma}
    Let $B$ be a $\delta-A$-algebra. Then $B \to W_2(B) \xrightarrow{F} B$ defines a lift of Frobenius which is compatible with the Frobenius on $A.$ 
\end{lemma}
\begin{proof}
    It is enough to check that $F:W_2(B) \to B$ is a map over $F_A: A \to A.$ Recall the map $A \to W_2(B)$ is given by $A \to W_2(A) \to W_2(B).$ Thus, it sends $a$ to $(a, \delta_a)$. Since the Frobenius on $A$ is given by $a^p+p \delta_a,$ the claim follows.
\end{proof}

\begin{remark}
    One obtains a category $\CAlg_{A, \delta}$ of $\delta-A$-algebras. The forgetful functor to $\CAlg_{A}$ admits a right adjoint which is given by $B \mapsto W(B).$ Indeed, let $R \to B$ be a map of $A$-algebras with $B \in \CAlg_{A, \delta}.$ Since $B$ is a $\delta$-ring, we obtain $R \to W(B).$ It is easy to see that it is a map of $\delta-A$-algebras.
\end{remark}

\begin{remark}
    Let $f:B \to C$ be a map in $\CAlg_{A, \delta}.$ Then the composition $B \xrightarrow{f} C \xrightarrow{w_C} W(C)$ is equal to $B \xrightarrow{w_B} W(B) \xrightarrow{W(f)} W(C)$
\end{remark}

\begin{lemma}\label{frob-lifts-append}
    Let $R \to S$ be a map of commutative $\bF_p$-algebras. Let $\tilde{S}$ be a flat $W_2(R)$-algebra lifting $R \to S.$ Giving a $W_2(R)$-semilinear morphism $F: \tilde{S} \to \tilde{S}$ lifting $F_S: S \to S$ is equivalent to giving a $W_2(R)$-semilinear map $\tilde{S} \to W_2(S')$.
\end{lemma}

\begin{proof}
    We have $0 \to S' \xrightarrow{i} \tilde{S} \to S \to 0$. Given $F$, note $F(a)=a^p+i(\delta_a)$ for unique $\delta_a \in S'.$ In particular, $F$ is determined by the map $\delta:\tilde{S} \to S'$. The map $f: \tilde{S} \to W_2(S')$ defined by $f(a)=(\nu_S(\bar{a}), \delta_a)$ is a homomorphism. In other direction, if $f: \tilde{S} \to W_2(S')$ is such a map, then define $F: \tilde{S} \to \tilde{S}$ to be $F(a)=a^p+i(\delta_a)$. The main input into showing it is a homomorphism is that the composition $\tilde{S} \to S \xrightarrow{\nu_S} S' \xrightarrow{i} \tilde{S}$ is equal to multiplication by $p.$ Indeed, this composition is obtained from tensoring $W_2(R) \to R \xrightarrow{F} F_*R \xrightarrow{V} W_2(R)$ with $\tilde{S}.$
\end{proof}

\printbibliography
\end{document}